\documentclass[11pt, letterpaper, oneside, reqno]{amsart}

\usepackage{amssymb, amsthm, amsmath, amsfonts}
\usepackage{mathtools, dsfont, mathrsfs, bm, enumitem, url, upgreek}
\usepackage{graphicx}
\usepackage{multicol}
\usepackage{algorithmicx,algorithm,algpseudocode}

\usepackage{bbm}
\usepackage{setspace}
\usepackage{multirow}
\usepackage{tikz}
\usepackage{longtable, booktabs}
\usepackage{dsfont}
\usepackage{fix-cm}

\graphicspath{{figures/}}

\usepackage[utf8]{inputenc}    
\usepackage[T1]{fontenc}       
\usepackage{tabularx}
\usepackage{lmodern}           
\usepackage{color, xcolor}
\usepackage[colorlinks=true, linkcolor=red, citecolor=blue]{hyperref}

\usepackage[margin=1in]{geometry}

\usepackage{fancyhdr}
\fancypagestyle{firstpage}{
     \fancyhf{}
    \fancyfoot[C]{\scriptsize\thepage}  
}

\DeclareMathOperator{\Tr}{tr}

\DeclareMathOperator{\Diag}{Diag}

\DeclareMathOperator*{\argmin}{arg\,min}
\DeclareMathOperator*{\argmax}{arg\,max}

\newcommand{\be}{\begin{equation}}
\newcommand{\ee}{\end{equation}}

\newcommand{\opt}{^\star}
\newcommand{\inner}[2]{\big \langle #1, #2 \big \rangle }

\newcommand{\R}{\mathbb{R}}

\newcommand{\Let}{\triangleq}

\newcommand{\dd}{\mathrm{d}}

\newtheorem{theorem}{Theorem}
\newtheorem{lemma}{Lemma}
\newtheorem{proposition}{Proposition}

\newtheorem{definition}{Definition}
\newtheorem{assumption}{Assumption}
\newtheorem{remark}{Remark}

\begin{document}
\onehalfspacing

\title{Selective Ambulance Dispatch Under Contextual Travel-Time Uncertainty}

\author{
Zikun Lin, 
Daniel Zhuoyu Long,
Viet Anh Nguyen
}
\thanks{\small The authors are with the Department of Systems Engineering and Engineering Management, The Chinese University of Hong Kong (\texttt{zklin@se.cuhk.edu.hk, zylong@se.cuhk.edu.hk, nguyen@se.cuhk.edu.hk})}

\begin{abstract}
Ambulance response is time-critical in out-of-hospital cardiac arrest (OHCA), where dispatchers must balance timely arrivals with limited fleet capacity. Static territories and deterministic travel-time estimates are vulnerable to dynamic congestion, while always-dual dispatch adds redundancy but consumes fleet capacity. We propose IDEAL (Intelligent Dual dispatch of Emergency AmbuLances), a selective dual-dispatch framework that sends a second ambulance only when the optimistic gap between primary and secondary paths exceeds a threshold. IDEAL learns context-specific edge travel times from trip-level dispatch records, including unobserved routes, using a weakly supervised bilevel representation network. We train the nonsmooth model with mini-batch conservative gradients and prove an asymptotic convergence guarantee. IDEAL models uncertainty via Burg-divergence perturbations to a shared metric in the learned representation space, thereby inducing correlated changes in edge travel times and learning context-specific radii from historical underprediction errors. For real-time decisions, IDEAL casts optimistic-gap computation as a difference-of-convex program and derives an efficient oracle with complexity guarantees. In collaboration with the Hong Kong Fire Services Department, we evaluate IDEAL using historical OHCA records and real-time adaptive simulations. The results achieve a stronger response-time/resource trade-off relative to all region-based and Google-based baselines.
\end{abstract}

\maketitle
\thispagestyle{firstpage}

\section{Introduction}\label{sec:introduction}

Ambulance dispatchers race the clock when a caller reports a medical emergency outside a hospital. In out-of-hospital cardiac arrest (OHCA), minutes or even seconds can make the difference between survival and death: based on a survival model in~\cite{ref:waalewijn2001survival}, estimated survival is about 49\% with immediate resuscitation, compared with about 22\% after a four-minute delay. In addition, cities also face rising emergency demand as populations age and cardiovascular events become increasingly common among older adults~\cite{ref:north2012intersection}. Dispatchers, therefore, must make judicious choices while ambulance resources remain limited.

We define response time as the time from the emergency call to reaching the patient. It can be decomposed into three parts: activation time, ambulance travel time, and vertical access time. Activation time covers call processing and crew turnout. Travel time is the time required to drive from the depot to the demand. Vertical access time is the time required to reach the patient after the ambulance arrives at the address, typically within a building. Dispatch centers can often closely control activation time. In Hong Kong, 96.5\% of emergency calls met a two-minute activation target during 1--14 February 2009~\cite{ref:pac2009fsd}. In contrast, travel time is more difficult to control and predict, as many emergency medical services (EMS) systems report~\cite{ref:ufa2020soc, ref:indiantown2019ems}. We therefore treat travel time as the main source of uncertainty and the part most affected by dispatch decisions.

Dispatch centers often follow one of two rules. Region-based dispatch assigns each demand location to a fixed territory. These regions simplify operations but ignore traffic disruptions such as incidents, signal cycles, and short-lived congestion. Such rigidity can delay arrivals when traffic shifts across neighborhoods~\cite{ref:jagtenberg2017optimal}. Prediction-based dispatch instead chooses the unit with the shortest predicted travel time. The Hong Kong system uses this idea in practice~\cite{ref:audit2008fsd}. However, point predictions can be misleading when traffic conditions change during the trip and nearby roads slow down at the same time. Two ambulances may have nearly identical predicted travel times, yet the selected one may arrive later. Dispatchers, therefore, face a shortest-path problem in which edge travel times depend on a dynamic traffic state rather than only on the static road map.

Some EMS systems create response redundancy for high-priority calls by sending a second ambulance or another trained responder in parallel with the primary response; dual-dispatch and multi-responder OHCA policies have been studied in several systems~\cite{ref:hollenberg2009dual, ref:kim2019application, ref:strnad2022impact}.
This redundancy consumes capacity because the second ambulance cannot respond to the next demand. Dispatchers, therefore, face two linked decisions: whether to send a second ambulance and, if so, which path it should follow.
Traffic correlations complicate this choice because congestion and incidents can slow many nearby links at the same time~\cite{ref:zhang2017lagrangian, ref:srinivasan2014reliable}. A second ambulance that shadows the primary path can be delayed for the same reason, even if it departs from a different depot. Rules that simply choose the second-best predicted option may consume an additional ambulance while offering little improvement.

These considerations motivate a selective-redundancy rule, but two practical questions remain: when do static or prediction-only rules fail, and what information would justify sending a second ambulance? The next subsection answers these questions through examples from the study region.

\subsection{Motivating Examples}\label{sec:examples}
Dispatchers in our study area, the northwest region of Hong Kong Island, often choose between two nearby ambulance depots, Morrison Hill Ambulance Depot (MOR) and Mount Davis Ambulance Depot (MOU). The examples below show three practical problems: static regions fail near the boundary, point forecasts change during the trip, and optimistic--pessimistic traffic-model estimates overlap so much that they do not support a clear second-ambulance decision. We generate these examples with a travel-time simulator that mimics turn-by-turn navigation. The simulator queries the Google Routes API at dispatch, traverses one road segment using the current prediction for that segment, and then re-queries from the new location; the detailed procedure is provided in Appendix~\ref{ec:adaptive-simulation}. We focus on 50 demand nodes in the northwest region of Hong Kong Island, where historical incidents suggest that both depots can be competitive.

\textbf{Issue 1: Static Regions Fail Near the Boundary.} A region-based policy assigns each demand node to a fixed depot. This rule simplifies operations but ignores real-time traffic. We replay the 50 nodes under a range of traffic conditions in January 2025. For each replay, we call a depot choice \textit{optimal} if its realized travel time is no greater than that of the other depot. This region-based policy selects the optimal depot in only about 56\% of runs. Static regions, therefore, perform poorly in the overlap zone, where both depots often compete to serve the same demand nodes.

\textbf{Issue 2: A Point Forecast Can Reverse During the Trip.} On January 30, 2025, at 18{:}29, we simulated an incident in Mid-Levels, a residential area on Hong Kong Island. Google's initial query predicted travel times of 712 seconds from MOR to the incident and 802 seconds from MOU. A prediction-based rule would dispatch from MOR. However, our adaptive simulator showed realized travel times of 932 seconds from MOR and 792 seconds from MOU. 
Figure~\ref{fig:illustration} plots the depots and the demand and traces the total predicted travel time, defined as the elapsed time plus the current predicted remainder, as the simulated ambulances progressed segment by segment. The forecast for the ambulance from MOR climbed by more than two minutes after departure, while the forecast for the one from MOU stayed nearly flat.

\begin{figure}[ht]
    \centering
    \includegraphics[width=\linewidth]{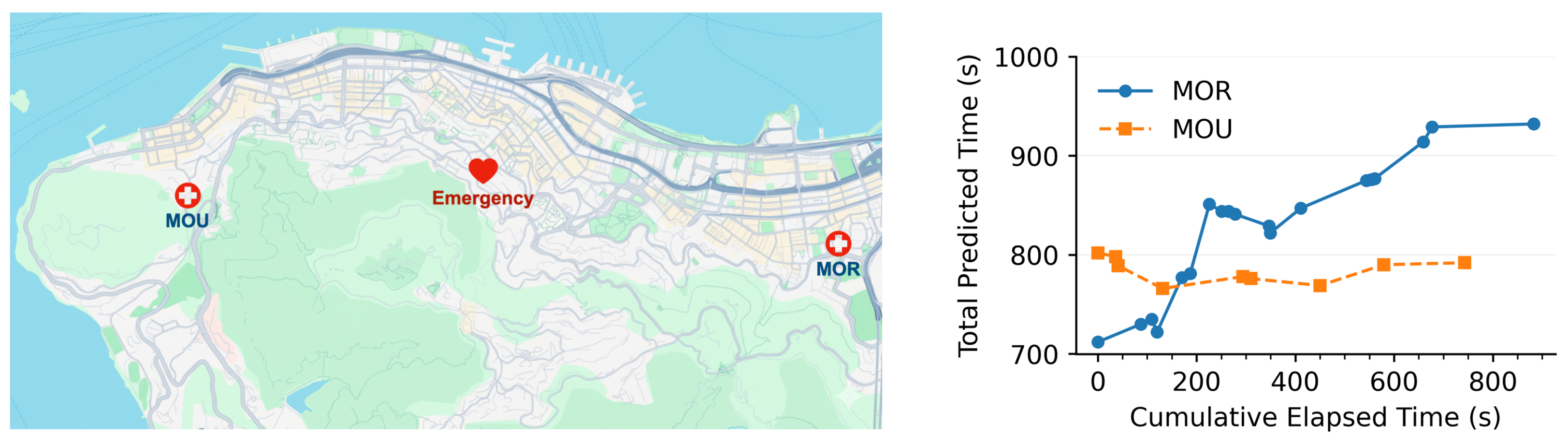}
    \caption{(Left) Two depots respond to one demand. The heart marks the incident; crosses mark MOU (west) and MOR (east). (Right) Segment-by-segment evolution of Google Routes' total predicted travel time, defined as elapsed time plus the current predicted remainder. Each marker is recorded when a segment completes, and the API is re-queried. The horizontal axis is cumulative elapsed time; the vertical axis is total predicted travel time.}
    \label{fig:illustration}
\end{figure}

\textbf{Issue 3: Optimistic--Pessimistic Traffic-Model Estimates Overlap and Hide Dominance.}
The Google Routes API can also provide optimistic and pessimistic traffic-model travel-time estimates at dispatch.\footnote{Google Routes API traffic-model documentation: \url{https://developers.google.com/maps/documentation/routes/traffic-model}, accessed March~1,~2025.} Among the selected demand nodes, the two depots' optimistic--pessimistic estimate ranges overlap for about 99.5\% of simulations, with a mean overlap of about 377~seconds. As the mean realized travel time is about 596~seconds, the overlap often spans a large fraction of the trip. In the Mid-Levels incident in the previous example, the API reported optimistic--pessimistic estimates of 668--1379~seconds for MOR and 646--1233~seconds for MOU, so the wide estimate ranges overlap almost completely and offer no clear choice.

These issues motivate a selective redundancy rule that, given the primary path, asks whether another path could plausibly arrive earlier by a meaningful margin.

\subsection{IDEAL Framework for Scenario-Aware Selective Dual Dispatch}\label{sec:framework}
Motivated by these examples, we develop IDEAL (Intelligent Dual dispatch of Emergency AmbuLances), a scenario-aware selective dual-dispatch framework that serves as a decision-support overlay on the primary dispatch workflow. From historical dispatch records, IDEAL learns a model that maps the call-time context to a nominal edge-travel-time vector and an uncertainty radius. These two outputs define the scenario set used by IDEAL. Given a primary path from an internal dispatch rule or a routing engine, IDEAL evaluates whether another path offers a plausible arrival-time advantage sufficient to justify the capacity cost of activating a second ambulance.

Figure~\ref{fig:framework} illustrates IDEAL as a three-step flow. In the \textit{scenario set construction} step, the bilevel representation learning network maps static edge information and call-time context to learned representation vectors and a nominal edge-travel-time vector, while the radius prediction network maps the learned contextual representation vector to an uncertainty radius; together, these outputs define the scenario set. In the \textit{optimistic gap computation} step, IDEAL evaluates the primary path over this scenario set and identifies an optimistic scenario, the corresponding optimistic time gap, and a candidate secondary path. In the \textit{decision} step, IDEAL compares the gap to the threshold and dispatches the second ambulance only when the gap exceeds it.

\begin{figure}[ht]
    \centering
    \includegraphics[width=\linewidth]{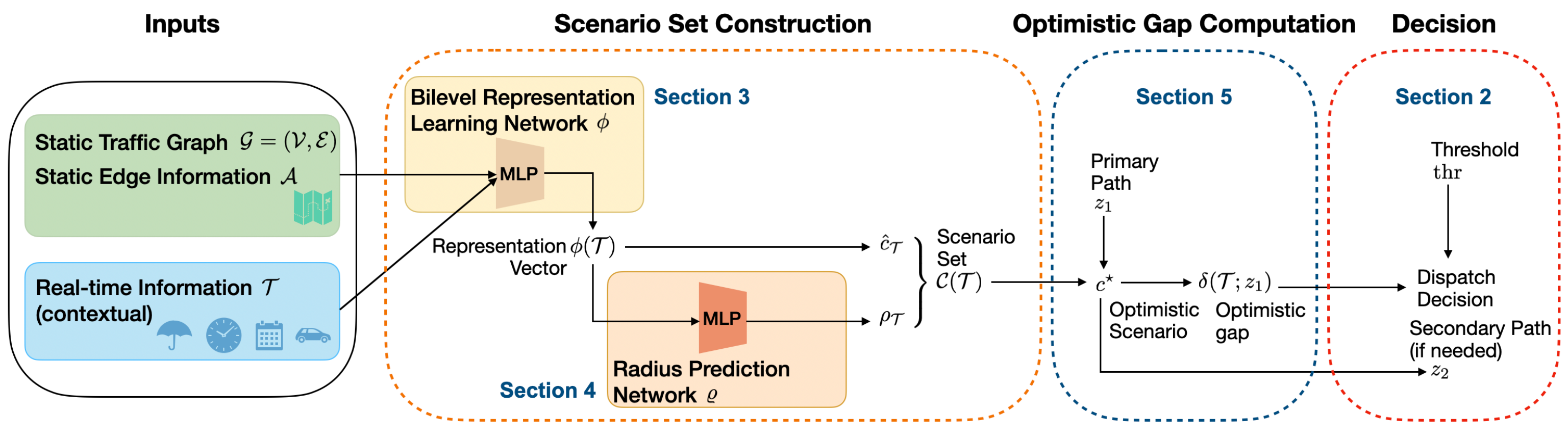}
    \caption{
    IDEAL decision flow. 
    \textbf{Inputs:} a traffic graph $\mathcal G=(\mathcal V,\mathcal E)$ with edge attributes $\mathcal A$, and real-time contextual variables $\mathcal T$. 
    \textbf{Step~1: Scenario set construction.} The bilevel representation learning network $\phi=(\vartheta,\theta,\psi)$ fuses $\mathcal A$ and $\mathcal T$ to produce the representation vector $\phi(\mathcal T)$, and the nominal edge-travel-time vector $\hat c_{\mathcal T}$; the radius prediction network $\varrho$ maps the learned representation vector to the radius $\rho_{\mathcal T}$. These two outputs define the scenario set $\mathcal C(\mathcal T)$.
    \textbf{Step~2: Optimistic gap computation.} Given a primary path $z_1$, IDEAL finds the optimistic scenario $c\opt$ over $\mathcal C(\mathcal T)$ and computes the optimistic time gap $\delta(\mathcal T;z_1)$ and a candidate secondary path $z_2$.
    \textbf{Step~3: Decision.} IDEAL compares $\delta(\mathcal T;z_1)$ with the threshold and dispatches the second ambulance along $z_2$ only when the gap exceeds the threshold.
    }
    \label{fig:framework}
\end{figure}

\noindent\textbf{Contributions and technical innovations.} This paper makes three advances.

\begin{enumerate}[leftmargin=7mm]
\item \textit{Weakly supervised bilevel travel-time representation learning.} 
We develop a weakly supervised representation network that fuses static road attributes and call-time contextual information to produce representation vectors and predict travel times for each edge. The dataset used for representation learning does not reveal the detailed routes that ambulances actually take. We therefore use a bilevel formulation, in which the lower-level problem selects the shortest path based on the travel-time estimates for each edge, and the upper-level problem updates the network parameters. The overarching objective is to minimize the matching loss between the predicted and observed travel times, plus a graph-regularization term that enforces spatial coherence in static edge embeddings. We train the nonsmooth bilevel model using a mini-batch conservative gradient method and prove an asymptotic convergence guarantee for the generated iterates.

\item \textit{Contextual correlated scenario set from metric perturbations.} We construct a contextual scenario set by perturbing a shared metric in the learned representation space, with Burg divergence measuring the perturbation magnitude. Under this construction, edges with similar contextual representation vectors respond similarly to the same metric perturbation. From historical underprediction errors, we train a separate radius prediction network whose input is the learned contextual representation vector and whose output is the context-specific radius.

\item \textit{Efficient oracle and convergence analysis for the optimistic gap.} We show that the optimistic gap over the Burg scenario set can be written as a difference-of-convex program over the metric matrix. The key computational step is an efficient oracle for the DCA (difference-of-convex algorithm) matrix subproblem: the optimal matrix admits a quasi-closed-form expression that depends on a scalar multiplier, which can be obtained via a one-dimensional root search over an interval. This oracle replaces a high-dimensional matrix optimization at dispatch time with an eigenvalue computation and a scalar search. We establish the resulting per-iteration complexity and convergence guarantee for the dispatch-time DCA procedure.
\end{enumerate}

We evaluate IDEAL through controlled replay experiments that construct real-time OHCA dispatch scenarios on real road networks and compare our policy against region-based and Google-based baselines. The results show that, at similar dispatch volumes, IDEAL reduces ambulance travel time compared with the baselines. Beyond average performance, IDEAL also improves tail performance by reducing the risk of severe response delays in high-uncertainty scenarios. Additional transfer experiments in another urban district further demonstrate that our policy remains robust to changes in road network structure.

\noindent\textbf{Governmental Collaboration.} We conducted this study with the Hong Kong Fire Services Department. The department provided de-identified OHCA records from January~2014 to December~2023, and a senior ambulance officer advised on dispatch procedures and constraints. This partnership ensured that the policy remained interpretable and auditable within existing workflows.

\noindent\textbf{Broader Impact.} The same framework applies to other time-critical services that travel under correlated, context-varying conditions, including fire and rescue services~\cite{ref:wan2022initial, ref:hassler2024socio}, police patrols~\cite{ref:deangelo2023police, ref:blanes2018effect}, roadside assistance~\cite{ref:salum2020impact}, utility restoration crews~\cite{ref:arif2018optimizing, ref:zhou2024resilience, ref:balut2019ranking, ref:liu2020survivability}, and rapid repair teams for critical infrastructure~\cite{ref:lin2024risk}.

\subsection{Related Work} \label{sec:related-work}
Ambulance location and dispatching have long been core topics in operations research; see the survey of location and relocation models in~\cite{ref:brotcorne2003ambulance}. Classical approaches include the hypercube queueing model and covering formulations with expected- and backup-coverage extensions~\cite{ref:larson1974hypercube, ref:toregas1971location, ref:church1974maximal, ref:daskin1983maximum, ref:hogan1986backup}. Modern studies incorporate uncertainty and redundancy via robust dual-coverage and location--routing designs~\cite{ref:wajid2022robust, ref:boutilier2020ambulance}. Because dispatch decisions depend on travel-time distributions, Bayesian and large-network methods estimate predictive travel-time uncertainty from ambulance GPS data~\cite{ref:westgate2013travel, ref:westgate2016large}. Integrated siting--dispatch and dynamic dispatch--relocation policies, including ADP-based redeployment, yield consistent gains over static or closest-idle rules~\cite{ref:yoon2021stochastic, ref:gendreau2001dynamic, ref:nasrollahzadeh2018real, ref:maxwell2010approximate}.

Representation learning provides a paradigm for learning compact embeddings that capture latent structure in complex data~\cite{ref:bengio2013representation, ref:lecun2015deep}. For graph-structured transportation networks, spatio-temporal GNNs combine topology with temporal signals to forecast traffic states and estimate travel time/ETA~\cite{ref:li2017diffusion, ref:yu2018stgcn, ref:wu2019graphwavenet, ref:derrowpinion2021eta}. Recent surveys and self-supervised road-network embeddings synthesize and extend these ideas toward task-agnostic representations for large ITS systems~\cite{ref:rahmani2023graph, ref:wang2020rlrn, ref:chang2023spatial}.

Bilevel optimization (or multilevel programming) models hierarchical leader--follower decisions~\cite{ref:bracken1973mathematical, ref:bard1998practical}. Foundational results establish optimality conditions and single-level reformulations, often via KKT/complementarity, and surveys synthesize canonical problem classes and algorithms~\cite{ref:dempe2002foundations, ref:fortuny1981representation, ref:colson2005bilevel}. Recent work covers mixed-integer methods for discrete bilevel models~\cite{ref:kleinert2021survey} and first-order approaches with convergence guarantees for modern bilevel learning and related nonsmooth nonconvex settings~\cite{ref:ji2021convergence, ref:lu2025tsp, ref:cao2025complexity, ref:cutkosky2023optimal}.

Conservative set-valued fields provide a nonsmooth calculus that aligns backpropagation with subgradient-based stationarity and, together with tame/definable geometry, yield convergence tools for learning objectives~\cite{ref:bolte2021conservative, ref:ioffe2009invitation}. Under these assumptions, stochastic (sub)gradient methods with diminishing stepsizes have stationary limit points, whereas constant-stepsize SGD concentrates near the critical set~\cite{ref:davis2020stochastic, ref:bianchi2022convergence}. Recent analyses extend similar guarantees to adaptive schemes in the Adam family~\cite{ref:kingma2015adam, ref:reddi2018adam, ref:xiao2024adam}. The same framework also supports nonsmooth implicit differentiation in implicit layers and hyperparameter models~\cite{ref:bolte2021nonsmooth}.

Robust optimization (RO) seeks decisions that remain feasible for all realizations in a prescribed uncertainty set~\cite{ref:ben2009robust, ref:bertsimas2011theory}. Budgeted polyhedral sets mitigate conservatism by tuning a budget parameter $\Gamma$~\cite{ref:bertsimas2004price}. Data-driven RO/DRO learns uncertainty or ambiguity sets from historical data, leading to tractable moment- and Wasserstein-based formulations~\cite{ref:delage2010distributionally, ref:mohajerin2018data, ref:bertsimas2018data}. Contextual variants condition protection on covariates, and learning-based approaches exploit richer set structure~\cite{ref:kannan2024residuals, ref:bandi2022cro, ref:goerigk2023dnn}. For a recent overview of DRO, see~\cite{ref:kuhn2025dro}.

Difference-of-convex (DC) programming provides a standard framework for expressing a nonconvex objective as the difference of two convex functions~\cite{ref:tao1997convex, ref:horst1999dc, ref:lethi2018thirty}. 
The DC algorithm (DCA), closely related to the convex--concave procedure, iteratively linearizes the concave component and solves the resulting convex surrogate subproblem~\cite{ref:tao1998dc, ref:lipp2016ccp}. 
Recent work establishes nonasymptotic convergence-rate guarantees, proposes boosted line-search variants, and develops methods targeting stronger stationarity notions for nonsmooth DC programs~\cite{ref:abbaszadehpeivasti2024rate, ref:aragon2018accelerating, ref:pang2017bstationary}.

\section{An Optimism-Based Framework for Dual Dispatch} \label{sec:ideal-framework}
A dual dispatch policy needs to make two decisions upon receiving an emergency call. First, the dispatcher decides the primary path for the first ambulance from a depot to the incident location. Second, conditional on this primary dispatch, the dispatcher evaluates whether one marginal backup ambulance is worth activating and, if so, plans its path. We propose IDEAL (Intelligent Dual dispatch of Emergency AmbuLances), a framework that assesses the largest plausible reduction in response time from dispatching the second ambulance and compares this gain against a clinically meaningful time threshold. The verbal form of IDEAL is summarized below.

\noindent\fbox{\begin{minipage}{\linewidth}
\underline{\textbf{IDEAL policy, verbal form.}}

\textbf{Inputs:} a primary path for the first ambulance, current contextual information, trained prediction models, and a time threshold.

\begin{enumerate}[leftmargin=7mm]
\item \textbf{Scenario set construction.} Use the trained models to construct a scenario set of plausible edge-travel-time vectors based on the current contextual information.

\item \textbf{Optimistic gap computation.} Over this scenario set, find the secondary path with the largest plausible arrival-time advantage over the primary path and compute the corresponding optimistic time gap.

\item \textbf{Decision.} Compare the optimistic time gap with the time threshold. If the gap exceeds the threshold, dispatch a second ambulance on the corresponding path. Otherwise, dispatch only the primary ambulance.
\end{enumerate}
\end{minipage}
}

This section will unpack the above verbal form into a quantitative model: Section~\ref{sec:optimistic-time-gap} describes the optimistic time gap, and Section~\ref{sec:ideal-policy} provides the algorithmic form of IDEAL.

\subsection{Optimistic Time Gap and Optimistic Secondary Path} \label{sec:optimistic-time-gap}

To simplify the exposition, we use the specific example described in Figure~\ref{fig:illustration}: We model the road network as a directed graph $\mathcal G = (\mathcal V,\mathcal E)$, and two depots $\mathrm{MOR}$ and $\mathrm{MOU}$ are nodes in $\mathcal G$. We collect the call-time information in a context vector $\mathcal T$, which captures factors such as weather, date, and time of day that affect traffic conditions at call time. IDEAL uses two trained modules at decision time: a representation-learning network that produces the nominal edge-travel-time vector $\hat c_{\mathcal T}$, and a radius prediction network that produces the uncertainty radius $\rho_{\mathcal T}$. Together, these outputs define the scenario set $\mathcal C(\mathcal T)$ around $\hat c_{\mathcal T}$.

Let $\mathcal Z$ denote the set of all feasible paths from either depot (MOR or MOU) to the incident demand node $D \in \mathcal V$:
\[
\mathcal Z \Let \left\{
    z \in \{0, 1\}^{|\mathcal E|}: 
    \begin{array}{l}
    \forall v \in \mathcal V: \\[0.5ex]
    \displaystyle
    \sum_{e \in \mathcal E: o(e) = v} z[e] - \sum_{e \in \mathcal E: d(e) = v} z[e]  \begin{cases}
        \le 1 & \text{if } v\in\{\mathrm{MOR},\mathrm{MOU}\}, \\
        = -1 & \text{if } v = D, \\
        = 0 & \text{otherwise.}
    \end{cases}
    \end{array}
\right\},
\]
where $o(e)$ and $d(e)$ denote the origin and destination nodes for each directed edge $e \in \mathcal E$. 
Any $z \in \mathcal Z$ describes a route from a depot to $D$. Since all edge travel times used below are nonnegative, redundant directed cycles in a path can be removed without increasing cost, so every relevant minimizer admits a cycle-free path representative. We therefore represent the primary path by its binary edge-incidence vector $z_1 \in \mathcal Z$.
For systems with more than two candidate depots, IDEAL can be implemented through pairwise comparison over pairs of depots. Each comparison uses the same two-depot formulation, and the policy selects the secondary path with the largest positive surplus over the threshold.

The \textit{optimistic time gap} quantifies the largest plausible arrival-time advantage of a secondary path over the primary path:
\begin{equation}
\label{eq:delta-rho}
\delta(\mathcal T; z_1) \Let
\max_{c\in\mathcal C(\mathcal T)}
\Bigl( c^\top z_1  -  \underbrace{\min_{z \in\mathcal Z} c^\top z}_{\text{fastest arrival under scenario } c} \Bigr).
\end{equation}
For an edge-travel-time scenario $c$, the inner minimum returns the shortest travel time from either depot to $D$. The difference $c^\top z_1 - \min_{z\in\mathcal Z} c^\top z$ is the time advantage of the fastest path over the primary path in that scenario. The outer maximization selects the scenario in $\mathcal C(\mathcal T)$ that maximizes this advantage.

Next, we determine the path for the second ambulance. When $\mathcal C(\mathcal T)$ is nonempty and compact, the maximization problem in~\eqref{eq:delta-rho} admits at least one maximizer, and we denote the maximizing scenario by $c\opt$ and define the associated optimistic secondary path $z_2$ as the shortest path under $c\opt$:
\[
c\opt \in \argmax_{c\in\mathcal C(\mathcal T)}
\bigl(c^\top z_1 - \min_{z\in\mathcal Z} c^\top z\bigr), \qquad z_2 \in \argmin_{z\in\mathcal Z} (c\opt)^\top z.
\]

\subsection{IDEAL Policy: Algorithmic Form and Optimality} \label{sec:ideal-policy}
IDEAL uses a nonnegative threshold $\mathrm{thr}(\mathcal T)$, measured in seconds, as a capacity shadow price. Operationally, this threshold is the response-time reduction required to justify making one additional ambulance temporarily unavailable for other calls. The underlying operational burden captures factors such as coverage loss, fleet availability, and dispatch volume. To connect this operational threshold to clinical risk, the dispatch center may approximate the increase in clinical risk per additional second of travel based on the nominal travel time of the primary ambulance. Dividing the operational burden of a second ambulance by this risk sensitivity yields the time-saving threshold at which the operational cost and clinical benefit balance. Thus, time-sensitive contexts lead to a lower threshold, whereas tighter capacity conditions lead to a higher threshold. Appendix~\ref{ec:calibration} formalizes this calibration through a first-order approximation and an operational cost parameter.

IDEAL dispatches a second ambulance when the optimistic time gap exceeds the threshold:
\begin{equation}
\label{eq:decision-rule}
\text{dispatch a second ambulance along } z_2
\quad\Longleftrightarrow\quad
\delta(\mathcal T; z_1) > \mathrm{thr}(\mathcal T).
\end{equation} 
We now state the algorithmic form of the IDEAL policy with these inputs.

\noindent\fbox{\begin{minipage}{\linewidth}
\underline{\textbf{IDEAL policy, algorithmic form}}

\textbf{Inputs:} primary path $z_1 \in \mathcal Z$; current contextual information $\mathcal T$; trained models $\phi$ and $\varrho$; time threshold $\mathrm{thr}(\mathcal T)$.

\begin{enumerate}[leftmargin=7mm]
\item \textbf{Scenario set construction.}
Use $\phi$ and $\varrho$ with the contextual information $\mathcal T$ to obtain the nominal edge-travel-time vector $\hat c_{\mathcal T}$ and the radius $\rho_{\mathcal T}$. Form the scenario set $\mathcal C(\mathcal T)$.

\item \textbf{Optimistic gap computation.} Compute
\[
c\opt \in \argmax_{c \in \mathcal C(\mathcal T)} \Bigl( c^\top z_1 - \min_{z\in\mathcal Z} c^\top z \Bigr),
\qquad
\delta(\mathcal T; z_1) \Let (c\opt)^\top z_1 - \min_{z\in\mathcal Z} (c\opt)^\top z,
\]
and choose $z_2 \in \argmin_{z\in\mathcal Z} (c\opt)^\top z$.

\item \textbf{Decision.}
If $\delta(\mathcal T; z_1) > \mathrm{thr}(\mathcal T)$, dispatch a second ambulance along $z_2$; otherwise, dispatch only the primary ambulance.
\end{enumerate}

\textbf{Outputs:} either no second ambulance is dispatched, or the policy selects a secondary path $z_2$.
\end{minipage}}

For a fixed context $\mathcal T$, a primary path $z_1 \in \mathcal Z$, and any secondary path $z \in \mathcal Z$, define the path-specific optimistic time gap $\Delta(z;\mathcal T,z_1)$ by
\[
\Delta(z;\mathcal T,z_1) \Let \max_{c\in\mathcal C(\mathcal T)} \bigl(c^\top z_1 - c^\top z\bigr).
\]
Theorem~\ref{thm:ideal_optimality} shows that the algorithm solves a surplus-maximization problem and achieves the optimistic time gap $\delta(\mathcal T; z_1)$.

\begin{theorem}[Optimality of the IDEAL policy]
\label{thm:ideal_optimality}
Fix a context $\mathcal T$ and a primary path $z_1 \in \mathcal Z$. Consider the optimization problem
\begin{equation}
\label{eq:pthr}
\tag{$\mathcal P_{\text{thr}}$}
\max_{\tau\in\{0,1\},\,z\in\mathcal Z} \Bigl\{ \tau\bigl( \Delta(z;\mathcal T,z_1) - \mathrm{thr}(\mathcal T) \bigr) \Bigr\}.
\end{equation}
When $\mathcal C(\mathcal T)$ is nonempty and compact, the IDEAL policy produces a pair $(\tau,z_2)$ that solves~\eqref{eq:pthr}, and the following two statements hold:
\begin{enumerate}[leftmargin=7mm]
\item The scenario $c\opt$ and secondary path $z_2$ constructed in the algorithm achieve the optimistic time gap $\Delta(z_2;\mathcal T,z_1) = \delta(\mathcal T; z_1)$.
\item The policy dispatches a second ambulance if and only if every optimal solution of~\eqref{eq:pthr} has $\tau=1$, which is equivalent to condition~\eqref{eq:decision-rule}.
\end{enumerate}
\end{theorem}

The next three sections separate the method into a training stage and an inference stage. Sections~\ref{sec:bilevel-representation-learning} and~\ref{sec:metric-perturbation} describe the training stage: learn context-dependent edge representation vectors that produce the nominal edge-travel-time vector $\hat c_{\mathcal T}$, and construct the metric-perturbation scenario set $\mathcal C(\mathcal T)$ by training the radius predictor $\varrho$. Section~\ref{sec:solution-procedure} describes the inference stage: given the call-time context $\mathcal T$ and a primary path $z_1$, compute an estimate of the optimistic time gap and a candidate secondary path efficiently, and dispatch the second ambulance only when this estimate exceeds the threshold.

\section{Travel Time Predictor by Bilevel Representation Learning}
\label{sec:bilevel-representation-learning}
This section describes the first part of the training stage. Using historical dispatch records, we train the representation network $\phi$ so that, for any call-time context $\mathcal T$, it outputs the nominal edge-travel-time vector $\hat c_{\mathcal T}$ used in Step~1 of IDEAL. We first introduce the data and architecture, then formulate the bilevel problem, and finally describe the training algorithm with convergence guarantees.

\subsection{Dataset and Architecture} \label{sec:data}

The dataset records out-of-hospital cardiac arrest (OHCA) cases managed by the Hong Kong Fire Services Department (HKFSD). We keep the road-network notation $\mathcal G=(\mathcal V,\mathcal E)$ from Section~\ref{sec:ideal-framework}. Each directed edge $e\in\mathcal E$ has a static feature vector $\mathcal A[e]\in\R^{d_{\mathcal A}}$ that includes geometric and roadway attributes such as length, speed limit, and number of lanes. The full data-construction pipeline, including data sources, the study region, geocoding, network construction, covariate encoding, travel-time target extraction, and preprocessing rules, is documented in Appendix~\ref{ec:data-details}.

Let $N$ denote the number of samples in the dataset. Each sample $i\in[N]$ contains:
\begin{enumerate}[leftmargin=7mm]
\item A context vector $\mathcal T^{(i)} \in \R^{d_{\mathcal T}}$ that encodes calendar indicators and weather at call time.
\item A destination node $v^{(i)}_{\mathrm{d}} \in \mathcal V$ indicating the location of the demand, and an origin node $v^{(i)}_{\mathrm{o}} \in \mathcal V$ indicating the depot from which the ambulance was dispatched. These nodes are obtained by projecting real-world addresses to the nearest nodes in the traffic graph $\mathcal G$.
\item An observed road travel time $t^{(i)} \in \R_+$ obtained from operational timestamps after adjusting the call-receipt-to-CPR interval for preparation time and vertical-access times.
\end{enumerate}

The tuple $\bigl(\mathcal G, \mathcal A, \{\mathcal T^{(i)}\}_{i=1}^N, \{v^{(i)}_{\mathrm{o}}\}_{i=1}^N, \{v^{(i)}_{\mathrm{d}}\}_{i=1}^N, \{t^{(i)}\}_{i=1}^N\bigr)$ 
serves as the input to our learning module. Figure~\ref{fig:framework-representation} depicts our representation network. The architecture consists of three components:
\begin{enumerate}[leftmargin=7mm]
    \item A static network $\vartheta$ that maps the static edge feature vector to a static edge embedding $\vartheta[e] \in \R^d$ for each $e \in \mathcal E$.
    \item A dynamic network $\theta$ that maps each context to a contextual representation vector
    $\theta(\mathcal T) \in \R^d$.
    \item A cross network $\psi$ that fuses the static edge embeddings and the contextual representation vector into context-dependent edge representation vectors
    \begin{equation}
    \phi(\mathcal T)[e] \Let \psi\bigl(\vartheta[e], \theta(\mathcal T)\bigr) \in \R^d \quad \forall e \in \mathcal E.
    \label{eq:phi-def}
    \end{equation}
\end{enumerate}
All three representations $\vartheta[e]$, $\theta(\mathcal T)$, and $\phi(\mathcal T)[e]$ share the common dimension $d$.
\begin{figure}[ht]
    \centering
    \includegraphics[width=\linewidth]{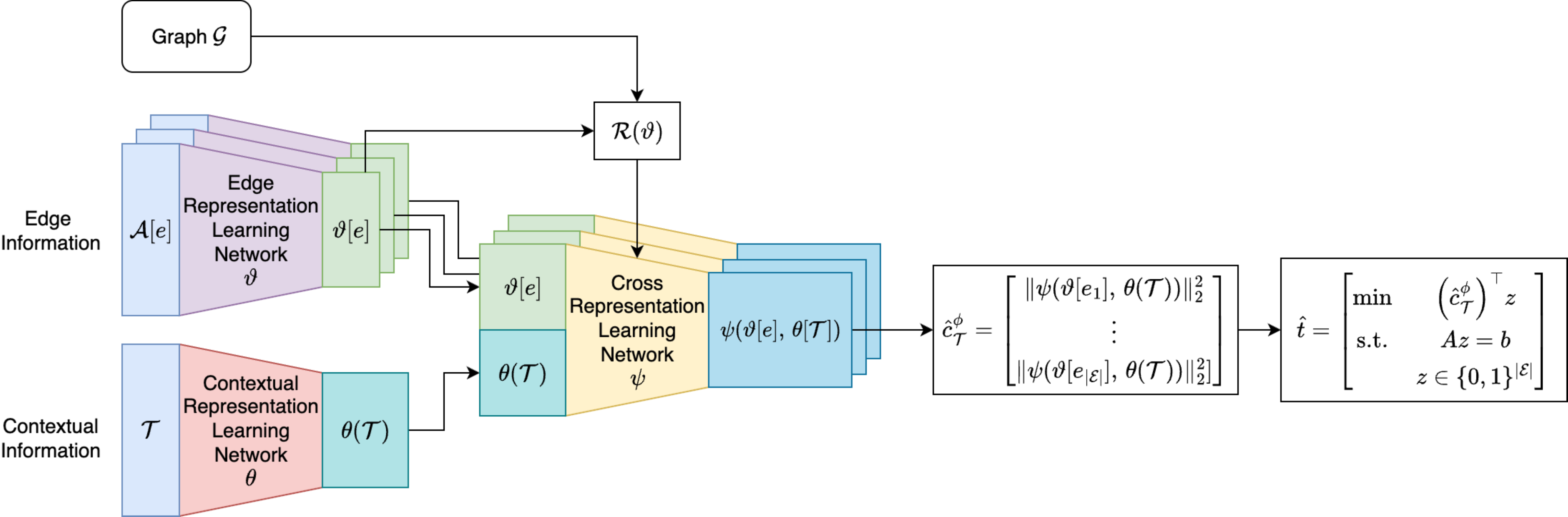}
    \caption{Representation-learning architecture for edge travel time prediction. The static \textit{edge representation learning network} embeds edge information. The dynamic \textit{contextual representation learning network} embeds contextual information. The \textit{cross representation learning network} fuses both embeddings to produce context-dependent edge representation vectors and edge travel time estimates. A shortest-path problem aggregates edge travel times to match the observed travel times.}
    \label{fig:framework-representation}
\end{figure}

We convert the representation $\phi(\mathcal T)[e]$ into an edge travel time estimate:
\begin{equation}
\hat c_{\mathcal T}^\phi[e] \Let \bigl\|\phi(\mathcal T)[e]\bigr\|_2^2 \quad \forall e \in \mathcal E,
\label{eq:c-map}
\end{equation}
where the squared Euclidean norm guarantees nonnegative edge travel times.
Collecting all edges yields the predicted cost vector
\begin{equation} \label{eq:c-def}
\R_+^{|\mathcal E|} \ni \hat c_{\mathcal T}^\phi 
= \begin{pmatrix}
    \hat c_{\mathcal T}^\phi[e_1] \\
    \vdots \\
    \hat c_{\mathcal T}^\phi[e_{|\mathcal E|}] 
\end{pmatrix} 
\stackrel{\text{by}~\eqref{eq:c-map}}{=} 
\begin{pmatrix}
    \| \phi(\mathcal T)[e_1] \|_2^2\\
    \vdots \\
    \| \phi(\mathcal T)[e_{|\mathcal E|}] \|_2^2
\end{pmatrix}
\stackrel{\text{by}~\eqref{eq:phi-def}}{=}
\begin{pmatrix}
    \| \psi(\vartheta[e_1],  \theta (\mathcal T) ) \|_2^2 \\
    \vdots \\
    \| \psi(\vartheta[e_{|\mathcal E|}],  \theta (\mathcal T) ) \|_2^2
\end{pmatrix}.
\end{equation}
We interpret $\hat c_{\mathcal T}^\phi$ as the model's nominal travel-time estimate under context $\mathcal T$.

\subsection{Bilevel Learning with Unobserved Paths}
\label{sec:bilevel}

We learn the representation network weights $\phi$ using a weakly supervised learning technique~\cite{ref:zhou2018brief}. Consider the $i$-th sample from the training data. We pass its context $\mathcal T^{(i)}$ through the representation network $\phi$ to obtain the point estimate $\hat c_{\mathcal T^{(i)}}^\phi \in \R_+^{|\mathcal E|}$ for the edge travel times. The shortest-path travel time from the originating depot $v^{(i)}_{\mathrm{o}}$ to the incident destination $v^{(i)}_{\mathrm{d}}$ is obtained by solving the shortest-path problem:
\begin{equation}
g^{(i)}(\phi) \Let \min_{\substack{z \in \{0,1\}^{|\mathcal E|},\, A z = b^{(i)}}}
\bigl(\hat c_{\mathcal T^{(i)}}^\phi\bigr)^\top z,
\label{eq:gi-def}
\end{equation}
where  $A \in \R^{|\mathcal V| \times |\mathcal E|}$ is the static node-arc incidence matrix
\[
\forall (v,e)\in\mathcal V\times\mathcal E:\quad 
A_{v,e} =
\begin{cases}
  1 & \text{if edge } e \text{ leaves node } v,\\
 -1 & \text{if edge } e \text{ enters node } v,\\
  0 & \text{otherwise,}
\end{cases}
\]
and $b^{(i)} \in \R^{|\mathcal V|}$ encodes the origin and destination of sample $i$:
\[
\forall v \in \mathcal V:\quad
b_v^{(i)} =
\begin{cases}
  1 & \text{if } v = v^{(i)}_{\mathrm{o}},\\
 -1 & \text{if } v = v^{(i)}_{\mathrm{d}},\\
  0 & \text{otherwise.}
\end{cases}
\]
We fit $\phi$ by minimizing the discrepancy between the observed travel time $t^{(i)}$ and the point estimate of the shortest-path travel time $g^{(i)}(\phi)$. We penalize the discrepancy with a continuously differentiable loss $\ell$; examples include the squared-error and Huber losses.
Additionally, we introduce a regularization term $\mathcal R(\vartheta)$ that penalizes the differences between the length-normalized edge embeddings of adjacent edges:
\begin{equation}
\mathcal R(\vartheta) \Let \sum_{(e_1,e_2)\in \mathcal E_{\mathrm{adj}}} \left\| \frac{\vartheta[e_1]}{\sqrt{l[e_1]}} - \frac{\vartheta[e_2]}{\sqrt{l[e_2]}} \right\|_2^2,
\label{eq:regularization}
\end{equation}
where $\mathcal E_{\mathrm{adj}} \Let \bigl\{(e,e')\in\mathcal E\times\mathcal E: e\neq e', \; e \text{ and } e' \text{ share an endpoint}\bigr\}$, and $l[e]>0$ denotes the length of edge $e$.
This regularization enforces spatial smoothness in the context-invariant part of the representation, reflecting local geometric similarity. The normalization by $\sqrt{l[e]}$ scales the penalty per unit length.
Because $\mathcal R$ regularizes only the edge embeddings, it depends only on $\vartheta$.

With this regularizer, the empirical learning problem is formally the bilevel program
\begin{equation}
\label{eq:bilevel-learning-problem}
\begin{aligned}
\min_{\phi} \quad
& \frac{1}{N}\sum_{i\in[N]} \ell\bigl((\hat c_{\mathcal T^{(i)}}^\phi)^\top z^{(i)}, t^{(i)}\bigr) + \beta \mathcal R(\vartheta) \\
\text{s.t.}\quad
& z^{(i)} \in \argmin_{\substack{z \in \{0,1\}^{|\mathcal E|},\, A z = b^{(i)}}} (\hat c_{\mathcal T^{(i)}}^\phi)^\top z, \quad i\in[N].
\end{aligned}
\end{equation}
Here $\phi$ is the upper-level learning variable, while each $z^{(i)}$ is a lower-level optimal solution based on the current value of $\phi$. Using the lower-level optimal value $g^{(i)}(\phi)$, the upper-level objective function of~\eqref{eq:bilevel-learning-problem} is equivalently written as
\begin{equation}
J(\phi) \Let \frac{1}{N} \sum_{i \in [N]} \ell\bigl(g^{(i)}(\phi), t^{(i)}\bigr) + \beta \mathcal R(\vartheta).
\label{eq:compact}
\end{equation}
The empirical objective~\eqref{eq:compact} can be interpreted through the lens of inverse optimization. The forward problem is a context-dependent shortest-path model: given the call-time context, it induces edge travel times and returns the shortest-path travel time from the dispatched depot to the incident. The inverse problem then adjusts the representation parameters so that these forward predictions match historical travel-time observations. This connects our formulation to data-driven inverse optimization~\cite{ref:mohajerin2018inverse, ref:chan2025inverse} and to routing-specific inverse-optimization models~\cite{ref:zattoni2025inverse}. Unlike classical inverse shortest-path models that observe the chosen path~\cite{ref:burton1992instance}, our data only reveal the trip-level travel time, similar to inverse shortest-path-length and objective-value settings~\cite{ref:cui2010complexity}.

\subsection{First-Order Training via Mini-Batch Conservative Gradients} \label{sec:conservative-training}

The shortest-path lower-level problem in~\eqref{eq:gi-def} makes $J(\phi)$ nonsmooth: when two or more paths tie in the lower-level problem~\eqref{eq:gi-def}, $g^{(i)}(\phi)$ is not differentiable, so ordinary backpropagation does not produce a unique gradient. 
We use the conservative-field framework of~\cite{ref:bolte2021conservative} to handle this nonsmoothness; the formal definitions used in the convergence analysis are detailed in Appendix~\ref{ec:training-convergence-proof}. For each $i \in [N]$, define the per-sample objective
\begin{equation}
J^{(i)}(\phi)
\Let
\ell\bigl(g^{(i)}(\phi), t^{(i)}\bigr) + \beta \mathcal R(\vartheta),
\label{eq:Ji-def}
\end{equation}
so that $J$ is the average of the per-sample objectives: $J(\phi) = \frac{1}{N} \sum_{i \in [N]} J^{(i)}(\phi)$, which recovers exactly the objective in~\eqref{eq:compact}.
When each $J^{(i)}$ admits a conservative field $D^{(i)}$, we define the aggregate field
\[
D_J(\phi) \Let \frac{1}{N}\operatorname{conv}\bigl(\sum_{i\in[N]} D^{(i)}(\phi)\bigr),
\]
where $\sum$ denotes the Minkowski sum of sets. By the calculus rules for conservative fields~\cite[Section~6.1]{ref:bolte2021conservative}, $D_J$ is a conservative field for $J$.

We now introduce Algorithm~\ref{alg:conservative-sgd}, a mini-batch conservative gradient descent scheme that solves~\eqref{eq:compact} using conservative-field elements.

\begin{algorithm}
\caption{Mini-batch conservative gradient descent for \eqref{eq:compact}}
\begin{algorithmic}[1]
\Require Initial parameter $\phi_0\in\R^p$, deterministic step sizes $\{\alpha_k\}_{k\in\mathbb N}$, mini-batches $\{B_k\}_{k\in\mathbb N}$, perturbation radii $\{r_k\}_{k\in\mathbb N}$ with $r_k\ge 0$ for all $k$, and a fixed total order $\prec$ on feasible paths.
\For{$k = 0, 1, 2, \dots$}
    \State Draw $u_k\sim \mathrm{Unif}(\{u\in\R^p:\|u\|_2\le 1\})$ and set $\tilde\phi_k \leftarrow \phi_k + r_k u_k$.
    \For{each $i \in B_k$}
        \State Compute $\hat c_{\mathcal T^{(i)}}^{\tilde\phi_k}$ as in \eqref{eq:c-map}.
        \State Compute
        $z_k^{(i)} \in \argmin\{ (\hat c_{\mathcal T^{(i)}}^{\tilde\phi_k})^\top z : z\in\{0,1\}^{|\mathcal E|},\, A z=b^{(i)}\}$ and break ties by $\prec$.
        \State Set $d_k^{(i)} \leftarrow \nabla_1 \ell\bigl((\hat c_{\mathcal T^{(i)}}^{\tilde\phi_k})^\top z_k^{(i)}, t^{(i)}\bigr) \bigl(\nabla_\phi \hat c_{\mathcal T^{(i)}}^{\tilde\phi_k}\bigr)^\top z_k^{(i)}
        + \beta \nabla_\phi \mathcal R(\tilde\vartheta_k)$.
    \EndFor
    \State Set $d_k \leftarrow |B_k|^{-1} \sum_{i \in B_k} d_k^{(i)}$ and $\phi_{k+1} \leftarrow \phi_k - \alpha_k d_k$.
\EndFor
\Ensure Sequence $\{\phi_k\}_{k\in\mathbb N}$.
\end{algorithmic}
\label{alg:conservative-sgd}
\end{algorithm}

Algorithm~\ref{alg:conservative-sgd} follows stochastic gradient descent. Each iteration samples a mini-batch, solves a shortest-path problem for each sample, backpropagates through the selected path, averages the resulting directions, and updates $\phi$. The order $\prec$ makes the shortest-path oracle single-valued when ties occur. In line 6 of Algorithm~\ref{alg:conservative-sgd}, $\nabla_1\ell$ is the derivative with respect to the first argument, $\tilde\vartheta_k$ is the $\vartheta$-component of $\tilde\phi_k$, and $\nabla_\phi\mathcal R(\tilde\vartheta_k)$ denotes $(\nabla_\vartheta\mathcal R(\tilde\vartheta_k),0,0)$ since $\mathcal R$ depends only on $\vartheta$. Appendix~\ref{ec:training-convergence-proof} constructs conservative fields for $J^{(i)}$ and shows that the directions computed above are valid field elements. 
The nonnegative radius schedule $\{r_k\}$ enables Algorithm~\ref{alg:conservative-sgd} to cover two regimes, the unperturbed recursion ($r_k\equiv0$) and the perturbed recursion ($r_k > 0$), with different guarantees. The former yields the asymptotic accumulation-set guarantee in Theorem~\ref{thm:convergence}, while the latter supports the finite-time $(\delta,\varepsilon)$-stationarity guarantee in Appendix~\ref{ec:finite-time-proof}.

\subsection{Convergence Analysis} \label{sec:training-convergence}

We first state the assumptions and the asymptotic accumulation-set guarantee for the unperturbed recursion with $r_k=0$ for all $k$ in Algorithm~\ref{alg:conservative-sgd}.

\begin{assumption}[Smooth loss and bounded activation Jacobians]
\label{assmpt:smoothness-bounded}
For any compact set $U\subset\R^p$, there exists a constant $L>0$ (depending on $U$) such that the following conditions hold:
\begin{enumerate}[label=(\alph*),leftmargin=7mm]
\item \textbf{Loss smoothness.} For each fixed $t \in \R$, the function $x \mapsto \ell(x, t)$ is continuously differentiable.
\item \textbf{Activation smoothness.} All activation functions used in $\vartheta$, $\theta$, and $\psi$ are continuously differentiable in a neighborhood of $U$.
\item \textbf{Bounded activation Jacobians.} For all $\phi \in U$, $i \in [N]$, and $e \in \mathcal E$, $\|\nabla_\phi(\phi(\mathcal T^{(i)})[e])\|_2 \le L$.
\end{enumerate}
\end{assumption}

\begin{assumption}[Definable architecture]
\label{assmpt:definable}
All primitive building blocks are definable in a fixed o-minimal structure, for example $\mathbb R_{\mathrm{an},\exp}$: the subnetworks $\vartheta$, $\theta$, and $\psi$ (and hence $\phi$), the loss $x \mapsto \ell(x, t)$ for each $t$, and the regularizer $\mathcal R$.
\end{assumption}

Assumptions~\ref{assmpt:smoothness-bounded}--\ref{assmpt:definable} hold for common modeling choices. Squared-error and Huber losses satisfy Assumption~\ref{assmpt:smoothness-bounded}(a). Smooth activations such as sigmoid, $\tanh$, SoftPlus~\cite{ref:dugas2000incorporating}, and Swish~\cite{ref:ramachandran2017searching} satisfy Assumption~\ref{assmpt:smoothness-bounded}(b) and Assumption~\ref{assmpt:definable}. In our implementation, weight decay and early stopping are used as practical controls on the parameter magnitude, which supports applying Assumption~\ref{assmpt:smoothness-bounded}(c) locally on the bounded-iterates event used in Theorem~\ref{thm:convergence}.

We now present the asymptotic accumulation-set result for Algorithm~\ref{alg:conservative-sgd} when $r_k=0$ for all $k$. Let $\mathbb P$ denote the probability measure induced by all random draws in Algorithm~\ref{alg:conservative-sgd}. Define the bounded-iterates event $\mathcal B \Let \{\sup_{k\in\mathbb N}\|\phi_k\|_2<+\infty \}$. We assume $\mathbb{P}(\mathcal B)>0$ and interpret the convergence result below as holding almost surely on $\mathcal B$.

\begin{theorem}[Asymptotic accumulation-set guarantee]
\label{thm:convergence}
Under Assumptions~\ref{assmpt:smoothness-bounded}--\ref{assmpt:definable}, consider Algorithm~\ref{alg:conservative-sgd} with perturbation radii $r_k=0$ for all $k$ and any fixed deterministic tie-breaking order $\prec$. Assume that the mini-batches $\{B_k\}_{k\in\mathbb N}$ are drawn independently and uniformly at random and that the step sizes $\{\alpha_k\}_{k\in\mathbb N}\subset(0, +\infty)$ are deterministic and satisfy $\sum_{k=0}^{+\infty}\alpha_k = +\infty$ and $\lim_{k\to+\infty}\alpha_k\log(k+2)=0$.
Then, almost surely on $\mathcal B$, the sequence $\{\phi_k\}_{k\in\mathbb N}$ generated by Algorithm~\ref{alg:conservative-sgd} has a nonempty accumulation set; every accumulation point $\bar\phi$ satisfies $0\in D_J(\bar\phi)$; and the objective $J$ is constant on the accumulation set.
\end{theorem}

In \cite{ref:bolte2021conservative}, the nonsmoothness of the problem arises from the activation functions of the neural network (e.g., ReLU). In contrast, our problem is nonsmooth even with smooth activation functions because multiple shortest paths can tie in~\eqref{eq:gi-def}. Fortunately, the conservative-field framework still applies: each shortest-path selection yields a valid conservative-field element for the per-sample objective, and Algorithm~\ref{alg:conservative-sgd} averages such elements over the mini-batch.

\begin{remark}[Smoothing the lower-level problem]
\label{rem:smoothing}
To smoothen the value function $\hat g^{(i)}$ for differentiability, we can add a quadratic regularization to the lower-level problem:
\[
\hat g^{(i)}(\phi)
\Let
\min_{\substack{z \in [0,1]^{|\mathcal E|},\, A z = b^{(i)}}}
\bigl(\hat c_{\mathcal T^{(i)}}^\phi\bigr)^\top z
+ \gamma \|z\|_2^2,
\quad \gamma > 0.
\]
This modification makes the lower-level objective strongly convex in $z$, so the lower-level minimizer is unique, and differentiability is satisfied. The main advantage of this alternative is that one can use classical Karush-Kuhn-Tucker-based implicit differentiation and smooth bilevel machinery to solve the resulting problem. However, the lower-level problem requires solving a quadratic program at each iteration. General-purpose quadratic-programming solvers typically scale superlinearly in the number of variables and constraints, and in our setting, this cost is substantially higher than a single run of Dijkstra's algorithm, which runs in $O(|\mathcal E| + |\mathcal V|\log|\mathcal V|)$ time.

Consequently, the unregularized shortest-path formulation~\eqref{eq:bilevel-learning-problem} is preferable for our purpose: it avoids approximation error and permits each training iteration to use shortest-path selection followed by backpropagation, which is computationally cheaper than solving a quadratic program.
\end{remark}

We also develop a finite-time $(\delta,\varepsilon)$-stationarity guarantee for Algorithm~\ref{alg:conservative-sgd} with a deterministic nonincreasing positive perturbation-radius sequence $\{r_k\}_{k\in\mathbb N}$.
The finite-time proof allows any deterministic positive step-size sequence; when the same algorithm is used for asymptotic convergence, one may impose the step-size regime in Theorem~\ref{thm:convergence}.
The stationarity scale is fixed as $\delta \Let 2r_0$, where $r_0$ is the initial perturbation radius.
The finite-time statement and proof are given in Appendix~\ref{ec:finite-time-proof}.

\section{Travel Time Scenario Set from Metric Perturbations} \label{sec:metric-perturbation}
This section connects the learned travel-time predictor from Section~\ref{sec:bilevel-representation-learning} to uncertainty quantification. We first define the family of scenario sets used at inference time and describe how to train the radius prediction network $\varrho$ to predict the uncertainty radius for a new demand.

\subsection{Metric Perturbations and the Scenario Set}\label{sec:metric-perturbations-scenario-set}
The predictor in Section~\ref{sec:bilevel-representation-learning} maps each context $\mathcal T$ to edge embeddings $\phi(\mathcal T)[e] \in \R^d$ for $e \in \mathcal E$. To simplify notation, we omit the network parameter $\phi$ and write
\begin{equation} \label{eq:nominal-metric}
\phi_{\mathcal T}[e] \Let \phi(\mathcal T)[e],
\qquad
\hat c_{\mathcal T}[e] \Let \|\phi_{\mathcal T}[e]\|_2^2 = \phi_{\mathcal T}[e]^\top I\,\phi_{\mathcal T}[e],
\qquad e\in\mathcal E,
\end{equation}
where $I$ denotes the $d\times d$ identity matrix. Equation~\eqref{eq:nominal-metric} interprets $\hat c_{\mathcal T}[e]$ as a quadratic form of the embedding $\phi_{\mathcal T}[e]$ under the identity matrix, which matches standard formulations in metric learning~\cite{ref:bellet2013survey}. We use a positive-definite matrix $X \in \mathbb S_{++}^d$ to define the perturbed edge-travel-time vector
\begin{equation}
c_{\mathcal T}(X)[e] \Let \phi_{\mathcal T}[e]^\top X \phi_{\mathcal T}[e]
\qquad \forall e \in \mathcal E.
\label{eq:cX}
\end{equation}
The nominal travel times $\hat c_{\mathcal T}[e]$ in~\eqref{eq:nominal-metric} correspond to $X = I$. 
The uncertainty is therefore parameterized by a symmetric matrix $X$ in the $d(d+1)/2$-dimensional space of symmetric matrices.

We measure metric perturbations by the Burg matrix divergence from the nominal metric $I$:
\[
D_{\mathrm{Burg}}(X,I)
= \Tr(X) - \log\det(X) - d,
\qquad X \in \mathbb S_{++}^d.
\]
The divergence is nonnegative and equals zero only at $X = I$~\cite{ref:kulis2006learning}. It penalizes both large and nearly singular matrices.
Given a radius $\rho_{\mathcal T} \ge 0$, we define the metric-perturbation set
\[
\mathcal X_{\rho_{\mathcal T}}
\Let
\left\{
X \in \mathbb S_{++}^d:
D_{\mathrm{Burg}}(X,I) \le \rho_{\mathcal T}
\right\},
\]
containing all metrics whose Burg divergence from the nominal Euclidean metric is at most $\rho_{\mathcal T}$. Applying the map $X \mapsto c_{\mathcal T}(X)$ to each element in the metric-perturbation set generates the scenario set for the travel times on each edge.

\begin{definition}[Burg scenario set] \label{def:C-rho}
For each context $\mathcal T$, the Burg scenario set is the edge-travel-time set induced by perturbing the matrix within $\mathcal X_{\rho_{\mathcal T}}$ around the identity matrix $I$:
\begin{equation}
\label{eq:scenario-set}
\mathcal C(\mathcal T)
\Let
\bigl\{
c_{\mathcal T}(X) \in \R_+^{|\mathcal E|} : \exists X \in \mathcal X_{\rho_{\mathcal T}} \text{ such that } c_{\mathcal T}(X)[e] = \phi_{\mathcal T}[e]^\top X \phi_{\mathcal T}[e]
~ \forall e \in \mathcal E
\bigr\}.
\end{equation}
\end{definition}

The radius $\rho_{\mathcal T}$ controls the size of $\mathcal X_{\rho_{\mathcal T}}$, and consequently the size of $\mathcal C(\mathcal T)$. When $\rho_{\mathcal T}=0$, this set collapses to the singleton, point prediction $\{\hat c_{\mathcal T}\}$. 
Operationally, a larger predicted radius indicates a context in which the nominal edge-travel-time prediction is less reliable and admits wider metric perturbations; a smaller predicted radius keeps the scenario set close to the nominal prediction. Because the same matrix $X$ perturbs all edge embeddings jointly via the relationship~\eqref{eq:cX}, scenario variation is coupled across the network: edges with similar contextual representation vectors respond similarly to the same perturbation rather than varying independently edge by edge.
The set $\mathcal X_{\rho_{\mathcal T}}$ is convex because $X\mapsto \Tr(X)$ is linear and $X\mapsto -\log\det(X)$ is convex on $\mathbb S_{++}^d$. The map $X\mapsto c_{\mathcal T}(X)$ is linear, so $\mathcal C(\mathcal T)$ is the linear image of a convex set and is therefore convex.
Proposition~\ref{prop:attainment} summarizes the resulting well-posedness properties of $\mathcal C(\mathcal T)$.

\begin{proposition}[Convexity and compactness of the scenario set]
\label{prop:attainment}
Fix a context $\mathcal T$ and a radius $\rho_{\mathcal T} \ge 0$. The Burg scenario set $\mathcal C(\mathcal T)$ is nonempty, convex, and compact.
\end{proposition}

Thus, for any fixed context and predicted radius, the scenario set used by IDEAL is mathematically well-posed: it is nonempty, preserves the convexity needed for the subsequent optimization steps, and is bounded and closed in the induced travel-time space.

\subsection{Radius Prediction Network for Adaptive Optimism}
\label{sec:rho}

We use a radius prediction network $\varrho$ to map the learned representation vector $\theta(\mathcal T)$ to a nonnegative, context-specific radius $\rho_{\mathcal T}$. Once the representation network $\phi=(\vartheta,\theta,\psi)$ has been trained, $\theta(\mathcal T)$ summarizes call-time contextual factors that are informative for travel-time prediction. Therefore, it is a natural input for predicting the size of the scenario set under the current environment. 

\subsubsection{Offline Target Radius Calculation}

We train the radius prediction network in a supervised manner. We assign a target radius $\rho^{(i)}$ to each training sample $i \in [N]$. The target is the smallest metric perturbation measured by Burg divergence that can make the predicted shortest-path time at least the observed travel-time target $t^{(i)}$. For any $X \in \mathbb S_{++}^d$, define $c_{\mathcal T^{(i)}}(X)$ as in~\eqref{eq:cX} and define the induced shortest-path value
\[
h^{(i)}(X)
\Let \min_{\substack{z \in \{0,1\}^{|\mathcal E|},\, A z = b^{(i)}}} c_{\mathcal T^{(i)}}(X)^\top z.
\]
The nominal prediction is $\hat t^{(i)} \Let h^{(i)}(I)$. We then define the target radius $\rho^{(i)}$ by
\begin{equation}
\rho^{(i)}
\Let \min \left\{ D_{\mathrm{Burg}}(X,I) : X \in \mathbb S_{++}^d,\; h^{(i)}(X) \ge t^{(i)} \right\}.
\label{eq:uncertainty-radius}
\end{equation}
If $\hat t^{(i)} \ge t^{(i)}$, the identity matrix $I$ is feasible and yields $\rho^{(i)} = 0$. If $\hat t^{(i)} < t^{(i)}$, then any feasible solution must satisfy $\rho^{(i)} > 0$, encoding the late-arrival risk that drives our selective second-ambulance decision. Linear-programming duality can be applied to reformulate the constraint $h^{(i)}(X) \ge t^{(i)}$, and we obtain an equivalent convex program with explicit constraints.

\begin{proposition}[Dual reformulation of the target radius]
\label{prop:rho-refor}
Let $\mathbf 1\in\R^{|\mathcal E|}$ denote the all-ones vector. For each training sample $i \in [N]$, the target radius in~\eqref{eq:uncertainty-radius} satisfies
\begin{equation}
\rho^{(i)}
= \left\{
\begin{array}{cl}
\displaystyle \min & D_{\mathrm{Burg}}(X,I) \\
\mathrm{s.t.} & X \in \mathbb S_{++}^d,\; \pi \in \R^{|\mathcal V|},\; \omega \in \R_+^{|\mathcal E|} \\
& c_{\mathcal T^{(i)}}(X)[e] = \phi_{\mathcal T^{(i)}}[e]^\top X \phi_{\mathcal T^{(i)}}[e]
\quad \forall e \in \mathcal E, \\
& A^\top \pi - \omega \le c_{\mathcal T^{(i)}}(X), \\
& (b^{(i)})^\top \pi - \mathbf 1^\top \omega \ge t^{(i)}.
\end{array}
\right.
\label{eq:rho-refor}
\end{equation}

\end{proposition}

Problem~\eqref{eq:rho-refor} is convex in $(X,\pi,\omega)$, since the objective is convex on $\mathbb S_{++}^d$ and the constraints are affine. Standard interior-point solvers handle this problem reliably at our scale.

\subsubsection{Radius Prediction Network}
We fit $\varrho$ on the training pairs $\{(\theta(\mathcal T^{(i)}), \rho^{(i)})\}_{i=1}^N$ by solving a standard estimation problem
\[
\min_{\varrho}
\sum_{i=1}^N
\ell_\rho\Bigl(\varrho\bigl(\theta(\mathcal T^{(i)})\bigr),\, \rho^{(i)}\Bigr),
\]
where $\ell_\rho$ is a nonnegative loss on $\R_+ \times \R_+$. We use the mean absolute error loss in practice.

In the inference phase, the trained representation network $\phi=(\vartheta,\theta,\psi)$ and the trained radius prediction network $\varrho$ are held fixed. Given the call-time context $\mathcal T$ for a new demand, we first evaluate the representation network to obtain the nominal edge-travel-time vector $\hat c_{\mathcal T}$ and the contextual representation vector $\theta(\mathcal T)$. We then evaluate the radius prediction network to obtain $\rho_{\mathcal T}\Let \varrho(\theta(\mathcal T))$, and form the scenario set $\mathcal C(\mathcal T)$ via~\eqref{eq:scenario-set}.

\section{Optimistic Time Gap Computation by Difference-of-Convex Programming} \label{sec:solution-procedure}
We now focus on the last component of the IDEAL framework: the computation of the optimistic time gap. At inference time when we deploy IDEAL on an incoming emergency call, the networks $\phi$ and $\varrho$ are fixed and Step~1 of IDEAL evaluates these trained networks on the current context $\mathcal T$ to obtain the point estimate $\hat c_{\mathcal T}$, the scenario radius $\rho_{\mathcal T}$, and hence the scenario set $\mathcal C(\mathcal T)$. This section addresses Step~2: given $\mathcal C(\mathcal T)$ and a primary path $z_1$, we formulate the optimistic time gap $\delta(\mathcal T; z_1)$ and compute an estimate $\widehat{\delta}(\mathcal T; z_1)$ together with an associated candidate secondary path. We show that the problem has a difference-of-convex (DC) structure and that each subproblem can be solved efficiently. Exploiting Definition~\ref{def:C-rho} of the Burg scenario set, equation~\eqref{eq:delta-rho} becomes
\[
\delta(\mathcal T; z_1)
=
\max_{X \in \mathcal X_{\rho_{\mathcal T}}}
\Bigl(
c_{\mathcal T}(X)^\top z_1 - \min_{z \in \mathcal Z} c_{\mathcal T}(X)^\top z
\Bigr).
\]
The above problem is a DC program in $X$: For any path $z\in\mathcal Z$, define the path matrix $G(z)$ by
\begin{equation}
\label{eq:outer_product}
G(z) \Let \sum_{e\in\mathcal E} \phi_{\mathcal T}[e]\bigl(\phi_{\mathcal T}[e]\bigr)^\top z[e] \in \mathbb S^d,
\end{equation}
so that $c_{\mathcal T}(X)^\top z=\inner{G(z)}{X}$ under the Frobenius inner product $\inner{A}{B}\Let \Tr(A^\top B)$. Define the DC components $f_1$, $f_2$ and the resulting objective $f$ by
\[
f_1(X)\Let -\inner{G(z_1)}{X},
\qquad
f_2(X)\Let -\min_{z\in\mathcal Z}\inner{G(z)}{X},
\qquad
f(X)\Let f_1(X)-f_2(X).
\]
The map $X\mapsto\min_{z\in\mathcal Z}\inner{G(z)}{X}$ is concave as a pointwise minimum of affine functions, so $f_2$ is convex. This convexity yields the following DC representation:
\begin{equation}
\label{eq:DC}
\delta(\mathcal T; z_1)
= - \min_{X \in \mathcal X_{\rho_{\mathcal T}}} f(X)
= - \min_{X \in \mathcal X_{\rho_{\mathcal T}}} \bigl( f_1(X) - f_2(X) \bigr).
\end{equation}

Equation~\eqref{eq:DC} shows that the exact optimistic gap $\delta(\mathcal T; z_1)$ is obtained by minimizing the DC function $f$ over $\mathcal X_{\rho_{\mathcal T}}$. Since this DC problem is generally nonconvex, we apply the Difference-of-Convex Algorithm (DCA) as a local descent procedure: at every iterate, DCA linearizes $f_2$ at the incumbent solution and solves a convex surrogate over $\mathcal X_{\rho_{\mathcal T}}$. The resulting feasible matrix produces a computable estimate $\widehat{\delta}(\mathcal T;z_1)$ of the optimistic gap. The next subsection describes $\partial f_2(X)$ and gives a closed-form solution to the DCA subproblem.

\subsection{Difference-of-Convex Algorithm}
\label{subsec:DCA}

DCA requires subgradients of the convex component $f_2$, and we use the matrices $G(z)$ to represent these subgradients.

\begin{proposition}[Subdifferential of $f_2$]
\label{prop:subgradient-sA}
The function $f_2$ is convex on $\mathbb S_{++}^d$ and its subdifferential at $X \in \mathbb S_{++}^d$ satisfies $\partial f_2(X) = \operatorname{conv}\left\{ -G(z) : z \in \mathcal Z\opt(X) \right\}$, where $\mathcal Z\opt(X)$ is the set of shortest paths under $X$,
\[
\mathcal Z\opt(X)
\Let
\argmin_{z \in \mathcal Z}
\sum_{e \in \mathcal E}
\Bigl(\phi_{\mathcal T}[e]^\top X \phi_{\mathcal T}[e]\Bigr) z[e]
=
\argmin_{z \in \mathcal Z} c_{\mathcal T}(X)^\top z.
\]
\end{proposition}

Proposition~\ref{prop:subgradient-sA} yields a shortest-path subgradient: for $z^{(k)}\in\mathcal Z\opt(X^{(k)})$, set $\zeta^{(k)} = -G\bigl(z^{(k)}\bigr)\in \partial f_2\bigl(X^{(k)}\bigr)$. 
When shortest paths tie, any choice of $z^{(k)}$ is valid. DCA then solves the subproblem
\[
X^{(k+1)} \in \argmin_{X\in\mathcal X_{\rho_{\mathcal T}}}\bigl(f_1(X)-\inner{\zeta^{(k)}}{X}\bigr)
=
\argmin_{X\in\mathcal X_{\rho_{\mathcal T}}}\inner{G^{(k)}}{X},
\qquad \text{where }
G^{(k)} \Let G\bigl(z^{(k)}\bigr)-G(z_1).
\]
The DCA subproblem therefore reduces to a linear optimization over $\mathcal X_{\rho_{\mathcal T}}$:
\begin{equation}
\label{eq:DC-sub}
\min_{X\in\mathcal X_{\rho_{\mathcal T}}}
\inner{G^{(k)}}{X}.
\end{equation}
Let $\lambda_1,\dots,\lambda_d$ be the eigenvalues of $G^{(k)}$, and set $\lambda_{\min} \Let \min_{i=1,\dots,d} \lambda_i$ and $\gamma_0 \Let \max\{0,-\lambda_{\min}\}$. The following theorem provides a closed-form solution to~\eqref{eq:DC-sub} in terms of the eigenvalues of $G^{(k)}$ and a scalar parameter.

\begin{theorem}[Analytical solution of the subproblem]
\label{thm:solution-of-DC-sub}
Suppose that $\rho_{\mathcal T} > 0$ and $G^{(k)} \neq 0$. Then problem~\eqref{eq:DC-sub} has a unique minimizer characterized by
\begin{subequations} \label{eq:X-opt-full}
\begin{equation}
\label{eq:X-opt}
X\opt = \gamma\opt\bigl(G^{(k)} + \gamma\opt I\bigr)^{-1},
\end{equation}
where the scalar $\gamma\opt$ is the unique solution to
\begin{equation}
\label{eq:lagrangian-target}
\sum_{i=1}^d \Bigl( \frac{\gamma}{\lambda_i + \gamma} - \log\Bigl(\frac{\gamma}{\lambda_i + \gamma}\Bigr) - 1 \Bigr)
= \rho_{\mathcal T},
\end{equation}
and satisfies $\gamma\opt \in (\gamma_0,\gamma_{\max}]$, where
\begin{equation}
\label{eq:upper-gamma}
\gamma_{\max}
= \frac{-\lambda_{\min} + \sqrt{\lambda_{\min}^2 + \frac{4 \sum_{i=1}^d \lambda_i^2}{\rho_{\mathcal T}}}}{2}.
\end{equation}
\end{subequations}
\end{theorem}
The closed form~\eqref{eq:X-opt-full} reduces the optimization problem~\eqref{eq:DC-sub} to a one-dimensional search for $\gamma\opt$ over $(\gamma_0,\gamma_{\max}]$. Lemma~\ref{lem:hybrid-gamma-search-app} shows that by applying the hybrid interval-selection scheme of~\cite[Theorem~2 and Section~4]{ref:ye1992new}, the scalar equation~\eqref{eq:lagrangian-target} can be solved numerically to absolute accuracy $\epsilon_\gamma$ in $O\bigl(\log\log(\Delta_\gamma/\epsilon_\gamma)\bigr)$ scalar iterations, where $\Delta_\gamma = \gamma_{\max} - \gamma_0$ is the length of the search interval.

\begin{lemma}[Complexity]
\label{lem:hybrid-gamma-search-app}
For every target accuracy $\epsilon_\gamma \in (0,\Delta_\gamma]$, the unique solution $\gamma\opt$ of~\eqref{eq:lagrangian-target} can be computed to absolute accuracy $\epsilon_\gamma$ in $O\bigl(\log\log(\Delta_\gamma/\epsilon_\gamma)\bigr)$ scalar iterations.
\end{lemma}

Algorithm~\ref{alg:DCA-algorithm} summarizes the resulting DCA scheme. The algorithm iteratively updates the matrix $X$ and the associated shortest path until a stopping criterion is met.
It returns a DCA-based estimate $\widehat{\delta}(\mathcal T; z_1)$ of the optimistic time gap, along with a candidate optimistic secondary path.
When $G^{(k)}=0$, the linear objective in the DCA subproblem~\eqref{eq:DC-sub} is identically zero on $\mathcal X_{\rho_{\mathcal T}}$, so every feasible matrix is an exact minimizer of the subproblem. In this degenerate case, we use a feasible restart to reduce the risk of remaining at the zero-gap fixed point that can occur when the current shortest path is the primary path $z_1$.

\begin{algorithm}[ht]
\caption{Difference-of-Convex Algorithm for estimating the optimistic time gap in~\eqref{eq:DC}}
\label{alg:DCA-algorithm}
\begin{algorithmic}[1]
\Require Edge embeddings $\{\phi_{\mathcal T}[e]\}_{e\in\mathcal E}$, primary path $z_1\in\mathcal Z$, feasible path set $\mathcal Z$, uncertainty radius $\rho_{\mathcal T}\ge0$, DCA stopping tolerance $\epsilon>0$, maximum iterations $K\in\mathbb N_+$.
\State Initialize $X^{(0)} \gets I$ and precompute $G(z_1)$ from~\eqref{eq:outer_product}.
\For{$k=0,1,\dots,K-1$}
    \State Compute $z^{(k)} \in \argmin_{z\in\mathcal Z} c_{\mathcal T}(X^{(k)})^\top z$ with costs from~\eqref{eq:cX}.
    \State Set $\zeta^{(k)}\gets -G(z^{(k)})$ and $G^{(k)}\gets G(z^{(k)})-G(z_1)$.
    \State \textbf{if} $\rho_{\mathcal T}=0$ \textbf{then} $X^{(k+1)}\gets I$.
    \State \textbf{else if} $G^{(k)}=0$ \textbf{then} randomly choose a feasible restart matrix $X^{(k+1)}\in\mathcal X_{\rho_{\mathcal T}}\setminus\{X^{(k)}\}$.
    \State \textbf{else}
        \State \;\;\quad Compute an eigendecomposition $G^{(k)} = Q\Lambda_\lambda Q^\top$, where $\Lambda_\lambda\gets\Diag(\lambda_1,\dots,\lambda_d)$.
        \State \;\;\quad Find $\gamma\opt$ by Theorem~\ref{thm:solution-of-DC-sub}, and set $X^{(k+1)}\gets \gamma\opt\bigl(G^{(k)}+\gamma\opt I\bigr)^{-1}$.
    \State \textbf{end if}
    \State Compute $W_k\gets f_1(X^{(k)})-f_1(X^{(k+1)})-\inner{\zeta^{(k)}}{X^{(k)}-X^{(k+1)}}$.
    \State \textbf{if} $\rho_{\mathcal T}=0$ or $\bigl(G^{(k)}\neq 0 \text{ and } W_k\le\epsilon\bigr)$ \textbf{ then break.}
\EndFor
\State Let $k_{\mathrm{last}}\gets k+1$ and compute $z^{(k_{\mathrm{last}})} \in \argmin_{z\in\mathcal Z} c_{\mathcal T}(X^{(k_{\mathrm{last}})})^\top z$.
\State Output the estimate $\widehat{\delta}(\mathcal T; z_1)\gets c_{\mathcal T}(X^{(k_{\mathrm{last}})})^\top z_1 - c_{\mathcal T}(X^{(k_{\mathrm{last}})})^\top z^{(k_{\mathrm{last}})}$ and the candidate secondary path $z^{(k_{\mathrm{last}})}$.
\end{algorithmic}
\end{algorithm}
Because $X^{(k_{\mathrm{last}})}\in\mathcal X_{\rho_{\mathcal T}}$ is feasible and $z^{(k_{\mathrm{last}})}$ is shortest under $X^{(k_{\mathrm{last}})}$, the returned estimate satisfies $\widehat{\delta}(\mathcal T; z_1) \leq \delta(\mathcal T; z_1)$. Thus $\widehat{\delta}(\mathcal T; z_1)$ is a valid lower-bound estimate of the exact optimistic gap.

\begin{proposition}[Computational complexity per DCA iteration]
\label{prop:complexity}
Let $d$ be the embedding dimension of $\phi_{\mathcal T}[e]$, and let $|\mathcal V|$ and $|\mathcal E|$ denote the number of nodes and edges in the traffic graph. Let $\Delta_\gamma = \gamma_{\max} - \gamma_0$, and let $\epsilon_\gamma$ denote the target accuracy for $\gamma$. One iteration of Algorithm~\ref{alg:DCA-algorithm} has worst-case complexity $O\bigl(|\mathcal E| d^2 + |\mathcal V| \log |\mathcal V| + d^3 + d \log\log(\Delta_\gamma/\epsilon_\gamma)\bigr)$.
\end{proposition}

\subsection{Convergence Analysis}\label{sec:solution-convergence}
Algorithm~\ref{alg:DCA-algorithm} produces the sequence of matrices $\{X^{(k)}\}$. This subsection studies the limit behavior of this sequence and the corresponding objective values. Lemma~\ref{lem:burg-compact} shows that the feasible set $\mathcal X_{\rho_{\mathcal T}}$ is nonempty, compact, and convex. Because we solve~\eqref{eq:DC-sub} over $\mathcal X_{\rho_{\mathcal T}}$, we can write~\eqref{eq:DC} as an unconstrained difference-of-convex program with objective
\[
X \mapsto f_1(X) + \iota_{\mathcal X_{\rho_{\mathcal T}}}(X) - f_2(X),
\]
where $\iota_{\mathcal X_{\rho_{\mathcal T}}}$ is the indicator function of $\mathcal X_{\rho_{\mathcal T}}$. The function $f_1 + \iota_{\mathcal X_{\rho_{\mathcal T}}}$ is proper, convex, and lower semicontinuous on $\mathbb S^d$. The function $f_2$ is convex and finite on $\mathcal X_{\rho_{\mathcal T}}$. Standard subdifferential calculus for convex functions and normal cones applies here; see \cite[Section~1]{ref:le2014dc}.
We call a point $X \in \mathcal X_{\rho_{\mathcal T}}$ \textit{DC-critical} for $f$ if $0 \in \partial\bigl(f_1 + \iota_{\mathcal X_{\rho_{\mathcal T}}}\bigr)(X) - \partial f_2(X)$, that is, if there exist $u \in \partial\bigl(f_1 + \iota_{\mathcal X_{\rho_{\mathcal T}}}\bigr)(X)$ and $v \in \partial f_2(X)$ with $u = v$. 

The next theorem provides the finite-run descent bound used by Algorithm~\ref{alg:DCA-algorithm} and the corresponding asymptotic DCA consequence, following the classical analysis in~\cite[Theorem~3(iv)]{ref:tao1997convex} and~\cite[Section~1]{ref:le2014dc}.

\begin{theorem}[Finite-run and asymptotic convergence of Algorithm~\ref{alg:DCA-algorithm}]
\label{thm:convergence-of-DCA}
Let $K_{\mathrm{run}}\le K$ denote the number of completed DCA iterations before Algorithm~\ref{alg:DCA-algorithm} terminates. For every completed iteration $k$, we have $f(X^{(k+1)})\le f(X^{(k)})$.
Moreover, letting $f\opt = \min_{X\in\mathcal X_{\rho_{\mathcal T}}} f(X)$, for $m=1,\ldots,K_{\mathrm{run}}$,
\[
0\le \sum_{k=0}^{m-1}W_k
\le f(X^{(0)})-f(X^{(m)})
\le f(X^{(0)})-f\opt.
\]
If the same DCA recursion is continued indefinitely, then $\sum_{k=0}^{+\infty}W_k<+\infty$, $W_k\to0$, and every accumulation point $\bar X$ of the resulting sequence $\{X^{(k)}\}$ is a DC-critical point of $f$ on $\mathcal X_{\rho_{\mathcal T}}$. In this asymptotic regime, the objective values $f(X^{(k)})$ converge.
\end{theorem}

This completes the DCA-based implementation of Step~2 of IDEAL.

\section{Numerical Experiments}\label{sec:experiments}
We evaluate IDEAL using real-time traffic estimates by replaying historical incident locations. We compare it with practical baselines that a dispatch center could implement using depot locations and Google travel-time predictions. We report the share of incidents on which each strategy matches the best realized time among the simulated candidates, the regret when it does not, and how this trade-off changes as the second-ambulance dispatch rate rises. Appendix~\ref{ec:sha-tin-transfer} reports a separate transfer test from Hong Kong Island to Sha Tin.

\subsection{Experiment Setup}\label{sec:experiments:setup}
The main experiment is a controlled replay study of OHCA dispatch in the northwestern region of Hong Kong Island. The study area contains the Morrison Hill (MOR) and Mount Davis (MOU) ambulance depots and nearby neighborhoods, where the two depots often provide plausible competing responses. We build the directed road network $\mathcal G = (\mathcal V,\mathcal E)$ from the corresponding OSMnx subgraph, and selected historical HKFSD OHCA records provide the demand locations used in the replay. Details on data cleaning and study-region construction are provided in Appendix~\ref{ec:dataset-and-region}.

For each evaluation instance, we map the demand to its nearest node in $\mathcal G$, treat that node as the demand location, and launch the replay at the evaluation time. At the start of the replay, we collect contextual covariates from government data sources; the travel-time simulator then uses the real-time Google traffic estimates available during that replay. This design keeps the observed spatial demand pattern while exposing each policy to current traffic and environmental conditions.

\subsubsection{Travel-Time Simulator}
The evaluation uses a controlled travel-time replay built on the Google Routes API. Each evaluation instance is launched at the replay time, when contextual covariates are collected, and real-time traffic estimates are queried. For a given incident, all candidate simulations start from the same call time and use the same segment-level realized-time convention, so strategies are compared under a common traffic source, contextual state, and execution protocol. Adaptive simulations repeatedly traverse the next API-returned segment and re-query from the updated location; hybrid simulations first follow a prescribed prefix and then switch to the same adaptive rule. Appendix~\ref{ec:adaptive-simulation} and Appendix~\ref{ec:hybrid-simulation} give the full procedures.

\subsubsection{Baselines and Metrics}
The number of feasible depot-to-demand paths in $\mathcal G$ is too large to simulate exhaustively. We therefore define, for each incident, a fixed candidate set from the paths generated by the implemented policies and simulate all candidates under the same protocol. This candidate set serves as the benchmark for the reported candidate-optimal rates and regrets. For IDEAL, $z_1$ is the Google primary path: the Google-recommended path from whichever of MOR and MOU has the shorter Google-predicted travel time at dispatch. Algorithm~\ref{alg:DCA-algorithm} is then run conditional on this $z_1$ to obtain the secondary path $z_2$.

Concretely, for each incident $i$, we simulate the following candidate paths:
\begin{enumerate}[leftmargin=10mm]
\item[(P1)] the Google-recommended path from MOR with adaptive rerouting,
\item[(P2)] the Google-recommended path from MOU with adaptive rerouting,
\item[(P3)] the IDEAL secondary path $z_2$, simulated using the hybrid execution rule below.
\end{enumerate}
Path (P3) is evaluated using the hybrid execution rule: the ambulance first follows the prescribed initial prefix of the IDEAL secondary path, then switches to adaptive rerouting for the remainder of the trip. This preserves the secondary-path information used at dispatch while still reflecting later traffic updates through fresh Google Routes queries. Appendix~\ref{ec:hybrid-simulation} shows the exact procedure.

We compare IDEAL with four baselines:
\begin{enumerate}[leftmargin=7mm]
\item \textbf{Region-based dispatch:} dispatch one ambulance from the depot assigned by the fixed service-region rule for the incident location, using adaptive rerouting.
\item \textbf{Google primary dispatch:} dispatch one ambulance from the depot with the shorter Google-predicted travel time at dispatch, using adaptive rerouting.
\item \textbf{Google dual dispatch:} dispatch one ambulance from each depot, both using adaptive rerouting, and take the earlier arrival.
\item \textbf{Google interval-based selective dispatch:} start from Google primary dispatch and add one ambulance from the other depot, using adaptive rerouting, only when the selected depot's Google pessimistic estimate exceeds the other depot's Google optimistic estimate by more than a time threshold.
\end{enumerate}
These baselines cover region-based, point-estimate, always-dual, and Google interval-based selective rules under the same adaptive simulator. This design keeps the traffic oracle and execution convention common across strategies. IDEAL changes the uncertainty representation and secondary-path selection while using the same adaptive simulation protocol. We also report \textit{IDEAL dual}, an always-dual operating point that dispatches one ambulance along the Google primary path and one ambulance along the IDEAL secondary path for every incident.

All reported metrics are based on the simulator's realized travel times. Let $n$ be the number of evaluation instances. For incident $i$ and strategy $s$, let $t^{(i)}(s)$ denote the realized travel time in seconds obtained by the simulated execution of strategy $s$ on that incident. This value is computed from the simulated candidate paths selected by the strategy. For a single-dispatch strategy, $t^{(i)}(s)$ is the realized time of the selected candidate path. For a dual-dispatch strategy, $t^{(i)}(s)$ is the smaller of the two realized candidate times, that is, the arrival time of the earliest ambulance at the demand. We also track dispatch volume through
\[
\bar a(s) \Let \frac{1}{n}\sum_{i=1}^n a^{(i)}(s),
\]
where $a^{(i)}(s)\in\{1,2\}$ is the number of ambulances dispatched under strategy $s$ for incident $i$.

\noindent\textbf{Candidate-optimal benchmark.}
Let $\mathcal P_i$ be the set of candidate paths simulated for incident $i$ as described above, with each path following its specified execution rule. For any candidate path $p\in\mathcal P_i$, let $t^{(i)}(p)$ denote the simulator's realized travel time in seconds for that candidate. Define the best realized time within this candidate set by
\[
t^{(i)}_{\min} \Let \min_{p\in\mathcal P_i} t^{(i)}(p).
\]

\noindent\textbf{Candidate-optimal rate.}
A strategy $s$ is \textit{candidate-optimal} on incident $i$ if it attains the benchmark, i.e., if $t^{(i)}(s)=t^{(i)}_{\min}$. The candidate-optimal rate is the fraction of incidents for which this occurs:
\[
\mathrm{CandOptRate}(s) \Let \frac{1}{n}\sum_{i=1}^n \mathbbm{1}\{t^{(i)}(s)=t^{(i)}_{\min}\}.
\]
Figure~\ref{fig:opt-rate} reports $\mathrm{CandOptRate}(s)$ for each strategy.

\noindent\textbf{Regret and tail risk.}
To quantify how far a strategy $s$ falls short of the best candidate, we define the regret as
\[
\operatorname{Reg}_i(s) \Let t^{(i)}(s)-t^{(i)}_{\min} \ge 0.
\]
We summarize regret by its sample mean, selected quantiles, and tail-risk measures. OHCA outcomes deteriorate quickly when response is delayed, so a few severe delays matter even when the mean regret is small. We therefore report the 95th and 99th percentiles of regret, together with $\mathrm{CVaR}_{0.95}$ and $\mathrm{CVaR}_{0.99}$.

\noindent\textbf{Pareto plots.}
We sweep the time threshold over a grid of values. Each threshold value yields one operating point for IDEAL, with a corresponding average dispatch volume and regret level.
We plot average dispatch volume $\bar a(s)$ on the horizontal axis, and either mean regret or $\mathrm{CVaR}_{0.95}$ regret on the vertical axis in Figure~\ref{fig:pareto}. Each point corresponds to one threshold value.

\noindent\textbf{Regret distribution.}
Figure~\ref{fig:regret-cdf} plots the empirical cumulative distribution function of regret. The horizontal axis shows regret in seconds, and the vertical axis shows the fraction of incidents with regret no greater than that value. Table~\ref{tab:regret-metrics} reports the numerical summary.

\noindent\textbf{One-sided statistical tests.}
We test whether the per-incident gains of the dual-dispatch version of IDEAL are systematic. For each baseline $b$, we apply a one-sided paired Wilcoxon signed-rank test on $\operatorname{Reg}_i(b)-\operatorname{Reg}_i(\text{IDEAL dual})$, with the alternative hypothesis that baseline regret is larger. Table~\ref{tab:wilcoxon} reports these results.

\subsection{Results}\label{sec:experiments:results}

Figure~\ref{fig:opt-rate} compares candidate-optimal rates across strategies. Region-based dispatch is candidate-optimal for 49.7\% of incidents, suggesting that static regions do not capture real-time local congestion. Google primary dispatch raises the rate to 84.1\%, and Google dual dispatch reaches 88.4\% by sending ambulances from both depots. The IDEAL dual endpoint dispatches two ambulances for every incident and achieves a 95.8\% candidate-optimal rate. Figure~\ref{fig:opt-rate} also marks another operating point that exceeds the strongest baseline, reaching an 88.8\% candidate-optimal rate with a 48.7\% average increase in dispatch volume over single-dispatch baselines.

\begin{figure}[ht]
    \centering
    \includegraphics[width=0.5\linewidth]{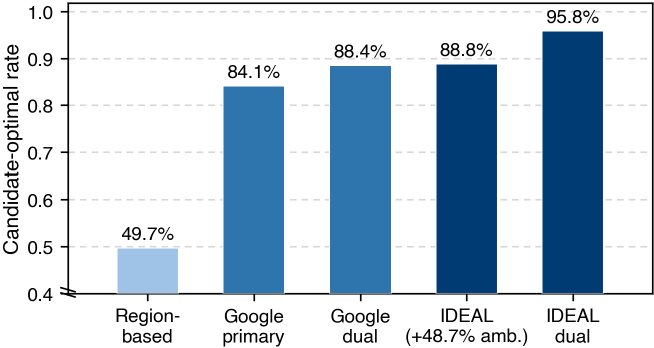}
    \caption{Candidate-optimal rates across strategies. The figure marks the always-dual IDEAL operating point and another IDEAL operating point that exceeds the strongest baseline with a lower average dispatch volume.}
    \label{fig:opt-rate}
\end{figure}

Figure~\ref{fig:pareto} plots mean regret and $\mathrm{CVaR}_{0.95}$ regret against average dispatch volume. A higher time threshold reduces second-ambulance dispatch and moves the operating point left. IDEAL achieves most of the mean-regret reduction with around 50\% more ambulances and approaches the best point with around 80\% more ambulances.
At $\bar a=2$, $\mathrm{CVaR}_{0.95}$ regret falls from 60.9~seconds under Google dual dispatch to 20.5~seconds under IDEAL, which is a 66.3\% reduction. These tail improvements show that IDEAL effectively reduces rare but severe delays, which are the incidents where additional travel-time delays matter most.

\begin{figure}[h]
    \centering
    \includegraphics[width=\linewidth]{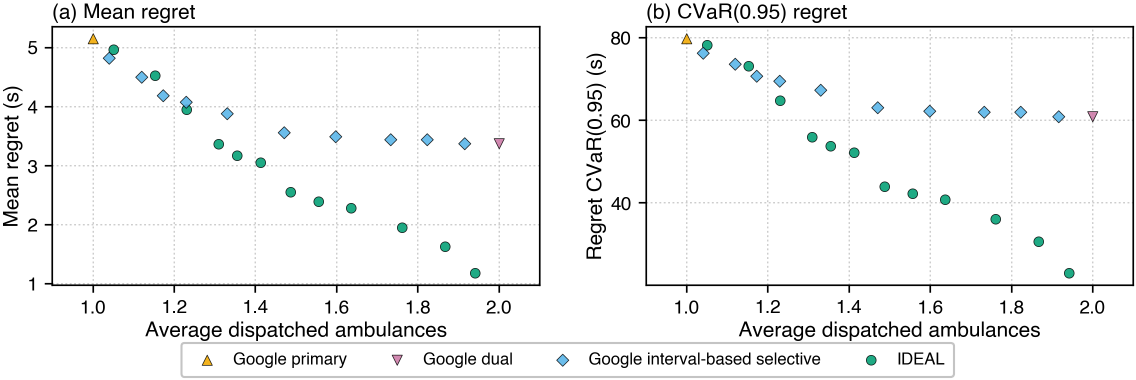}
    \caption{Mean-regret and tail-risk trade-offs for the same threshold sweep. Left: mean regret versus average dispatch volume. Right: $\mathrm{CVaR}_{0.95}$ regret versus average dispatch volume.}
    \label{fig:pareto}
\end{figure}

Figure~\ref{fig:regret-cdf} plots the empirical cumulative distribution function of regret. The horizontal axis shows regret in seconds; the vertical axis shows the fraction of incidents with regret no greater than that value. The IDEAL (dual) curve rises faster than the baseline curves, indicating lower regret in both typical incidents and the high-regret tail. The numerical summaries in Table~\ref{tab:regret-metrics} confirm this distributional pattern across mean regret, upper quantiles, and CVaR tail metrics. Under IDEAL (dual), the mean regret is 1.0~seconds, and the 95th-percentile regret is 0.0~seconds. Thus, for at least 95\% of incidents, IDEAL dual matches the best realized time in the candidate set.

\begin{figure}[h]
    \centering
    \includegraphics[width=0.45\linewidth]{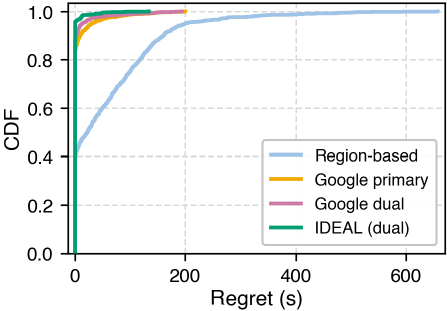}
    \caption{Empirical CDF of regret. For each regret $x$, the curve reports the fraction of incidents with regret at most $x$.}
    \label{fig:regret-cdf}
\end{figure}

\begin{table}[h]
    \centering
    \small
    \setlength{\tabcolsep}{3pt}
    \caption{Summary statistics of regret in seconds. Lower values indicate smaller deviations from the best realized candidate path. Rows containing a percentage correspond to threshold-based selective policies at selected dispatch-volume levels, where the percentage denotes the increase in average dispatch volume relative to a single-dispatch policy, and \textbf{IDEAL, dual} is the always-dual operating point.}
    \label{tab:regret-metrics}
    \begin{tabular}{lrrrrr}
    \toprule
    Strategy & Mean $\downarrow$ & P95 $\downarrow$ & P99 $\downarrow$ & $\mathrm{CVaR}_{0.95}\downarrow$ & $\mathrm{CVaR}_{0.99}\downarrow$ \\
    \midrule
Region-based & 59.6 & 200.0 & 414.4 & 324.1 & 502.0 \\
Google primary & 5.2 & 29.1 & 121.1 & 79.8 & 156.8 \\
Google interval, +25.2\% amb. & 4.1 & 20.0 & 113.4 & 69.5 & 156.8 \\
Google interval, +50.0\% amb. & 3.5 & 15.0 & 104.2 & 62.8 & 147.0 \\
Google interval, +75.3\% amb. & 3.4 & 13.0 & 104.2 & 62.0 & 147.0 \\
Google dual & 3.4 & 13.0 & 104.2 & 60.9 & 147.0 \\
\textbf{IDEAL, +25.0\% amb.} & 3.7 & 21.0 & 85.2 & 61.6 & 129.7 \\
\textbf{IDEAL, +50.0\% amb.} & 2.5 & 12.0 & 57.0 & 43.3 & 98.2 \\
\textbf{IDEAL, +75.0\% amb.} & 2.0 & 7.0 & 50.0 & 36.8 & 95.9 \\
\textbf{IDEAL, dual} & 1.0 & 0.0 & 37.0 & 20.5 & 63.6 \\
\bottomrule
\end{tabular}
\end{table}

We also test for consistent per-incident improvements against baselines using one-sided paired Wilcoxon signed-rank tests on $\operatorname{Reg}_i(b)-\operatorname{Reg}_i(\text{IDEAL dual})$. Table~\ref{tab:wilcoxon} reports the total evaluation sample size $n$, the number $n_{\ne0}$ of nonzero paired regret differences used by the Wilcoxon signed-rank test, the mean regret difference over all evaluation instances with a normal-approximation 95\% confidence interval, the Wilcoxon signed-rank statistic computed after discarding zero differences, and the one-sided asymptotic signed-rank $p$-value. Tied absolute differences receive average ranks; no multiple-comparison adjustment is applied. At the 5\% significance level, we reject the null against all three baselines. The largest $p$-value occurs in the comparison with Google dual dispatch and equals $4.0\times 10^{-7}$, indicating that the improvement of IDEAL dual remains statistically significant even against the strongest baseline.

\begin{table}[h]
\centering
\small
\caption{One-sided paired Wilcoxon signed-rank tests on per-incident regret against baselines in the northwest Hong Kong Island evaluation. The alternative hypothesis is that baseline regret exceeds the regret of \textbf{IDEAL dual}.}
\label{tab:wilcoxon}
\begin{tabular}{lrrrrr}
\toprule
Baseline & Evaluation $n$ & Wilcoxon $n_{\ne0}$ & Mean diff.\ (s; normal 95\% CI) & Statistic & $p$-value \\
\midrule
(1) Region-based & 1160 & 713 & 58.6 [53.7, 63.6] & 251497 & $3.4\times 10^{-113}$ \\
(2) Google primary & 1160 & 161 & 4.1 [3.0, 5.2] & 13041 & $1.7\times 10^{-28}$ \\
(3) Google dual & 1160 & 183 & 2.3 [1.2, 3.4] & 11957.5 & $4.0\times 10^{-7}$ \\
\bottomrule
\end{tabular}
\end{table}

\noindent\textbf{Managerial interpretation.}
Figures~\ref{fig:opt-rate}--\ref{fig:regret-cdf} translate the threshold into an operational capacity lever. A large threshold corresponds to low vehicle availability and reserves backup dispatch for cases where the primary path can be substantially dominated; a small threshold provides greater protection at higher dispatch volumes. A dispatch center can tune the threshold using local availability and coverage targets. In this study, the largest gains are observed in tail regret, which is operationally important for OHCA because rare, long delays can be clinically consequential.

\noindent\textbf{Transfer test.} Appendix~\ref{ec:sha-tin-transfer} repeats the same evaluation design in Sha Tin with a different road network and a different depot pair. IDEAL continues to outperform the Google-based baselines in this transfer setting without retraining the models. The results suggest that the learned scenario construction and dispatch overlay remain useful beyond the main study region.

\section{Conclusions}\label{sec:conclusions}
We study selective ambulance dispatch for OHCA when traffic conditions are context-dependent and ambulance capacity is limited. IDEAL formalizes this trade-off by comparing a primary path with plausible alternatives over a context-specific scenario set. It dispatches a second ambulance only when the optimistic time gap exceeds a decision threshold.

The IDEAL framework combines three components. First, it learns context-specific edge travel times from trip-level dispatch records whose routes are unobserved. This is done through a weakly supervised bilevel representation network. The resulting nonsmooth objective is trained with mini-batch conservative gradients, for which we prove an asymptotic accumulation-set guarantee.
Second, IDEAL constructs correlated travel-time scenarios by perturbing a shared metric in the learned representation space with Burg divergence controlling the perturbation size. A radius prediction network then predicts the context-specific radius based on the risk of underprediction.
Third, IDEAL computes the optimistic gap through a difference-of-convex reformulation. The resulting DCA oracle is efficiently solved by reducing its matrix subproblem to a one-dimensional scalar search. We also provide complexity guarantees for this computation.

Co-developed with the Hong Kong Fire Services Department, IDEAL is evaluated using a real-time adaptive replay simulator against historical OHCA demand locations in Hong Kong. It operates as an overlay on existing routing and dispatch workflows: external engines provide primary routes and real-time traffic estimates, and local capacity information determines the threshold. The results provide controlled replay evidence that selective dual dispatch improves the response-time/resource trade-off relative to region-based and Google-based policies. IDEAL achieves higher candidate-optimal rates, lower mean regret, and lower tail regret.

\section*{Acknowledgments}
Viet Anh Nguyen gratefully acknowledges the support from the CUHK’s Improvement on Competitiveness in Hiring New Faculties Funding Scheme, UGC ECS Grant 24210924, and UGC GRF Grant 14208625.

\newpage
\bibliographystyle{siam}
\bibliography{bibliography}

\newpage
\appendix
\section{Simulation Frameworks} \label{ec:simulation-frameworks}

\subsection{Adaptive Simulation Procedure with Continuous Google Routes Re-querying}
\label{ec:adaptive-simulation}

We simulate an ambulance's travel time under dynamic traffic conditions by repeatedly re-querying the Google Routes API as the ambulance progresses along its route.

\begin{enumerate}[leftmargin=7mm]
\item \textbf{Initial query.}
At simulation time $T=0$, corresponding to demand arrival, we query the API for a route from the origin $v_{\mathrm{o}}$ to the destination $D$. The response gives the initial predicted duration $t_{\mathrm{dur}}^0$, the first road segment $(v_{\mathrm{o}},v_1)$, and its predicted travel time $\Delta t_0$. When the \texttt{OPTIMISTIC} and \texttt{PESSIMISTIC} traffic models are requested, we also record their duration estimates. We initialize $v_{\mathrm{current}}\leftarrow v_{\mathrm{o}}$ and $T\leftarrow 0$.

\item \textbf{Segment update and re-querying.}
At each step, let $(v_{\mathrm{current}},v_{\mathrm{next}})$ be the next segment returned by the current API query, with predicted travel time $\Delta t$. The simulator treats $\Delta t$ as the realized time for that segment, sets $T\leftarrow T+\Delta t$, and updates $v_{\mathrm{current}}\leftarrow v_{\mathrm{next}}$. If $v_{\mathrm{current}}\neq D$, it immediately re-queries the API from $v_{\mathrm{current}}$ to $D$ and obtains an updated route, remaining duration, next segment, and segment time.

\item \textbf{Output.}
The procedure repeats until $v_{\mathrm{current}}=D$. The final clock value $T_{\mathrm{final}}$ is recorded as the realized travel time for the simulated trip.
\end{enumerate}

\subsection{Hybrid Simulation Procedure}
\label{ec:hybrid-simulation}

We simulate an ambulance's travel time under a hybrid execution rule by first following a fixed, precomputed prefix, and then switching to adaptive Google Routes to re-query the remainder of the trip.

\begin{enumerate}[leftmargin=7mm]
\item \textbf{Initial query and prefix initialization.}
Before the simulation begins, we are given a precomputed node sequence $\pi=(v_{\mathrm{o}},v_1,\ldots,v_{m-1},D)$ and a predetermined prefix length $q$. The first $\min\{q,m\}$ segments of $\pi$ form the fixed prefix. At simulation time $T=0$, corresponding to demand arrival, we initialize $v_{\mathrm{current}}\leftarrow v_{\mathrm{o}}$, $T\leftarrow 0$, and the prefix counter $h\leftarrow 0$. If $q>0$, we query the API for the first prescribed segment $(v_{\mathrm{o}},v_1)$ and obtain its predicted travel time $\Delta t_0$.

\item \textbf{Fixed-prefix segment update and adaptive switch.}
While $h<\min\{q,m\}$ and $v_{\mathrm{current}}\neq D$, let $(v_{\mathrm{current}},v_{\mathrm{next}})$ be the next prescribed segment in the fixed prefix, with predicted travel time $\Delta t$ returned by the current API query. The simulator treats $\Delta t$ as the realized time for that segment, sets $T\leftarrow T+\Delta t$, updates $v_{\mathrm{current}}\leftarrow v_{\mathrm{next}}$, and increments $h\leftarrow h+1$. If the fixed prefix has not ended and $v_{\mathrm{current}}\neq D$, the simulator queries the API for the next prescribed segment in $\pi$. Once the fixed prefix is completed, if $v_{\mathrm{current}}\neq D$, the simulator switches to adaptive re-querying: it queries the API from $v_{\mathrm{current}}$ to $D$, obtains an updated route, remaining duration, next segment, and segment time, and then continues with the adaptive segment-update rule in Appendix~\ref{ec:adaptive-simulation}.

\item \textbf{Output.}
The procedure repeats until $v_{\mathrm{current}}=D$. The final clock value $T_{\mathrm{final}}$ is recorded as the realized travel time for the simulated trip.
\end{enumerate}

\section{Calibration of Clinical Risk and Threshold Function} \label{ec:calibration}
This section calibrates the threshold $\mathrm{thr}(\mathcal T)$ through a marginal-benefit calculation, consistent with outcome-based approaches used in EMS planning and decision-analytic evaluation. Clinical OHCA studies associate response times with survival outcomes~\cite{ref:chen2022alsresponse}, and EMS planning models use response-time survival functions to estimate expected survivors and marginal effects of system-design changes~\cite{ref:sund2013analytical,ref:erkut2008ambulance}. In our implementation, we use the local derivative of the specified risk curve at the nominal travel time to approximate the incremental clinical benefit of a predicted time saving. The resulting threshold compares this benefit with the operational cost after converting both to a common scale, analogous to net-benefit approaches that trade off benefits and harms through a prespecified exchange rate~\cite{ref:vickers2016net}.

For each context $\mathcal T$, let $R(t;\mathcal T)$ be a differentiable function that maps a travel time $t$ to a scalar measure of clinical risk, with larger values corresponding to worse outcomes. Later arrivals increase risk, so $R(t;\mathcal T)$ is increasing in $t$. Let $t_{\text{nominal}}(\mathcal T) = \hat c_{\mathcal T}^\top z_1$ denote the nominal travel time of the primary ambulance under our model, where $z_1$ is the primary path used by IDEAL in context $\mathcal T$.

We approximate the behavior of $R(t;\mathcal T)$ near $t_{\text{nominal}}(\mathcal T)$ by a first-order expansion,
\begin{equation}
\label{eq:linearized-R-app}
\widetilde R(t;\mathcal T) \Let R\bigl(t_0;\mathcal T\bigr)
 + \lambda(t_0;\mathcal T)\bigl(t - t_0\bigr),
\qquad
\lambda(t_0;\mathcal T) \Let \frac{\partial R}{\partial t}\bigl(t_0;\mathcal T\bigr) > 0,
\end{equation}
where $t_0$ is a reference travel time and $\lambda(t_0;\mathcal T)$ is the local slope of the risk curve at $t_0$. When $t_0 = t_{\text{nominal}}(\mathcal T)$, $\lambda(t_{\text{nominal}}(\mathcal T);\mathcal T)$ measures how sensitive clinical risk is to small changes in travel time for the current demand.
Under this linear approximation, an improvement of $\Delta t$ units in travel time reduces clinical risk by approximately $\lambda(t_0;\mathcal T)\cdot\Delta t$. We want the threshold function to compare this optimistic reduction with the operational cost of sending a second ambulance.

We represent the operational burden of a second ambulance by a nonnegative scalar $C(\mathcal T)$ expressed in the same units as $R(t;\mathcal T)$. The quantity $C(\mathcal T)$ aggregates several components: the coverage opportunity cost of taking a second unit away from future demand, the expected exposure created by lower local availability, the expected travel time and temporary unavailability of a unit that may be canceled before arrival, redeployment or recall costs, and any other cost elements that the system designer chooses to include. The dispatch center may treat $C(\mathcal T)$ as constant or allow it to depend on coarse indicators, such as current system load, available units, or coverage status.

For a fixed context $\mathcal T$, we define the time-threshold function by $\eta_{\mathcal T}(t) \Let C(\mathcal T) / \lambda(t;\mathcal T)$,
where $\lambda(t;\mathcal T)$ is the local slope of the risk curve at time $t$. The value $\eta_{\mathcal T}(t)$ is the arrival-time gap that, under the linear model~\eqref{eq:linearized-R-app}, produces an optimistic reduction in clinical risk equal to the cost $C(\mathcal T)$. When the optimistic time gap exceeds this value, the reduction in clinical risk outweighs the cost of activating a second ambulance.

In the IDEAL policy, we apply $\eta_{\mathcal T}(\cdot)$ at the nominal travel time of the primary ambulance. The time threshold used in the decision rule~\eqref{eq:decision-rule} is therefore
\[
\mathrm{thr}(\mathcal T) \Let \eta_{\mathcal T}\bigl(t_{\text{nominal}}(\mathcal T)\bigr)
= \frac{C(\mathcal T)}{\lambda\bigl(t_{\text{nominal}}(\mathcal T);\mathcal T\bigr)}.
\]
Under this choice, we have the equivalence
\[
\delta(\mathcal T; z_1) > \eta_{\mathcal T}\bigl(t_{\text{nominal}}(\mathcal T)\bigr) \quad \iff \quad
\lambda\bigl(t_{\text{nominal}}(\mathcal T);\mathcal T\bigr) \cdot \delta(\mathcal T; z_1) > C(\mathcal T),
\]
which compares an optimistic reduction in linearized clinical risk with the operational cost of dispatching a second ambulance.

\section{Data Sources and Preprocessing Details}\label{ec:data-details}
This section describes the data sources and preprocessing steps that produce the modeling inputs.

\subsection{HKFSD OHCA Dataset and Study Region}
\label{ec:dataset-and-region}

The Hong Kong Fire Services Department (HKFSD) manages out-of-hospital emergencies across the territory. As of March~1,~2025, HKFSD operates 536 ambulances from 41 depots.\footnote{\url{https://www.hkfsd.gov.hk/eng/source/fire_services_Factsheet_eng.pdf}, accessed March~1,~2025.} We use operational records for OHCA cases from January~2014 to December~2023, which yield 72{,}157 OHCA records after basic cleaning.

For illustration and controlled experimentation, we focus on the northwest region of Hong Kong Island. The study area spans latitudes from $22.265669$ to $22.290994$ and longitudes from $114.116190$ to $114.184312$. Within this area, the Mount Davis (MOU) and Morrison Hill (MOR) depots frequently compete to serve nearby OHCAs. We restrict attention to records that are primarily served by these two depots and whose demands lie within the study region.

\subsection{Traffic Network Construction}
\label{ec:traffic-network}

We construct a drivable road network $\mathcal G = (\mathcal V, \mathcal E)$ for the study region using \texttt{OSMnx}.\footnote{\url{https://osmnx.readthedocs.io/en/stable/}, accessed March~1,~2025.} The construction includes public and service roads that are accessible to ambulances. Each node $v \in \mathcal V$ represents an intersection or a street endpoint. Each directed edge $e \in \mathcal E$ represents a road segment with length $l[e] > 0$ measured in meters. We map geocoded locations to the nearest network node. 
In the final subgraph for the study area, we obtain 1{,}415 nodes and 2{,}541 directed edges. For each edge $e$, we extract a static feature vector $\mathcal A[e] \in \R^{d_{\mathcal A}}$ that includes geometric and categorical characteristics. Table~\ref{tab:edge_features} lists these features.

\begin{table}[h]
\centering
\small
\caption{Geographic features for each edge in the traffic network.}
\begin{tabularx}{\textwidth}{llX}
\hline
\textbf{Feature name} & \textbf{Dim.} & \textbf{Description} \\
\hline
Length $l[e]$ & 1 & Edge length (m) \\
Speed limit & 1 & Posted or default speed limit (km/h; default 50) \\
Number of neighbors & 2 & Degree counts at origin and destination nodes \\
Is one-way & 1 & Indicator equal to 1 if one-way, 0 otherwise \\
Number of lanes & 1 & Total number of lanes \\
Type of edge & 14 & One-hot indicators for OSM functional classes (motorway, primary, residential, \ldots) \\
\hline
\textbf{Total} & \textbf{20} & \\
\hline
\end{tabularx}
\label{tab:edge_features}
\end{table}

\subsection{Environmental Covariates}
\label{ec:environmental-covariates}

We attach environmental covariates to each record based on its time and location. Hourly weather variables come from a public database maintained by the University of Hong Kong.\footnote{\url{https://cowin.hku.hk/english/downloadbulk.html}, accessed March~1,~2025.} For each record, we extract the weather at the hour closest to the dispatch time. 
The context vector $\mathcal T^{(i)} \in \R^{d_{\mathcal T}}$ includes both weather variables and calendar indicators. Table~\ref{tab:environmental_features} summarizes the encoding.

\begin{table}[h]
\centering
\small
\caption{Environmental features for each OHCA record.}
\begin{tabularx}{\textwidth}{llX}
\hline
\textbf{Feature name} & \textbf{Dim.} & \textbf{Description} \\
\hline
Weather variables & 5 & Temperature, rainfall, relative humidity, wind speed, sea-level pressure \\
Hour & 1 & Hour of day (0--23) \\
Day of week & 7 & One-hot indicators for Monday to Sunday \\
Month & 12 & One-hot indicators for January to December \\
Is holiday & 1 & Indicator for public holiday \\
Is holiday eve & 1 & Indicator for the day before a public holiday \\
\hline
\textbf{Total} & \textbf{27} & \\
\hline
\end{tabularx}
\label{tab:environmental_features}
\end{table}

\subsection{Travel-Time Target Extraction}
\label{ec:travel-time-target}

The raw timing measure is the interval from receipt of the emergency call to CPR initiation. It includes call processing, crew turnout/preparation, driving, vertical access, and any remaining time until CPR initiation. Our travel-time model focuses on the driving component. We therefore subtract two non-driving components to form the travel-time target $t^{(i)}$.

\begin{enumerate}[leftmargin=7mm]
\item \textbf{Depot-specific preparation time.} We estimate, for each depot, the typical time required to process the call, prepare the vehicle, and leave the station. This time depends on call routing and station-level operations, but not on road conditions.
\item \textbf{Vertical-access time.} We estimate an elevation component that captures travel within buildings, such as elevator rides and stair use. We fit an ordinary least squares model that regresses call-receipt-to-CPR time minus depot preparation time on the floor number. This model yields an estimate of $1.90$ seconds per floor, which we apply as a linear adjustment.
\end{enumerate}

We compute the travel-time target for record $i$ as
\[
t^{(i)} = \bigl(\text{call-receipt-to-CPR interval}\bigr)
- \bigl(\text{estimated depot preparation time}\bigr)
- \bigl(\text{estimated elevation time}\bigr).
\]

Table~\ref{tab:additional_features} summarizes the main components of this calculation.

\begin{table}[h]
\centering
\small
\caption{Additional quantities used to construct the travel-time target.}
\begin{tabularx}{\textwidth}{lX}
\hline
\textbf{Quantity} & \textbf{Description} \\
\hline
Incident location & Geocoded latitude--longitude of the demand, snapped to the road network \\
Depot preparation time & Estimated call-processing and vehicle-preparation time, by depot \\
Elevation time & $1.90$ seconds per floor estimated by OLS on adjusted call-receipt-to-CPR intervals \\
\hline
\end{tabularx}
\label{tab:additional_features}
\end{table}

\subsection{Exclusions and Final Sample}
\label{ec:exclusions}

We apply exclusion rules to focus on records in which travel time dominates the response interval and the operational conditions align with the modeling assumptions. The following filters define the final sample:
\begin{enumerate}[leftmargin=7mm]
\item \textbf{Pandemic anomalies.} We drop records in February, March, and December~2022, when COVID-related policies and traffic patterns changed sharply relative to typical operations.
\item \textbf{Extreme elevations.} We remove records with unrecognized floor numbers or floor numbers above 60, because the elevation adjustment becomes unreliable in those cases.
\item \textbf{Spatial outliers.} We exclude trips shorter than $0.5$ km or longer than $10$ km, measured along the road network, to remove very short trips and atypically long transfers.
\item \textbf{Extreme durations.} We remove records with call-receipt-to-CPR intervals longer than 30 minutes, which often correspond to exceptional clinical or operational circumstances.
\end{enumerate}

After applying these filters and restricting to the two-depot subset (MOU and MOR) within the study region, we retain 1{,}781 OHCA records. These records form the samples used in Section~\ref{sec:bilevel-representation-learning}.

\section{Technical Proofs for Section~\ref{sec:ideal-framework}}\label{ec:proofs-ideal-framework}
\begin{proof}[Proof of Theorem~\ref{thm:ideal_optimality}]
Fix a context $\mathcal T$ and a primary path $z_1 \in \mathcal Z$. Let $\delta(\mathcal T; z_1)$ be the optimistic time gap defined in~\eqref{eq:delta-rho}. Since $\mathcal Z$ is finite, the map $c\mapsto \min_{z\in\mathcal Z}c^\top z$ is the pointwise minimum of finitely many linear functions and is continuous. Hence $c\mapsto c^\top z_1-\min_{z\in\mathcal Z}c^\top z$ is continuous. Since $\mathcal C(\mathcal T)$ is nonempty and compact, the maximum in~\eqref{eq:delta-rho} is attained. Hence there exists
\[
c\opt \in \argmax_{c\in\mathcal C(\mathcal T)}
\bigl(c^\top z_1 - \min_{z\in\mathcal Z} c^\top z\bigr),
\]
and we can choose an associated secondary path $z_2 \in \argmin_{z\in\mathcal Z} (c\opt)^\top z$.

We first show that the pair $(c\opt,z_2)$ constructed by the IDEAL algorithm achieves the optimistic time gap in the sense that
\begin{equation}
\label{eq:delta-equals-Delta-max}
\Delta(z_2;\mathcal T,z_1) = \delta(\mathcal T; z_1)
= \max_{z\in\mathcal Z} \Delta(z;\mathcal T,z_1),
\end{equation}
and then show that the pair produced by IDEAL is an optimal solution of~\eqref{eq:pthr}, with the convention that zero surplus is resolved by choosing $\tau=0$.

For any secondary path $z\in\mathcal Z$ and any scenario $c\in\mathcal C(\mathcal T)$,
$c^\top z_1 - c^\top z \le c^\top z_1 - \min_{z\in\mathcal Z} c^\top z$. Taking the maximum over $c\in\mathcal C(\mathcal T)$ on both sides yields
\[
\Delta(z;\mathcal T,z_1)
= \max_{c\in\mathcal C(\mathcal T)} \bigl(c^\top z_1 - c^\top z\bigr)
\le
\max_{c\in\mathcal C(\mathcal T)}\bigl(c^\top z_1 - \min_{z\in\mathcal Z} c^\top z\bigr)
= \delta(\mathcal T; z_1),
\]
for every $z\in\mathcal Z$. Thus $\delta(\mathcal T; z_1)$ is an upper bound on $\Delta(z;\mathcal T,z_1)$ for all secondary paths $z$.

Now consider the specific path $z_2$ chosen by the algorithm. By construction, $(c\opt)^\top z_2 = \min_{z\in\mathcal Z} (c\opt)^\top z$, so
$(c\opt)^\top z_1 - (c\opt)^\top z_2
= (c\opt)^\top z_1 - \min_{z\in\mathcal Z} (c\opt)^\top z
= \delta(\mathcal T; z_1)$.
Therefore
\[
\Delta(z_2;\mathcal T,z_1)
= \max_{c\in\mathcal C(\mathcal T)} \bigl(c^\top z_1 - c^\top z_2\bigr)
\ge
(c\opt)^\top z_1 - (c\opt)^\top z_2
= \delta(\mathcal T; z_1).
\]
Combining this inequality with the upper bound $\Delta(z_2;\mathcal T,z_1)\le \delta(\mathcal T; z_1)$ gives
\[
\Delta(z_2;\mathcal T,z_1) = \delta(\mathcal T; z_1)
\quad\text{and}\quad
\Delta(z_2;\mathcal T,z_1) = \max_{z\in\mathcal Z} \Delta(z;\mathcal T,z_1),
\]
which proves~\eqref{eq:delta-equals-Delta-max} and establishes Item~1 in the theorem.

The optimization problem~\eqref{eq:pthr} can be analyzed by first fixing $z\in\mathcal Z$ and optimizing over $\tau\in\{0,1\}$. For a given $z$, the objective in~\eqref{eq:pthr} equals
\[
\max_{\tau\in\{0,1\}} \tau\bigl(\Delta(z;\mathcal T,z_1) - \mathrm{thr}(\mathcal T)\bigr)
=
\begin{cases}
0, & \text{if } \Delta(z;\mathcal T,z_1) \le \mathrm{thr}(\mathcal T),\\[0.5ex]
\Delta(z;\mathcal T,z_1) - \mathrm{thr}(\mathcal T) & \text{if } \Delta(z;\mathcal T,z_1) > \mathrm{thr}(\mathcal T).
\end{cases}
\]
Hence~\eqref{eq:pthr} is equivalent to
\[
\max_{z\in\mathcal Z}
\max\bigl\{0, \Delta(z;\mathcal T,z_1) - \mathrm{thr}(\mathcal T)\bigr\}.
\]
Let $\Delta_{\max} \Let \max_{z\in\mathcal Z} \Delta(z;\mathcal T,z_1)$. Then the optimal value of~\eqref{eq:pthr} equals $\max\{0,\Delta_{\max} - \mathrm{thr}(\mathcal T)\}$, and any maximizer $z\opt$ of $\Delta(z;\mathcal T,z_1)$ is optimal for~\eqref{eq:pthr} when paired with
\[
\tau\opt =
\begin{cases}
1, & \text{if } \Delta_{\max} > \mathrm{thr}(\mathcal T),\\
0, & \text{if } \Delta_{\max} \le \mathrm{thr}(\mathcal T).
\end{cases}
\]
Moreover, every optimal solution of~\eqref{eq:pthr} has $\tau=1$ exactly when $\Delta_{\max} > \mathrm{thr}(\mathcal T)$. If $\Delta_{\max}\le \mathrm{thr}(\mathcal T)$, the choice $\tau=0$ attains the optimal value $0$.

The identity~\eqref{eq:delta-equals-Delta-max} shows that $\Delta_{\max} = \delta(\mathcal T; z_1)$ and that the secondary path $z_2$ constructed by the IDEAL algorithm is a maximizer of $\Delta(z;\mathcal T,z_1)$. The algorithm uses the rule
\[
\tau =
\begin{cases}
1, & \text{if } \delta(\mathcal T; z_1) > \mathrm{thr}(\mathcal T),\\
0, & \text{if } \delta(\mathcal T; z_1) \le \mathrm{thr}(\mathcal T),
\end{cases}
\]
which matches the optimal choice of $\tau\opt$ above. Thus the pair $(\tau,z_2)$ produced by the IDEAL policy solves~\eqref{eq:pthr}, and the policy dispatches a second unit if and only if $\delta(\mathcal T; z_1) > \mathrm{thr}(\mathcal T)$, 
which is precisely the decision rule~\eqref{eq:decision-rule}. This establishes Item~2 and completes the proof.
\end{proof}

\section{Technical Proofs for Section~\ref{sec:bilevel-representation-learning}}\label{ec:proofs-bilevel-representation-learning}

\subsection{Proofs for Section~\ref{sec:training-convergence}}
\label{ec:training-convergence-proof}

We recall the notion of a conservative field~\cite[Definitions~1--2]{ref:bolte2021conservative}.

\begin{definition}[Conservative field and potential function]\label{def:conservative-field} 
Let $J^{(i)}: \R^p \to \R$ be a locally Lipschitz continuous function. A set-valued mapping $D^{(i)}: \R^p \rightrightarrows \R^p$ with closed graph and nonempty compact values is a \textit{conservative field} for $J^{(i)}$ if for any two points $\phi_0, \phi \in \R^p$ and any absolutely continuous path $\gamma: [0, 1] \to \R^p$ with $\gamma(0) = \phi_0$ and $\gamma(1) = \phi$, we have
\[
J^{(i)}(\phi) = J^{(i)}(\phi_0) + \int_0^1 \max_{v \in D^{(i)}(\gamma(t))} \inner{\dot{\gamma}(t)}{v} \, \dd t.
\]
The function $J^{(i)}$ is called a \textit{potential function} for the conservative field $D^{(i)}$.
\end{definition}

Conservative fields extend gradients to certain nonsmooth objectives and align with the directions produced by backpropagation~\cite{ref:bolte2020mathematical}. When each $J^{(i)}$ admits a conservative field $D^{(i)}$, we define the aggregate field
\[
D_J(\phi) \Let \frac{1}{N}\operatorname{conv}\bigl(\sum_{i\in[N]} D^{(i)}(\phi)\bigr),
\]
where $\sum$ denotes the Minkowski sum of sets. Following \cite[Section~6.1]{ref:bolte2021conservative}, $D_J$ is a conservative field for $J$.

To prove Theorem~\ref{thm:convergence}, we establish that each per-sample objective $J^{(i)}$ in~\eqref{eq:Ji-def} is definable and locally Lipschitz, then construct a set-valued mapping that serves as a conservative field for $J^{(i)}$. The convergence statement follows from the stochastic approximation results of \cite{ref:bolte2021conservative}.

Throughout, fix a sample index $i\in[N]$ and write the trainable parameters as $\phi=(\vartheta,\theta,\psi)\in\R^p$. Let $\mathcal Z^{(i)}\Let\{z\in\{0,1\}^{|\mathcal E|}:Az=b^{(i)}\}$ denote the feasible depot-to-demand paths for sample $i$. For a parameter vector $\phi$, define the shortest-path minimizer set $\mathcal Z_{\mathcal T^{(i)}}^\phi \Let \argmin_{z\in \mathcal Z^{(i)}}(\hat c_{\mathcal T^{(i)}}^\phi)^\top z$, where $\hat c_{\mathcal T^{(i)}}^\phi$ is the predicted edge-travel-time vector as in~\eqref{eq:c-map}--\eqref{eq:c-def}. Since $\mathcal Z^{(i)}$ is a nonempty finite subset of $\{0,1\}^{|\mathcal E|}$, the set $\mathcal Z_{\mathcal T^{(i)}}^\phi$ is nonempty and finite for every $\phi$.

We define the candidate conservative-field mapping for $J^{(i)}$:
\begin{equation}\label{eq:CF}
\operatorname{CF}^{(i)}(\phi)
\Let \nabla_1\ell\bigl(g^{(i)}(\phi),t^{(i)}\bigr)\cdot
\operatorname{conv}\Bigl\{
\bigl(\nabla_\phi \hat c_{\mathcal T^{(i)}}^\phi\bigr)^\top z:\ z\in\mathcal Z_{\mathcal T^{(i)}}^\phi
\Bigr\}
+\beta \nabla_\phi \mathcal R(\vartheta),
\end{equation}
where $g^{(i)}$ is defined in~\eqref{eq:gi-def} and $\nabla_1\ell$ denotes the partial derivative of $\ell$ with respect to its first argument.
Here and below, $\nabla_\phi\mathcal R(\vartheta)$ denotes the gradient of the map $\phi=(\vartheta,\theta,\psi)\mapsto\mathcal R(\vartheta)$; equivalently, its $\theta$- and $\psi$-components are zero because $\mathcal R$ depends only on $\vartheta$.

Since $\mathcal Z^{(i)}$ is finite, enumerate it as $\{z_1,\dots,z_{K^{(i)}}\}$. Set $F_k(\phi)\Let(\hat c_{\mathcal T^{(i)}}^\phi)^\top z_k$ and $a_k(\phi)\Let(\nabla_\phi\hat c_{\mathcal T^{(i)}}^\phi)^\top z_k$ for $k\in[K^{(i)}]$. Then $g^{(i)}(\phi)=\min_{k\in[K^{(i)}]}F_k(\phi)$ and $\mathcal Z_{\mathcal T^{(i)}}^\phi=\{z_k:F_k(\phi)=g^{(i)}(\phi)\}$. For later use, define $M(\phi)\Let\{k\in[K^{(i)}]:F_k(\phi)=g^{(i)}(\phi)\}$ and $S(\phi)\Let\operatorname{conv}\{a_k(\phi):k\in M(\phi)\}$.

We first prove the regularity of the per-sample objective and of the candidate field.

\begin{proposition}[Definability and local Lipschitz continuity of $J^{(i)}$]\label{prop:locally-Lipschitz-Ji}
Under Assumptions~\ref{assmpt:smoothness-bounded}--\ref{assmpt:definable}, the function $J^{(i)}$ defined in~\eqref{eq:Ji-def} is definable and locally Lipschitz for any $i\in[N]$.
\end{proposition}

\begin{proof}[Proof of Proposition~\ref{prop:locally-Lipschitz-Ji}]
Under Assumption~\ref{assmpt:definable}, each branch $F_k$ is definable, and hence $g^{(i)}=\min_{k\in[K^{(i)}]}F_k$ is definable. The loss and regularizer are definable by Assumption~\ref{assmpt:definable}; hence, the per-sample objective $J^{(i)}$ is definable.

It remains to verify local Lipschitz continuity. Fix a compact convex neighborhood $U\subset\R^p$. By Assumption~\ref{assmpt:smoothness-bounded}(b)--(c), each branch $F_k$ is continuously differentiable on a neighborhood of $U$. Hence $\nabla F_k$ is continuous on $U$ and is therefore bounded on $U$. Because $K^{(i)}<+\infty$, there is a common constant $L_F<+\infty$ such that
\[
\sup_{\phi\in U}\|\nabla F_k(\phi)\|_2\le L_F
\qquad \forall k\in[K^{(i)}].
\]
By the mean-value theorem on the convex set $U$,
\[
|F_k(\phi)-F_k(\phi')|\le L_F\|\phi-\phi'\|_2
\qquad \forall \phi,\phi'\in U,\ \forall k\in[K^{(i)}].
\]
For finite minima,
\[
\bigl|\min_k F_k(\phi)-\min_k F_k(\phi')\bigr|
\le \max_k |F_k(\phi)-F_k(\phi')|,
\]
so $g^{(i)}$ is Lipschitz on $U$. Since $g^{(i)}(U)$ is compact and $x\mapsto \ell(x,t^{(i)})$ is $C^1$, the composition $\phi\mapsto \ell(g^{(i)}(\phi),t^{(i)})$ is Lipschitz on $U$. Finally, the regularizer $\mathcal R$ defined in~\eqref{eq:regularization} is a finite sum of squared norms of differences of the edge embeddings, which is $C^1$ in $\vartheta$ because it is a finite sum of $C^1$ compositions. Therefore $\phi\mapsto\mathcal R(\vartheta)$ is $C^1$ on a neighborhood of $U$, and hence Lipschitz on $U$. Thus $J^{(i)}$ is Lipschitz on every compact neighborhood $U$, and hence locally Lipschitz.
\end{proof}

\begin{proposition}[Definability, closedness, and local boundedness of $\operatorname{CF}^{(i)}$]\label{prop:locally-bounded-CF}
Under Assumptions~\ref{assmpt:smoothness-bounded}--\ref{assmpt:definable}, the set-valued mapping $\operatorname{CF}^{(i)}$ defined in~\eqref{eq:CF} has nonempty compact convex values, is locally bounded, and has a definable and closed graph for any sample $i\in[N]$.
\end{proposition}

\begin{proof}[Proof of Proposition~\ref{prop:locally-bounded-CF}]
Recall the definitions of $M(\phi)$ and $S(\phi)$ above. Define
\[
s(\phi)\Let\nabla_1\ell(g^{(i)}(\phi),t^{(i)}),
\qquad
b(\phi)\Let\beta\nabla_\phi\mathcal R(\vartheta),
\]
where $\vartheta$ denotes the $\vartheta$-component of $\phi$. Then
$\operatorname{CF}^{(i)}(\phi)=s(\phi)S(\phi)+b(\phi)$.

The active set $M(\phi)$ is nonempty because $K^{(i)}<+\infty$, and hence $S(\phi)$ is the convex hull of finitely many vectors. Therefore $S(\phi)$, and consequently $\operatorname{CF}^{(i)}(\phi)$, has nonempty compact convex values.

Because $\phi\mapsto \hat c_{\mathcal T^{(i)}}^\phi$, $x\mapsto \ell(x,t^{(i)})$, and $\mathcal R$ are definable and $C^1$, their derivatives are definable; hence each $a_k$, as well as $s$ and $b$, is definable. We next prove graph definability. Since the functions $F_k$, $g^{(i)}$, $a_k$, $s$, and $b$ are definable, the graph of $\operatorname{CF}^{(i)}$ admits the explicit representation
\[
\operatorname{graph}\operatorname{CF}^{(i)}
= \bigl\{
(\phi,d):
\begin{array}{l}
\exists \lambda\in\Delta_{K^{(i)}} \text{ s.t. }
d=s(\phi)\displaystyle\sum_{k=1}^{K^{(i)}}\lambda_k a_k(\phi)+b(\phi),\,
\lambda_k\bigl(F_k(\phi)-g^{(i)}(\phi)\bigr)=0
\quad \forall k\in[K^{(i)}]
\end{array}
\bigr\},
\]
where 
$\Delta_{K^{(i)}} = 
\bigl\{\lambda\in\R_+^{K^{(i)}}:\sum_{k=1}^{K^{(i)}}\lambda_k=1\bigr\}$ is a simplex. 
This is the projection of a definable set in the variables $(\phi,d,\lambda)$, so $\operatorname{graph}\operatorname{CF}^{(i)}$ is definable.

We now prove closedness. Suppose $\phi_n\to\phi$, $d_n\to d$, and $d_n\in\operatorname{CF}^{(i)}(\phi_n)$. Then $d_n=s(\phi_n)v_n+b(\phi_n)$ for some $v_n\in S(\phi_n)$. Choose $\lambda_n\in\Delta_{K^{(i)}}$ such that
\[
v_n=\sum_{k=1}^{K^{(i)}}\lambda_{n,k}a_k(\phi_n),
\qquad
\lambda_{n,k}=0\quad\text{whenever }k\notin M(\phi_n).
\]
Passing to a subsequence if necessary, $\lambda_n\to\lambda\in\Delta_{K^{(i)}}$. Assumption~\ref{assmpt:smoothness-bounded}(b) implies that the mapping $\phi\mapsto\nabla_\phi \hat c_{\mathcal T^{(i)}}^\phi$ is continuous, and hence $a_k(\phi_n)\to a_k(\phi)$ for every $k$. Thus
\[
v_n\to v\Let\sum_{k=1}^{K^{(i)}}\lambda_k a_k(\phi).
\]
If $\lambda_k>0$, then $\lambda_{n,k}>0$ for all sufficiently large $n$, so $k\in M(\phi_n)$ for all sufficiently large $n$. Since each $F_k$ and $g^{(i)}$ is continuous, $F_k(\phi_n)=g^{(i)}(\phi_n)$ eventually implies $F_k(\phi)=g^{(i)}(\phi)$, so $k\in M(\phi)$. Hence $v\in S(\phi)$. The functions $s$ and $b$ are continuous, so 
$d=\lim_n d_n=s(\phi)v+b(\phi)\in\operatorname{CF}^{(i)}(\phi)$. 
The graph is closed.

Finally, local boundedness follows from the same finite representation. On any compact set $U$, the finitely many continuous maps $a_k$, $s$, and $b$ are bounded. Hence $S(\phi)$ and $s(\phi)S(\phi)+b(\phi)=\operatorname{CF}^{(i)}(\phi)$ are uniformly bounded over $\phi\in U$.
\end{proof}

We now verify that $\operatorname{CF}^{(i)}$ is conservative.

\begin{theorem}[$\operatorname{CF}^{(i)}$ is conservative]\label{thm:conservative-field}
The mapping $\operatorname{CF}^{(i)}$ defined in~\eqref{eq:CF} is a conservative field for $J^{(i)}$. In particular, $J^{(i)}$ admits a conservative field.
\end{theorem}

\begin{proof}[Proof of Theorem~\ref{thm:conservative-field}]
Let $\chi:[0,1]\to\R^p$ be any absolutely continuous curve. Define
\[
F_{k,\chi}(\sigma)\Let F_k(\chi(\sigma)) \qquad \text{for each } k\in[K^{(i)}].
\]
Each $F_{k,\chi}$ is absolutely continuous. Since the finite minimum of absolutely continuous functions is absolutely continuous, the function
\[
g_\chi(\sigma)\Let g^{(i)}(\chi(\sigma))
=\min_{k\in[K^{(i)}]}F_{k,\chi}(\sigma)
\]
is also absolutely continuous.
For almost every $\sigma\in(0,1)$, the derivatives of $\chi$, all $F_{k,\chi}$, and $g_\chi$ exist. Fix such a $\sigma$. We claim that the one-dimensional finite-minimum formula gives
\begin{equation}
\frac{\dd^+}{\dd\sigma}g_\chi(\sigma)
= \min_{k\in M(\chi(\sigma))} \frac{\dd}{\dd\sigma}F_{k,\chi}(\sigma),
\qquad
\frac{\dd^-}{\dd\sigma}g_\chi(\sigma)
= \max_{k\in M(\chi(\sigma))} \frac{\dd}{\dd\sigma}F_{k,\chi}(\sigma).
\label{eq:derivative-g-chi}
\end{equation}
Indeed, if $j\notin M(\chi(\sigma))$, then $F_{j,\chi}(\sigma)>g_\chi(\sigma)$.
By continuity and finiteness, all inactive branches remain strictly above the active branches for sufficiently small one-sided perturbations of $\sigma$. Thus only indices in $M(\chi(\sigma))$ can attain the minimum locally. For each $k\in M(\chi(\sigma))$, differentiability at $\sigma$ gives
\[
F_{k,\chi}(\sigma+q)
= g_\chi(\sigma) + q\frac{\dd}{\dd\sigma}F_{k,\chi}(\sigma) + r_k(q),
\qquad \frac{r_k(q)}{|q|}\to 0
\quad\text{as }q\to 0.
\]
Because $M(\chi(\sigma))$ is finite, the remainder can be made uniform: let $r(q)\Let \max_{k\in M(\chi(\sigma))}|r_k(q)|$, then $r(q)=o(|q|)$ as $q\to0$. Since only indices in $M(\chi(\sigma))$ can attain the minimum for all sufficiently small one-sided perturbations of $\sigma$, we obtain, for all sufficiently small $q>0$,
\[
q\min_{k\in M(\chi(\sigma))}\frac{\dd}{\dd\sigma}F_{k,\chi}(\sigma)-r(q)
\le g_\chi(\sigma+q)-g_\chi(\sigma)
\le q\min_{k\in M(\chi(\sigma))}\frac{\dd}{\dd\sigma}F_{k,\chi}(\sigma)+r(q),
\]
and, for all sufficiently small $q<0$,
\[
q\max_{k\in M(\chi(\sigma))}\frac{\dd}{\dd\sigma}F_{k,\chi}(\sigma)-r(q)
\le g_\chi(\sigma+q)-g_\chi(\sigma)
\le q\max_{k\in M(\chi(\sigma))}\frac{\dd}{\dd\sigma}F_{k,\chi}(\sigma)+r(q).
\]
Dividing by $q$ and letting $q\downarrow0$ or $q\uparrow0$ yields~\eqref{eq:derivative-g-chi}. Whenever $g_\chi$ is differentiable at $\sigma$, the left and right derivatives coincide. Therefore,
\[
\frac{\dd}{\dd\sigma}F_{k,\chi}(\sigma)
= \frac{\dd}{\dd\sigma}(g^{(i)}\circ\chi)(\sigma)
\qquad \forall k\in M(\chi(\sigma)).
\]
Since $F_k$ is $C^1$, $\frac{\dd}{\dd\sigma}F_{k,\chi}(\sigma) = \left\langle a_k(\chi(\sigma)),\chi'(\sigma)\right\rangle$.
Therefore, for every $v\in S(\chi(\sigma))$,
\[
\langle v,\chi'(\sigma)\rangle
= \frac{\dd}{\dd\sigma}(g^{(i)}\circ\chi)(\sigma) \qquad\text{for a.e. }\sigma\in[0,1].
\]

Now take any $d\in\operatorname{CF}^{(i)}(\chi(\sigma))$. By definition, there exists $v\in S(\chi(\sigma))$ such that
\[
d = \nabla_1\ell(g^{(i)}(\chi(\sigma)),t^{(i)})v + \beta\nabla_\phi\mathcal R(\vartheta_\chi(\sigma)),
\]
where $\vartheta_\chi(\sigma)$ denotes the $\vartheta$-component of $\chi(\sigma)$. For almost every $\sigma\in[0,1]$,
\begin{align*}
\langle d,\chi'(\sigma)\rangle
&= \nabla_1\ell(g^{(i)}(\chi(\sigma)),t^{(i)}) \frac{\dd}{\dd\sigma}(g^{(i)}\circ\chi)(\sigma)
+ \beta \left\langle \nabla_\phi\mathcal R(\vartheta_\chi(\sigma)),\chi'(\sigma) \right\rangle\\
&= \frac{\dd}{\dd\sigma} \bigl[ \ell(g^{(i)}(\chi(\sigma)),t^{(i)}) + \beta\mathcal R(\vartheta_\chi(\sigma)) \bigr]
= \frac{\dd}{\dd\sigma}(J^{(i)}\circ\chi)(\sigma),
\end{align*}
where the second equality uses the ordinary chain rule for the $C^1$ loss and the $C^1$ regularizer.

Proposition~\ref{prop:locally-Lipschitz-Ji} shows that $J^{(i)}$ is locally Lipschitz, and Proposition~\ref{prop:locally-bounded-CF} shows that $\operatorname{CF}^{(i)}$ has nonempty compact values, is locally bounded, and has a closed graph. The chain-rule characterization of conservative fields in \cite[Lemma~2]{ref:bolte2021conservative} therefore implies that $\operatorname{CF}^{(i)}$ is a conservative field for $J^{(i)}$.
\end{proof}

We now complete the convergence proof.

\begin{proof}[Proof of Theorem~\ref{thm:convergence}]
Set $r_k=0$ for every iteration $k$ in Algorithm~\ref{alg:conservative-sgd}, so that $\tilde\phi_k=\phi_k$ for every iteration $k$.
For each sample index $i\in[N]$, let 
\[
z_k^{(i)}\in\argmin_{z\in\mathcal Z^{(i)}} (\hat c_{\mathcal T^{(i)}}^{\phi_k})^\top z
= \mathcal Z_{\mathcal T^{(i)}}^{\phi_k}
\]
be the shortest path selected by the deterministic order $\prec$ when ties occur.
Since $z_k^{(i)}$ is optimal, $(\hat c_{\mathcal T^{(i)}}^{\phi_k})^\top z_k^{(i)}=g^{(i)}(\phi_k)$. Define the conservative-field direction with the selected shortest path by
\[
\bar d_k^{(i)} 
\Let \nabla_1\ell\bigl(g^{(i)}(\phi_k),t^{(i)}\bigr)\,\bigl(\nabla_\phi \hat c_{\mathcal T^{(i)}}^{\phi_k}\bigr)^\top z_k^{(i)} + \beta \nabla_\phi \mathcal R(\vartheta_k).
\]
Then $\bar d_k^{(i)}\in \operatorname{CF}^{(i)}(\phi_k)$ for all $i\in[N]$,
because $z_k^{(i)}\in\mathcal Z_{\mathcal T^{(i)}}^{\phi_k}$ and $\operatorname{CF}^{(i)}(\phi_k)$ is defined via the convex hull over all minimizers; thus, deterministic tie-breaking does not affect this inclusion.
For every sampled index $i\in B_k$, the per-sample direction used by the algorithm is exactly $d_k^{(i)}=\bar d_k^{(i)}$.
Hence averaging within the mini-batch yields
\[
d_k\Let \frac{1}{|B_k|}\sum_{i\in B_k} d_k^{(i)}
= \frac{1}{|B_k|}\sum_{i\in B_k} \bar d_k^{(i)}
\in \frac{1}{|B_k|}\sum_{i\in B_k}\operatorname{CF}^{(i)}(\phi_k).
\]
Here $\sum_{i\in B_k}\operatorname{CF}^{(i)}(\phi_k)$ denotes the Minkowski sum of the sets $\{\operatorname{CF}^{(i)}(\phi_k)\}_{i\in B_k}$.

Define the aggregate conservative-field mapping for the empirical objective by
\[
D_J(\phi)\Let \frac{1}{N}\operatorname{conv}\bigl(\sum_{i=1}^N \operatorname{CF}^{(i)}(\phi)\bigr).
\]
By Theorem~\ref{thm:conservative-field}, for each sample $i\in[N]$ the mapping $\operatorname{CF}^{(i)}$ is a conservative field for $J^{(i)}$. Proposition~\ref{prop:locally-Lipschitz-Ji} shows that $J^{(i)}$ is definable and locally Lipschitz. Proposition~\ref{prop:locally-bounded-CF} further shows that $\operatorname{CF}^{(i)}$ has nonempty compact convex values, is locally bounded, and has a definable and closed graph. Thus the structural assumptions of \cite[Theorem~9]{ref:bolte2021conservative} are satisfied with $f_i=J^{(i)}$ and $D_i=\operatorname{CF}^{(i)}$. It remains to verify that the stochastic recursion has conditional mean drift in the aggregate field $D_J$.

Assume that the mini-batches $\{B_k\}$ are drawn independently and uniformly at random among all subsets of $[N]$ of the prescribed batch size, and that the step sizes satisfy
$\sum_{k=0}^{+\infty}\alpha_k = +\infty$ and $\lim_{k\to+\infty}\alpha_k\log(k+2)=0$.
Let $\mathcal F_k$ denote the sigma-field generated by the algorithmic randomness up to iteration $k-1$. Since $\phi_k$ is $\mathcal F_k$-measurable and the tie-breaking rule $\prec$ is deterministic, each $\bar d_k^{(i)}$ is $\mathcal F_k$-measurable. Moreover, by uniform mini-batch sampling,
\[
\mathbb E[d_k\mid \mathcal F_k]
= \frac{1}{N}\sum_{i=1}^N \bar d_k^{(i)}
\in \frac{1}{N}\sum_{i=1}^N \operatorname{CF}^{(i)}(\phi_k)
\subseteq D_J(\phi_k).
\]
Therefore the recursion $\phi_{k+1}=\phi_k-\alpha_k d_k$ is an instance of the mini-batch conservative-field stochastic approximation recursion considered in \cite[Theorem~9]{ref:bolte2021conservative}.
Let $\mathcal B\Let\{\sup_k \|\phi_k\|_2< +\infty\}$ be the event that the iterates remain bounded.
On $\mathcal B$, the iterates lie in a compact set on which Assumptions~\ref{assmpt:smoothness-bounded}--\ref{assmpt:definable} hold and the regularity properties
required by \cite[Theorem~9]{ref:bolte2021conservative} are satisfied.
Therefore, \cite[Theorem~9]{ref:bolte2021conservative} implies that, almost surely on $\mathcal B$, the sequence $\{\phi_k\}$ has a nonempty accumulation set, every accumulation point $\bar\phi$ satisfies $0\in D_J(\bar\phi)$,
and $J$ is constant on this accumulation set.
This completes the proof.
\end{proof}

\subsection{Finite-Time $(\delta,\varepsilon)$-Stationarity via Parameter-Space Randomized Smoothing}
\label{ec:finite-time-proof}

Theorem~\ref{thm:convergence} treats the unperturbed recursion ($r_k\equiv 0$) and yields the asymptotic stationarity condition $0\in D_J(\bar\phi)$ for every accumulation point $\bar\phi$. We now add a complementary finite-time guarantee at a prescribed spatial scale for the perturbed recursion with a deterministic nonincreasing radius schedule $\{r_k\}_{k\in\mathbb N}$. Let 
$\mathbb{B} \Let \{u\in\R^p: \|u\|_2 \leq 1\}$.
We set $\delta \Let 2r_0$ and work with the parameter-smoothed objectives
\begin{equation}
\hat J_r(\phi) \Let \mathbb{E}_{u\sim \mathrm{Unif}(\mathbb{B})}\big[J(\phi+r u)\big].
\label{eq:Jhat}
\end{equation}
The perturbation acts only in parameter space, so the lower-level shortest-path problems are still solved exactly (cf.\ Remark~\ref{rem:smoothing}).

We first recall the definition of $(\delta,\varepsilon)$-stationarity~\cite[Definitions~4--5]{ref:cutkosky2023optimal}.
\begin{definition}[$(\delta,\varepsilon)$-stationarity at scale $\delta$]
\label{def:goldstein}
Let $F:\R^p\to\R$ be almost everywhere differentiable, and define the differentiability domain of $F$ by
\[
\mathcal D_F \Let \operatorname{dom}(\nabla F)=\{x\in\R^p:\nabla F(x)\ \text{exists}\}.
\]
For $\delta>0$, define the stationarity measure
\[
\|\nabla F(x)\|_{\delta}
\Let
\inf\Bigl\{
\bigl\|\frac{1}{|S|}\sum_{y\in S}\nabla F(y)\bigr\|_2:
S\subseteq \mathbb{B}(x,\delta)\cap \mathcal D_F,\; 0<|S|<+\infty,\; x=\frac{1}{|S|}\sum_{y\in S}y
\Bigr\},
\]
where $\mathbb{B}(x,\delta)\Let x+\delta\mathbb B$ is the closed Euclidean ball of radius $\delta$ centered at $x$.
We say that $x$ is $(\delta,\varepsilon)$-stationary for $F$ if $\|\nabla F(x)\|_{\delta}\le \varepsilon$.
\end{definition}

\begin{assumption}[Localized assumptions for randomized smoothing]
\label{ass:query-region}
Fix a deterministic nonincreasing sequence $\{r_k\}_{k\in\mathbb N}\subset(0,+\infty)$ and set $\delta = 2r_0$. Assume the following.
\begin{enumerate}[label=(\alph*),leftmargin=7mm]
\item \textbf{Compact convex iterate region.}
There exists a deterministic compact convex set $\mathcal U\subset\R^p$ such that $\phi_k\in \mathcal U$ almost surely for all $k$.

\item \textbf{Lipschitz objective on the expanded region.}
The objective $J$ is $L_J$-Lipschitz on the expanded region
$V = \mathcal U + 2\delta \mathbb B = \mathcal U + 4r_0\mathbb B$.

\item \textbf{Uniform conditional second-moment bound.}
The directions $d_k$ produced by Algorithm~\ref{alg:conservative-sgd} satisfy $\mathbb{E}\big[\|d_k\|_2^2 \,\big|\, \phi_k\big] \le D$ almost surely for all $k$.

\item \textbf{Lower bounded smoothed objectives.}
For every $k\in\mathbb N$, the smoothed objective $\hat J_{r_k}$ is bounded below on $\R^p$, i.e., 
$\inf_{\phi\in\R^p}\hat J_{r_k}(\phi)>-\infty$.

\item \textbf{Conditional unbiasedness of perturbation and mini-batching.}
For every $k$:
\begin{enumerate}[label=(\roman*),leftmargin=6mm]
\item conditional on $\phi_k$, the perturbation $u_k$ is distributed uniformly on $\mathbb B$;
\item the mini-batch $B_k$ is almost surely nonempty;
\item for every deterministic collection of vectors $\{v_i\}_{i=1}^N\subset\R^p$,
\[
\mathbb{E} \bigl[\frac1{|B_k|}\sum_{i\in B_k}v_i \,\mid\, \phi_k,u_k\bigr]
= \frac1N\sum_{i=1}^N v_i
\qquad\text{almost surely.}
\]
\end{enumerate}

\end{enumerate}
\end{assumption}

Fix a deterministic tie-breaking order $\prec$.
For each sample $i\in\{1,\dots,N\}$ and parameter vector $\phi\in\R^p$, let $z^{(i)}_{\prec}(\phi)$ denote the $\prec$-selected minimizer of the shortest-path problem in \eqref{eq:gi-def} under costs $\hat c_{\mathcal T^{(i)}}^\phi$. Define the associated backpropagation direction
\[
G^{(i)}_{\prec}(\phi)
\Let \nabla_1\ell\bigl((\hat c_{\mathcal T^{(i)}}^\phi)^\top z^{(i)}_{\prec}(\phi),\,t^{(i)}\bigr)\, \bigl(\nabla_\phi \hat c_{\mathcal T^{(i)}}^\phi\bigr)^\top z^{(i)}_{\prec}(\phi) +\beta \nabla_\phi \mathcal R(\vartheta).
\]

\begin{lemma}[Tie-breaking oracle agrees with gradients almost everywhere]
\label{lem:oracle-ae}
Assume Assumptions~\ref{assmpt:smoothness-bounded}--\ref{assmpt:definable}, and fix a deterministic tie-breaking order $\prec$.
Then there exists a full-measure set $\mathcal D_J\subseteq\R^p$ such that, for every $\phi\in\mathcal D_J$ and every $i\in\{1,\dots,N\}$,
$G^{(i)}_{\prec}(\phi)=\nabla J^{(i)}(\phi)$.
Consequently,
$\frac1N\sum_{i=1}^N G^{(i)}_{\prec}(\phi)=\nabla J(\phi)$ for all $\phi\in\mathcal D_J$.
\end{lemma}

\begin{proof}[Proof of Lemma~\ref{lem:oracle-ae}]
Fix $i\in\{1,\dots,N\}$. By Theorem~\ref{thm:conservative-field}, the set-valued mapping $\operatorname{CF}^{(i)}$ is a conservative field for the locally Lipschitz function $J^{(i)}$. Therefore, \cite[Theorem~1]{ref:bolte2021conservative} yields a full-measure set $\mathcal D^{(i)}\subseteq\R^p$ such that $\operatorname{CF}^{(i)}(\phi)=\{\nabla J^{(i)}(\phi)\}$ for all $\phi\in\mathcal D^{(i)}$.

For any $\phi\in\R^p$, the $\prec$-selected path $z^{(i)}_{\prec}(\phi)$ is an element of the optimal-path set $\mathcal Z_{\mathcal T^{(i)}}^\phi$.
By construction, 
$(\hat c_{\mathcal T^{(i)}}^\phi)^\top z^{(i)}_{\prec}(\phi)=g^{(i)}(\phi)$,
so the scalar factor in $G^{(i)}_{\prec}(\phi)$ is exactly the scalar factor appearing in the definition \eqref{eq:CF} of $\operatorname{CF}^{(i)}(\phi)$.
Hence
$G^{(i)}_{\prec}(\phi)\in \operatorname{CF}^{(i)}(\phi)$ for all $\phi\in\R^p$.
It follows that
$G^{(i)}_{\prec}(\phi)=\nabla J^{(i)}(\phi)$ for all $\phi\in\mathcal D^{(i)}$.
Let $\mathcal D_J = \bigcap_{i=1}^N \mathcal D^{(i)}$.
Because $N<+\infty$, the set $\mathcal D_J$ still has full Lebesgue measure.
For every $\phi\in\mathcal D_J$, the identity above holds for all $i$, and therefore
\[
\frac1N\sum_{i=1}^N G^{(i)}_{\prec}(\phi)
= \frac1N\sum_{i=1}^N \nabla J^{(i)}(\phi)
= \nabla J(\phi).
\]
This proves the claim.
\end{proof}

\begin{lemma}[Randomized-smoothing oracle identity]
\label{lem:smoothing-oracle}
Under Assumptions~\ref{assmpt:smoothness-bounded}--\ref{assmpt:definable} and Assumption~\ref{ass:query-region}(e), the smoothed objective $\hat J_r$ defined in \eqref{eq:Jhat} is differentiable on $\R^p$ for every $r>0$.
Moreover, if $\mathcal I$ is uniformly distributed on $\{1,\dots,N\}$ and $g_\prec(\phi, \mathcal I)\Let G^{(\mathcal I)}_{\prec}(\phi)$,
then
\[
\nabla \hat J_r(\phi)
= \mathbb E_{u, \mathcal I}\big[g_\prec(\phi+r u, \mathcal I)\big]
\qquad\forall \phi\in\R^p.
\]
In particular, Algorithm~\ref{alg:conservative-sgd} with perturbation radii $\{r_k\}$ satisfies $\mathbb E[d_k\mid \phi_k]=\nabla \hat J_{r_k}(\phi_k)$ almost surely for all $k$.
\end{lemma}

\begin{proof}[Proof of Lemma~\ref{lem:smoothing-oracle}]
By Proposition~\ref{prop:locally-Lipschitz-Ji}, each $J^{(i)}$ is locally Lipschitz. Hence
$J(\phi)=\frac1N\sum_{i=1}^N J^{(i)}(\phi)$
is locally Lipschitz.

Fix $\phi\in\mathcal D_J$, where $\mathcal D_J$ is the full-measure set from Lemma~\ref{lem:oracle-ae}.
Then
\[
\mathbb E_{\mathcal I}\big[g_\prec(\phi,\mathcal I)\big]
= \frac1N\sum_{i=1}^N G^{(i)}_{\prec}(\phi).
\]
By Lemma~\ref{lem:oracle-ae}, the right-hand side equals $\nabla J(\phi)$ for every $\phi\in\mathcal D_J$. Therefore
$\mathbb E_{\mathcal I}\big[g_\prec(\phi,\mathcal I)\big]=\nabla J(\phi)$ for all $\phi\in\mathcal D_J$.

Applying \cite[Proposition~2]{ref:cutkosky2023optimal} to the locally Lipschitz function $J$ and the stochastic oracle $g_\prec$ yields that $\hat J_r$ is differentiable on $\R^p$ and
$\nabla \hat J_r(\phi)
= \mathbb E_{u,\mathcal I}\big[g_\prec(\phi+r u,\mathcal I)\big]$ for all $\phi\in\R^p$.
By construction of Algorithm~\ref{alg:conservative-sgd},
$d_k=\frac1{|B_k|}\sum_{i\in B_k} G^{(i)}_{\prec}(\phi_k+r_k u_k)$.
Condition on $(\phi_k,u_k)$. The vectors
\[
v_i \Let G^{(i)}_{\prec}(\phi_k+r_k u_k),\qquad i=1,\dots,N
\]
are then deterministic, so Assumption~\ref{ass:query-region}(e)(iii) implies
\[
\mathbb E[d_k\mid \phi_k,u_k]
= \frac1N\sum_{i=1}^N G^{(i)}_{\prec}(\phi_k+r_k u_k)
= \mathbb E_{\mathcal I}\big[g_\prec(\phi_k+r_k u_k,\mathcal I)\mid \phi_k,u_k\big].
\]
Taking conditional expectation over $u_k$ given $\phi_k$ and using Assumption~\ref{ass:query-region}(e)(i) gives
\[
\mathbb E[d_k\mid \phi_k]
= \mathbb E_{u_k,\mathcal I}\big[g_\prec(\phi_k+r_k u_k,\mathcal I)\mid \phi_k\big]
= \nabla \hat J_{r_k}(\phi_k)
\qquad \text{almost surely } \forall k.
\qquad
\]
\end{proof}

\begin{lemma}[Localized smoothness of the smoothed objective]
\label{lem:smoothing-smoothness}
Under Assumptions \ref{ass:query-region}(a)--(b), let $r\in(0,r_0]$.
Then $\hat J_r$ is continuously differentiable on an open neighborhood of $\mathcal U$, and there exists a universal constant $C_{\mathrm{sm}}>0$ such that
\[
\|\nabla \hat J_r(\phi)-\nabla \hat J_r(\phi')\|_2
\le C_{\mathrm{sm}}\frac{L_J\sqrt p}{r}\,\|\phi-\phi'\|_2
\qquad \forall \phi,\phi'\in\mathcal U.
\]
In particular, any Lipschitz constant $L_r$ of $\nabla \hat J_r$ on $\mathcal U$ may be chosen so that
$L_r \le C_{\mathrm{sm}}L_J\sqrt p / r$.
\end{lemma}

\begin{proof}[Proof of Lemma~\ref{lem:smoothing-smoothness}]
By Assumption~\ref{ass:query-region}(b), the restriction of $J$ to $V$ is $L_J$-Lipschitz.
Define the McShane extension~(\cite{ref:mcshane1934extension}; \cite[Definition~2.1]{ref:petrakis2020mcshane})
\[
\bar J(x)
\Let \inf_{y\in V}\bigl\{J(y)+L_J\|x-y\|_2\bigr\},
\qquad x\in\R^p.
\]
Then $\bar J$ is globally $L_J$-Lipschitz on $\R^p$ and satisfies $\bar J = J$ on $V$~\cite[Theorem~2.3]{ref:petrakis2020mcshane}.

Let $W = \mathcal U + \frac{\delta}{2}\operatorname{int}\mathbb B$.
Then $W$ is an open neighborhood of $\mathcal U$.
Fix $\phi\in W$. Choose $\phi_0\in\mathcal U$ and $b\in \frac{\delta}{2}\operatorname{int}\mathbb B$ such that $\phi=\phi_0+b$.
For any $u\in\mathbb B$, since $r\le r_0=\delta/2$, we have
$\|\phi+r u-\phi_0\|_2
\le \|b\|_2+r
< \frac{\delta}{2}+\frac{\delta}{2}
= \delta$.
Hence
$\phi+r u \in \mathcal U+\delta\,\operatorname{int}\mathbb B \subseteq V$.
Therefore
\[
\hat J_r(\phi)
= \mathbb E_u[J(\phi+r u)]
= \mathbb E_u[\bar J(\phi+r u)]
= \widehat{\bar J}_r(\phi)
\qquad \forall \phi\in W,
\]
where $\widehat{\bar J}_r$ is the radius-$r$ smoothing of the McShane extension defined by
\[
\widehat{\bar J}_r(\phi)\Let \mathbb E_u[\bar J(\phi+r u)].
\]
Thus $\hat J_r$ and $\widehat{\bar J}_r$ coincide on the open neighborhood $W$ of $\mathcal U$.

Because $\bar J$ is globally $L_J$-Lipschitz, \cite[Proposition~2.2]{ref:lin2022gradient} implies that $\widehat{\bar J}_r$ is differentiable on $\R^p$ and its gradient is ($C_{\mathrm{sm}}L_J\sqrt p/r$)-Lipschitz for some universal constant $C_{\mathrm{sm}}>0$. In particular, $\widehat{\bar J}_r \in C^1(\R^p)$.
Since $\hat J_r=\widehat{\bar J}_r$ on the open set $W$, the function $\hat J_r$ is continuously differentiable on $W$ and
$\nabla \hat J_r(\phi)=\nabla \widehat{\bar J}_r(\phi)$ for all $\phi\in W$.
In particular, for all $\phi,\phi'\in\mathcal U$,
\[
\|\nabla \hat J_r(\phi)-\nabla \hat J_r(\phi')\|_2
=
\|\nabla \widehat{\bar J}_r(\phi)-\nabla \widehat{\bar J}_r(\phi')\|_2
\le C_{\mathrm{sm}}\frac{L_J\sqrt p}{r}\,\|\phi-\phi'\|_2.
\]
This proves the claim.
\end{proof}

\begin{lemma}[Cross-radius drift of the smoothed objectives]
\label{lem:smoothing-drift}
Suppose Assumption~\ref{ass:query-region}(a)--(b) holds.
Then, for every $\phi\in\mathcal U$ and every $r,s\in[0,r_0]$,
$\big|\hat J_r(\phi)-\hat J_s(\phi)\big| \le L_J|r-s|$.
\end{lemma}

\begin{proof}[Proof of Lemma~\ref{lem:smoothing-drift}]
Fix $\phi\in\mathcal U$ and $r,s\in[0,r_0]$.
For any $u\in\mathbb B$, both $\phi+r u$ and $\phi+s u$ belong to $\mathcal U+r_0\mathbb B\subseteq V$.
Hence Assumption~\ref{ass:query-region}(b) implies
$\big|J(\phi+r u)-J(\phi+s u)\big|
\le L_J\|(r-s)u\|_2
\le L_J|r-s|$.
Taking expectations over $u$ gives the result.
\end{proof}

For each $k$, let $L_k$ denote any Lipschitz constant of $\nabla \hat J_{r_k}$ on the compact query set $\mathcal U$.
By Lemma~\ref{lem:smoothing-smoothness}, we may choose $L_k \le C_{\mathrm{sm}}L_J\sqrt p / r_k$ for all $k$.

\begin{lemma}[Stationarity transfer from $\hat J_r$ to $J$]
\label{lem:transfer}
Assume Assumption~\ref{ass:query-region}(a)--(b) and let $r\in(0,r_0]$.
Then, for every $\phi\in\mathcal U$,
$\|\nabla J(\phi)\|_{2r}
\le \|\nabla \hat J_r(\phi)\|_{r}$.
\end{lemma}

\begin{proof}[Proof of Lemma~\ref{lem:transfer}]
Let $\bar J:\R^p\to\R$ be the McShane extension
\[
\bar J(x)
\Let \inf_{y\in V}\bigl\{J(y)+L_J\|x-y\|_2\bigr\}.
\]
Then $\bar J$ is globally $L_J$-Lipschitz and $\bar J=J$ on $V$.
Fix $\phi\in\mathcal U$. Let $y\in\mathbb B(\phi,2r)$.
Then the open ball $y+r \operatorname{int}\mathbb B$ satisfies
$y+r \operatorname{int}\mathbb B \subseteq \mathbb B(\phi,3r) \subseteq \mathbb B(\phi,3r_0) \subseteq V$,
because $r\le r_0$ and $3r_0<2\delta$.
Hence $J$ and $\bar J$ coincide on an open neighborhood of $y$.
Therefore $y\in\mathcal D_J$ if and only if $y\in\mathcal D_{\bar J}$, and whenever these gradients exist, $\nabla J(y)=\nabla \bar J(y)$.
Since this holds for every $y\in\mathbb B(\phi,2r)$, Definition~\ref{def:goldstein} yields
$\|\nabla J(\phi)\|_{2r}=\|\nabla \bar J(\phi)\|_{2r}$.

Define
$\widehat{\bar J}_r(z)\Let \mathbb E_u[\bar J(z+r u)]$.
Let $z\in\mathbb B(\phi,r)$.
If $w\in z+r\,\operatorname{int}\mathbb B$ and $u\in\mathbb B$, then
\[
\|w+r u-\phi\|_2
\le \|w-z\|_2+\|z-\phi\|_2+r
< r+r+r
= 3r
\le 3r_0
< 2\delta.
\]
Hence
$w+r u \in \mathcal U+2\delta\mathbb B = V$.
Therefore $J$ and $\bar J$ agree on $w+r\mathbb B$, so
$\hat J_r(w)=\mathbb E_u[J(w+r u)]
= \mathbb E_u[\bar J(w+r u)]
= \widehat{\bar J}_r(w)$.
Thus $\hat J_r$ and $\widehat{\bar J}_r$ coincide on an open neighborhood of every point
$z\in\mathbb B(\phi,r)$. Consequently,
$\|\nabla \hat J_r(\phi)\|_{r}=\|\nabla \widehat{\bar J}_r(\phi)\|_{r}$.

Let $a = \|\nabla \widehat{\bar J}_r(\phi)\|_{r}$.
For any $\eta>0$, we have $\|\nabla \widehat{\bar J}_r(\phi)\|_{r}\le a+\eta$.
Applying \cite[Corollary~6]{ref:cutkosky2023optimal} to the globally $L_J$-Lipschitz function
$\bar J$ with smoothing radius $r$ and scale parameter $\delta_0=r$ yields
$\|\nabla \bar J(\phi)\|_{2r}\le a+\eta$.
Since $\eta>0$ is arbitrary, letting $\eta\downarrow 0$ gives
$\|\nabla \bar J(\phi)\|_{2r}\le \|\nabla \widehat{\bar J}_r(\phi)\|_{r}$.
Combining this inequality with the identities above yields $\|\nabla J(\phi)\|_{2r}\le \|\nabla\hat J_r(\phi)\|_r$.
\end{proof}

Observe that if $F:\R^p\to\R$ is almost everywhere differentiable and $0<\delta_1\le \delta_2$, then
\[
\|\nabla F(x)\|_{\delta_2}\le \|\nabla F(x)\|_{\delta_1}
\qquad\forall x\in\R^p,
\]
because the admissible class of finite sets $S$ in Definition~\ref{def:goldstein} enlarges with the radius.
Since $\{r_k\}$ is nonincreasing and $\delta=2r_0$, this gives
\[
\|\nabla J(\phi)\|_{\delta}
\le \|\nabla J(\phi)\|_{2r_k}
\qquad\forall \phi\in\mathcal U,\ \forall k\in\mathbb N.
\]

Fix an integer $K\ge 1$ and run Algorithm~\ref{alg:conservative-sgd} for $K$ iterations with a deterministic positive step-size sequence $\{\alpha_k\}_{k\in\mathbb N}$ and with perturbation radii $\{r_k\}_{k\in\mathbb N}$ from Assumption~\ref{ass:query-region}.
Define the cumulative step-size summaries
$A_K\Let\sum_{k=0}^{K-1}\alpha_k$ and
$B_K\Let\sum_{k=0}^{K-1}\alpha_k^2/r_k$.
For the finite-time guarantee, we use the standard step-size-weighted randomized output convention:
after the iterates $\phi_0,\ldots,\phi_{K-1}$ have been generated, draw an index
$R\in\{0,1,\dots,K-1\}$ independently of all algorithmic randomness according to
$\mathbb P(R=k)=\alpha_k/A_K$ for $k=0,\ldots,K-1$, and report $\phi_R$ as the output.
Define the initial-to-terminal smoothed objective gap
\[
\hat\Delta_{0,K}
\Let \hat J_{r_0}(\phi_0)-\inf_{\phi\in\R^p}\hat J_{r_{K-1}}(\phi).
\]

\begin{theorem}[Finite-time $(\delta,\varepsilon)$-stationarity via parameter-space smoothing]
\label{thm:finite-time}

Let $\{r_k\}_{k\in\mathbb N}$ be the nonincreasing sequence satisfying Assumption~\ref{ass:query-region}, and set $\delta=2r_0$. Suppose Assumptions~\ref{assmpt:smoothness-bounded}--\ref{assmpt:definable} and Assumption~\ref{ass:query-region} hold.
For $K \ge 1$, let $\phi_R$ denote the output of Algorithm~\ref{alg:conservative-sgd} after $K$ iterations, using perturbation radii $\{r_k\}_{k\in\mathbb N}$ and any deterministic positive step-size sequence $\{\alpha_k\}_{k\in\mathbb N}$.
Then
\[
\mathbb{E}\big[\|\nabla J(\phi_R)\|_{\delta}\big]
\le \left(
\frac{2\hat\Delta_{0,K} + 2L_J(r_0-r_{K-1}) + C_{\mathrm{sm}}L_J\sqrt p D B_K}{2A_K}
\right)^{1/2}.
\]
\end{theorem}

\begin{proof}[Proof of Theorem~\ref{thm:finite-time}]
Set $f_k\Let\hat J_{r_k}$. By Lemma~\ref{lem:smoothing-oracle}, $\mathbb E[d_k\mid\phi_k]=\nabla f_k(\phi_k)$ almost surely. By Lemma~\ref{lem:smoothing-smoothness}, $f_k$ is $C^1$ on a neighborhood of $\mathcal U$ and
$\nabla f_k$ is $L_k$-Lipschitz on $\mathcal U$, with
$L_k\le C_{\mathrm{sm}}L_J\sqrt p/r_k$. 
Since under~Assumption~\ref{ass:query-region}(a), $\phi_k,\phi_{k+1}\in\mathcal U$ almost surely, the descent lemma and the update $\phi_{k+1}=\phi_k-\alpha_k d_k$ give
\[
f_k(\phi_{k+1})
\le f_k(\phi_k) -\alpha_k \big\langle \nabla f_k(\phi_k),\, d_k\big\rangle +\frac{L_k\alpha_k^2}{2}\|d_k\|_2^2.
\]
Taking conditional expectations given $\phi_k$ and using
$\mathbb E[d_k\mid \phi_k]=\nabla f_k(\phi_k)$
yields
\[
\mathbb E[f_k(\phi_{k+1})\mid \phi_k]
\le f_k(\phi_k) -\alpha_k \|\nabla f_k(\phi_k)\|_2^2 +\frac{L_k\alpha_k^2}{2}\,\mathbb E[\|d_k\|_2^2\mid \phi_k].
\]
Assumption~\ref{ass:query-region}(c) gives
$\mathbb E[\|d_k\|_2^2\mid \phi_k]\le D$ almost surely,
so
\[
\mathbb E[f_k(\phi_{k+1})]
\le \mathbb E[f_k(\phi_k)] -\alpha_k\,\mathbb E[\|\nabla f_k(\phi_k)\|_2^2] +\frac{L_k\alpha_k^2}{2}D.
\]
For every $k=0,1,\dots,K-2$, Lemma~\ref{lem:smoothing-drift} yields
\[
\mathbb E[f_{k+1}(\phi_{k+1})]
\le \mathbb E[f_k(\phi_{k+1})] + L_J(r_k-r_{k+1}).
\]
Combining the last two inequalities, we obtain
\[
\mathbb E[f_{k+1}(\phi_{k+1})]
\le \mathbb E[f_k(\phi_k)] -\alpha_k\,\mathbb E[\|\nabla f_k(\phi_k)\|_2^2] +\frac{L_k\alpha_k^2}{2}D + L_J(r_k-r_{k+1})
\]
for all $k=0,1,\dots,K-2$.
For the final iteration $k=K-1$, we keep
\[
\mathbb E[f_{K-1}(\phi_K)]
\le \mathbb E[f_{K-1}(\phi_{K-1})] -\alpha_{K-1}\,\mathbb E[\|\nabla f_{K-1}(\phi_{K-1})\|_2^2] +\frac{L_{K-1}\alpha_{K-1}^2}{2}D.
\]
Summing these inequalities from $k=0$ to $K-1$ gives
\[
\sum_{k=0}^{K-1}\alpha_k\,\mathbb E[\|\nabla f_k(\phi_k)\|_2^2]
\le f_0(\phi_0)-\mathbb E[f_{K-1}(\phi_K)] + L_J\sum_{k=0}^{K-2}(r_k-r_{k+1}) + \frac{D}{2}\sum_{k=0}^{K-1}L_k\alpha_k^2.
\]
By Assumption~\ref{ass:query-region}(d), $f_{K-1}=\hat J_{r_{K-1}}$ is bounded below, so
$\mathbb E[f_{K-1}(\phi_K)]\ge \inf_{\phi\in\R^p}f_{K-1}(\phi)$, and
\[
\sum_{k=0}^{K-1}\alpha_k\,\mathbb E[\|\nabla f_k(\phi_k)\|_2^2]
\le \hat\Delta_{0,K}+L_J(r_0-r_{K-1})+\frac{D}{2}\sum_{k=0}^{K-1}L_k\alpha_k^2
\le \hat\Delta_{0,K}+L_J(r_0-r_{K-1})+\frac{C_{\mathrm{sm}}L_J\sqrt p D}{2} B_K.
\]
By the definition of the random index $R$,
\[
\mathbb E\big[\|\nabla \hat J_{r_R}(\phi_R)\|_2^2\big]
= \sum_{k=0}^{K-1}\frac{\alpha_k}{A_K}\,\mathbb E[\|\nabla f_k(\phi_k)\|_2^2]
\le \frac{1}{A_K}\left(\hat\Delta_{0,K}+L_J(r_0-r_{K-1})+\frac{C_{\mathrm{sm}}L_J\sqrt p D}{2} B_K\right).
\]
By Jensen's inequality,
\[
\mathbb E\big[\|\nabla \hat J_{r_R}(\phi_R)\|_2\big]
\le \left(\mathbb E\big[\|\nabla \hat J_{r_R}(\phi_R)\|_2^2\big]\right)^{1/2}
\le \left(\frac{1}{A_K}\left(\hat\Delta_{0,K}+L_J(r_0-r_{K-1})+\frac{C_{\mathrm{sm}}L_J\sqrt p D}{2} B_K\right)\right)^{1/2}.
\]
Because $\hat J_{r_R}$ is continuously differentiable on a neighborhood of $\mathcal U$ and $\phi_R\in\mathcal U$ almost surely, the singleton set $S=\{\phi_R\}$ is admissible in Definition~\ref{def:goldstein}, and hence
$\|\nabla \hat J_{r_R}(\phi_R)\|_{r_R}\le \|\nabla \hat J_{r_R}(\phi_R)\|_2$.
By the monotonicity observation above and Lemma~\ref{lem:transfer},
\[
\|\nabla J(\phi_R)\|_{\delta}
\le \|\nabla J(\phi_R)\|_{2r_R}
\le \|\nabla \hat J_{r_R}(\phi_R)\|_{r_R}
\le \|\nabla \hat J_{r_R}(\phi_R)\|_2.
\]
Taking expectations and combining the inequalities above yields
\[
\mathbb{E}\bigl[\|\nabla J(\phi_R)\|_{\delta}\bigr]
\leq \Bigl(\frac{2\hat\Delta_{0,K} + 2L_J(r_0-r_{K-1}) + C_{\mathrm{sm}}L_J\sqrt p D B_K}{2A_K}\Bigr)^{1/2}.
\]
This completes the proof.
\end{proof}

The finite-time estimate above is compatible with the step-size schedule used in Theorem~\ref{thm:convergence}. If one chooses a positive sequence satisfying
$\sum_{k=0}^{+\infty}\alpha_k=+\infty$
and
$\lim_{k\to+\infty}\alpha_k\log(k+2)=0$,
then the first condition makes $A_K=\sum_{k=0}^{K-1}\alpha_k$ diverge, which is the denominator governing the finite-time randomized-smoothing bound, while the second condition is the fast-decay hypothesis needed for the asymptotic conservative-field theorem~\cite[Theorem~9]{ref:bolte2021conservative}. Thus both conditions are imposed for the single step-size schedule.

\section{Technical Proofs for Section~\ref{sec:metric-perturbation}}\label{ec:proofs-metric-perturbation}
\begin{lemma}[Compactness of the metric-perturbation set]
\label{lem:burg-compact}
Let $d \in \mathbb N$ and $\rho_{\mathcal T} \ge 0$. The set $\mathcal X_{\rho_{\mathcal T}}$ is nonempty and compact in $\mathbb S^d$.
\end{lemma}

\begin{proof}[Proof of Lemma~\ref{lem:burg-compact}]
Fix $d \in \mathbb N$ and $\rho_{\mathcal T} \ge 0$, and define the Burg divergence from the identity matrix by
$\Psi(X) \Let \Tr(X) - \log\det(X) - d$, where $X \in \mathbb S_{++}^d$.
Then $\mathcal X_{\rho_{\mathcal T}} = \left\{X \in \mathbb S_{++}^d : \Psi(X) \le \rho_{\mathcal T}\right\}$.
The identity matrix $I$ belongs to $\mathbb S_{++}^d$ and satisfies
$\Psi(I)=d-0-d=0\le \rho_{\mathcal T}$,
so $\mathcal X_{\rho_{\mathcal T}}$ is nonempty.

Now take any $X \in \mathcal X_{\rho_{\mathcal T}}$. Since $X$ is symmetric positive definite, there exist an orthogonal matrix $Q$ and positive eigenvalues $\lambda_1,\dots,\lambda_d$ such that
$X = Q \operatorname{diag}(\lambda_1,\dots,\lambda_d) Q^\top$.
Because the trace and determinant depend only on the eigenvalues,
$\Psi(X)
= \sum_{j=1}^d \bigl(\lambda_j - \log \lambda_j - 1\bigr)$.
Define the scalar Burg-divergence function for an eigenvalue by
\[
\varphi(\lambda) \Let \lambda - \log \lambda - 1,
\qquad \lambda>0.
\]
Then
$\Psi(X)=\sum_{j=1}^d \varphi(\lambda_j)$.
The function $\varphi$ is continuously differentiable on $(0, +\infty)$, with
$\varphi'(\lambda)=1-1/\lambda$
and
$\varphi''(\lambda)=1/{\lambda^2}>0$.
Hence $\varphi$ is strictly convex. Since $\varphi'(1)=0$ and $\varphi(1)=0$, the point $\lambda=1$ is the unique global minimizer, and therefore
$\varphi(\lambda)\ge \varphi(1) = 0$ for all $\lambda>0$.
Moreover,
$\lim_{\lambda\downarrow 0}\varphi(\lambda)=+\infty$ and $\lim_{\lambda\to+\infty}\varphi(\lambda)=+\infty$.
Because $\varphi$ is continuous, strictly decreasing on $(0,1]$, and strictly increasing on $[1, +\infty)$, for every $\rho\ge 0$ the sublevel set $\{\lambda>0:\varphi(\lambda)\le \rho\}$ is a compact interval, denoted by $[m_\rho,M_\rho]$, with
$0<m_\rho\le 1\le M_\rho<+\infty$.
Applying this with $\rho=\rho_{\mathcal T}$, we obtain
\[
\varphi(\lambda)\le \rho_{\mathcal T}
\quad\Longleftrightarrow\quad
\lambda\in [m_{\rho_{\mathcal T}},M_{\rho_{\mathcal T}}].
\]
If some eigenvalue $\lambda_j$ of $X$ satisfied $\varphi(\lambda_j)>\rho_{\mathcal T}$, then
$\Psi(X)=\sum_{k=1}^d \varphi(\lambda_k)\ge \varphi(\lambda_j)>\rho_{\mathcal T}$,
contradicting $X\in\mathcal X_{\rho_{\mathcal T}}$. Hence every eigenvalue of every matrix in $\mathcal X_{\rho_{\mathcal T}}$ lies in $[m_{\rho_{\mathcal T}},M_{\rho_{\mathcal T}}]$. Equivalently,
$m_{\rho_{\mathcal T}} I \preceq X \preceq M_{\rho_{\mathcal T}} I$ for all $X\in\mathcal X_{\rho_{\mathcal T}}$.
In particular, the set is bounded.

It remains to prove that $\mathcal X_{\rho_{\mathcal T}}$ is closed in $\mathbb S^d$. Let $\{X_k\}_{k\ge 1}\subset \mathcal X_{\rho_{\mathcal T}}$ and suppose that
$X_k \to \bar X \in \mathbb S^d$ 
in Frobenius norm. Since
$m_{\rho_{\mathcal T}} I \preceq X_k \preceq M_{\rho_{\mathcal T}} I$ for all $k$,
we have, for every $u\in\mathbb R^d$,
$u^\top(X_k-m_{\rho_{\mathcal T}}I)u \ge 0$ and $u^\top(M_{\rho_{\mathcal T}}I-X_k)u \ge 0$.
Passing to the limit yields
\[
u^\top(\bar X-m_{\rho_{\mathcal T}}I)u \ge 0,
\qquad
u^\top(M_{\rho_{\mathcal T}}I-\bar X)u \ge 0
\qquad \forall u\in\mathbb R^d.
\]
Therefore
$m_{\rho_{\mathcal T}} I \preceq \bar X \preceq M_{\rho_{\mathcal T}} I$.
In particular, $\bar X \succ 0$, so $\bar X \in \mathbb S_{++}^d$.
Since $\Psi$ is continuous on $\mathbb S_{++}^d$, we have
$\Psi(\bar X)=\lim_{k\to+\infty}\Psi(X_k)\le \rho_{\mathcal T}$,
since each $X_k$ belongs to $\mathcal X_{\rho_{\mathcal T}}$. Hence $\bar X\in\mathcal X_{\rho_{\mathcal T}}$, so $\mathcal X_{\rho_{\mathcal T}}$ is closed in $\mathbb S^d$.
Thus $\mathcal X_{\rho_{\mathcal T}}$ is closed and bounded in the finite-dimensional space $\mathbb S^d$, which is isomorphic to $\mathbb R^{d(d+1)/2}$. By the Heine--Borel theorem \cite[Theorem~2.41]{ref:rudin1976principles}, it is compact in $\mathbb S^d$.
\end{proof}

\begin{proof}[Proof of Proposition~\ref{prop:attainment}]
By Lemma~\ref{lem:burg-compact}, the set $\mathcal X_{\rho_{\mathcal T}}$ is nonempty and compact in $\mathbb S^d$. It is also convex because
$\mathcal X_{\rho_{\mathcal T}}
= \left\{X\in\mathbb S_{++}^d : \Tr(X)-\log\det(X)-d\le \rho_{\mathcal T}\right\}$
is the sublevel set of a convex function on $\mathbb S_{++}^d$.
Define the matrix-to-edge-travel-time map by
\[
f:\mathbb S^d\to\mathbb R^{|\mathcal E|},
\qquad
f(X)[e] \Let \phi_{\mathcal T}[e]^\top X \phi_{\mathcal T}[e], \quad e\in\mathcal E.
\]
For each fixed edge $e$,
$f(X)[e]
= \phi_{\mathcal T}[e]^\top X \phi_{\mathcal T}[e]
= \langle X,\phi_{\mathcal T}[e]\phi_{\mathcal T}[e]^\top\rangle$,
so each component is linear in $X$. Hence $f$ is linear and continuous. Therefore $\mathcal C(\mathcal T)=f(\mathcal X_{\rho_{\mathcal T}})$ is nonempty, compact, and convex as the linear image of a nonempty, compact, convex set. 
As a consequence, for any fixed $z_1\in\mathcal Z$, the function $F(c) = c^\top z_1 - \min_{z\in\mathcal Z} c^\top z$ is continuous on $\mathcal C(\mathcal T)$. Indeed, $\mathcal Z\subseteq\{0,1\}^{|\mathcal E|}$ is finite, so $c\mapsto \min_{z\in\mathcal Z} c^\top z$ is the pointwise minimum of finitely many linear functions and is therefore continuous. Hence $F$ attains its maximum over the compact set $\mathcal C(\mathcal T)$.
\end{proof}

\begin{proof}[Proof of Proposition~\ref{prop:rho-refor}]
We express~\eqref{eq:uncertainty-radius} as:
\[
\rho^{(i)}
= \min \left\{ D_{\mathrm{Burg}}(X,I): X\in\mathbb S_{++}^d,\; h^{(i)}(X)\ge t^{(i)} \right\},
\]
where
\[
h^{(i)}(X)
= \min_{\substack{z\in\{0,1\}^{|\mathcal E|},\, Az=b^{(i)}}} c_{\mathcal T^{(i)}}(X)^\top z,
\qquad
c_{\mathcal T^{(i)}}(X)[e]
= \phi_{\mathcal T^{(i)}}[e]^\top X\phi_{\mathcal T^{(i)}}[e].
\]
Fix $X\in\mathbb S_{++}^d$. The inner problem is the shortest-path problem
$\min_{\substack{z\in\{0,1\}^{|\mathcal E|},\, Az=b^{(i)}}} c_{\mathcal T^{(i)}}(X)^\top z$.
Consider its linear-programming relaxation
\[
\mathcal P^{(i)}
\Let \left\{z\in\mathbb R^{|\mathcal E|} : Az=b^{(i)},\; 0\le z\le 1 \right\}.
\]
Because $A$ is a node--arc incidence matrix, it is totally unimodular. Therefore the block matrix $[A; I; -I]$ is also totally unimodular, where semicolons denote vertical concatenation. Since the right-hand side $[b^{(i)}; \mathbf 1; \mathbf 0]$ is integral, every extreme point of $\mathcal P^{(i)}$ is integral. Together with the bounds $0\le z\le 1$, this implies that every extreme point of $\mathcal P^{(i)}$ belongs to $\{0,1\}^{|\mathcal E|}$. Hence the shortest-path problem admits the exact relaxation
\[
h^{(i)}(X)
= \min_{z\in\mathcal P^{(i)}} c_{\mathcal T^{(i)}}(X)^\top z
= \min\left\{
c_{\mathcal T^{(i)}}(X)^\top z :
Az=b^{(i)},\;
0\le z\le 1
\right\}.
\]
Introduce a free multiplier $\pi\in\mathbb R^{|\mathcal V|}$ for the equality constraint $Az=b^{(i)}$ and a nonnegative multiplier $\omega\in\mathbb R_+^{|\mathcal E|}$ for the upper bound constraint $z\le 1$. Its dual program is
\[
\max_{\pi,\omega}
\left\{
(b^{(i)})^\top\pi - \mathbf 1^\top\omega :
A^\top\pi-\omega\le c_{\mathcal T^{(i)}}(X),\;
\pi\in\mathbb R^{|\mathcal V|},\;
\omega\in\mathbb R_+^{|\mathcal E|}
\right\}.
\]

For the samples under consideration, the corresponding origin--destination pair is feasible, so the primal LP is feasible; moreover the bounds $0\le z\le 1$ make the feasible region bounded. Hence strong duality holds, and thus
\[
h^{(i)}(X)
= \max_{\substack{\pi\in\mathbb R^{|\mathcal V|},\, \omega\in\mathbb R_+^{|\mathcal E|}}}
\left\{(b^{(i)})^\top\pi - \mathbf 1^\top\omega : A^\top\pi-\omega\le c_{\mathcal T^{(i)}}(X) \right\}.
\]
It follows that
$h^{(i)}(X)\ge t^{(i)}$ 
if and only if there exist dual variables $(\pi,\omega)$ such that
\[
A^\top\pi-\omega\le c_{\mathcal T^{(i)}}(X),
\qquad
\omega\in\mathbb R_+^{|\mathcal E|},
\qquad
(b^{(i)})^\top\pi-\mathbf 1^\top\omega\ge t^{(i)}.
\]
Substituting this characterization into the definition of $\rho^{(i)}$ yields
\[
\rho^{(i)}
=
\min
\left\{
D_{\mathrm{Burg}}(X,I):
\begin{array}{l}
X\in\mathbb S_{++}^d,\;
\pi\in\mathbb R^{|\mathcal V|},\;
\omega\in\mathbb R_+^{|\mathcal E|},\,
A^\top\pi-\omega\le c_{\mathcal T^{(i)}}(X),\,
(b^{(i)})^\top\pi-\mathbf 1^\top\omega\ge t^{(i)}
\end{array}
\right\}.
\]
Using the definition of $c_{\mathcal T^{(i)}}(X)$,
$c_{\mathcal T^{(i)}}(X)[e]
= \phi_{\mathcal T^{(i)}}[e]^\top X\phi_{\mathcal T^{(i)}}[e]
= \langle X,\phi_{\mathcal T^{(i)}}[e]\phi_{\mathcal T^{(i)}}[e]^\top\rangle$,
we obtain exactly the formulation in~\eqref{eq:rho-refor}.

Finally, the objective $D_{\mathrm{Burg}}(X,I)$ is convex in $X$, and every constraint is affine in $(X,\pi,\omega)$. Hence the reformulated problem is a convex optimization problem equivalent to~\eqref{eq:uncertainty-radius}.
\end{proof}

\section{Technical Proofs for Section~\ref{sec:solution-procedure}}\label{ec:proofs-solution-procedure}
\begin{proof}[Proof of Proposition~\ref{prop:subgradient-sA}]
For each fixed $z$, the map $X \mapsto c_{\mathcal T}(X)^\top z$ is affine in $X$, and the same holds for $X \mapsto -c_{\mathcal T}(X)^\top z = -\inner{G(z)}{X}$. 
The function $f_2$ satisfies
\[
f_2(X)
= - \min_{z \in \mathcal Z} c_{\mathcal T}(X)^\top z
= \max_{z \in \mathcal Z} \bigl(-c_{\mathcal T}(X)^\top z\bigr)
= \max_{z \in \mathcal Z} \bigl(-\inner{G(z)}{X}\bigr).
\]
Since $\mathcal Z$ is a finite, nonempty set of $0$--$1$ path indicators, the minimum $\min_{z \in \mathcal Z} c_{\mathcal T}(X)^\top z$ is attained for every $X$. The maximum above is a finite pointwise maximum of affine functions of $X$. Therefore, $f_2$ is convex.

For each $z \in \mathcal Z$, denote the affine component associated with path $z$ by
$\varphi_z(X) = -\inner{G(z)}{X}$.
The map $\varphi_z$ is affine in $X$ with gradient $\nabla_X \varphi_z(X) = -G(z)$ for all $X$. The finite-max version of Danskin's theorem~\cite{ref:danskin2012theory} implies that the subdifferential of $f_2$ at $X$ is the convex hull of the gradients of the active functions at $X$:
\[
\partial f_2(X)
= \operatorname{conv} \bigl\{ \nabla_X \varphi_z(X) : z \in \argmax_{u \in \mathcal Z} \varphi_u(X) \bigr\}.
\]
The active set coincides with the set of shortest paths under $X$, because
\[
\argmax_{u \in \mathcal Z} \varphi_u(X)
= \argmax_{u \in \mathcal Z} \bigl(-c_{\mathcal T}(X)^\top u\bigr)
= \argmin_{u \in \mathcal Z} c_{\mathcal T}(X)^\top u
= \mathcal Z\opt(X).
\]
Substituting the explicit gradient yields
$\partial f_2(X) = \operatorname{conv}\bigl\{ -G(z) : z \in \mathcal Z\opt(X) \bigr\}$.
\end{proof}

\begin{proof}[Proof of Theorem~\ref{thm:solution-of-DC-sub}]
We consider the convex optimization problem
\begin{equation}
\label{eq:DC-sub-X-proof}
\min_{X \in \mathbb S_{++}^d} \bigl\{ \inner{G^{(k)}}{X} : \Tr(X) - \log\det(X) - d \le \rho_{\mathcal T} \bigr\}.
\end{equation}
The objective is linear in $X$, and the feasible set is a sublevel set of the strictly convex function $\Psi(X) \Let \Tr(X) - \log\det(X) - d$ on $\mathbb S_{++}^d$. When $\rho_{\mathcal T} > 0$, Slater's condition holds: there exists $\epsilon > 0$ small enough so that $X = (1+\epsilon)I$ satisfies $\Psi(X) < \rho_{\mathcal T}$, because $d\bigl((1+\epsilon) - \log(1+\epsilon) - 1\bigr)$ tends to $0$ as $\epsilon \downarrow 0$. Hence the Karush--Kuhn--Tucker (KKT) conditions 
\[
\begin{cases}
\text{\textbf{Primal feasibility: }}
X\opt \in \mathbb S_{++}^d,
\Tr(X\opt) - \log\det(X\opt) - d \le \rho_{\mathcal T};\\
\text{\textbf{Dual feasibility: }}
\gamma\opt \ge 0;\\
\text{\textbf{Complementary slackness: }}
\gamma\opt \bigl(\Tr(X\opt) - \log\det(X\opt) - d - \rho_{\mathcal T} \bigr) = 0;\\
\text{\textbf{Stationarity: }}
\nabla_X L(X\opt,\gamma\opt) = 0
\end{cases}
\]
are therefore necessary and sufficient for optimality~\cite[Theorem~5.3.1]{bazaraa2006nonlinear}.

We now construct a primal--dual pair for~\eqref{eq:DC-sub-X-proof} and verify the KKT conditions:
Let $\lambda_1,\dots,\lambda_d$ be the eigenvalues of $G^{(k)}$, let $\lambda_{\min} \Let \min_{i=1,\dots,d}\lambda_i$, and set $\gamma_0 \Let \max\{0,-\lambda_{\min}\}$. For $\gamma>\gamma_0$, define the scalar Burg-divergence term by
\[
x_i(\gamma) \Let \frac{\gamma}{\lambda_i+\gamma},
\qquad
\kappa(x) \Let x-\log x-1,
\qquad
\varphi(\gamma) = \sum_{i=1}^d \kappa\bigl(x_i(\gamma)\bigr) - \rho_{\mathcal T}.
\]
We first show that the equation $\varphi(\gamma)=0$ admits a unique solution on $(\gamma_0,+\infty)$. The definition of $\gamma_0$ ensures $\lambda_i+\gamma>0$ for all $i$ and all $\gamma>\gamma_0$, so each $x_i(\gamma)$ is well defined and positive on this interval. If $\lambda_{\min}<0$, then some $\lambda_j=\lambda_{\min}$ and $\lambda_j+\gamma \downarrow 0$ as $\gamma \downarrow \gamma_0$, so $x_j(\gamma) \to +\infty$ and $\kappa\bigl(x_j(\gamma)\bigr)\to +\infty$. If $\lambda_{\min}\ge 0$, then $\gamma_0=0$ and $G^{(k)}\neq 0$ implies some $\lambda_i>0$, so $x_i(\gamma)\to 0$ as $\gamma\downarrow 0$ and $\kappa\bigl(x_i(\gamma)\bigr)\to +\infty$. In both cases $\varphi(\gamma)\to +\infty$ as $\gamma \downarrow \gamma_0$. In contrast, $x_i(\gamma)\to 1$ as $\gamma\to+\infty$, so $\kappa\bigl(x_i(\gamma)\bigr)\to 0$ and hence $\varphi(\gamma)\to -\rho_{\mathcal T}$ as $\gamma\to+\infty$. Differentiation gives
\[
\frac{\dd x_i(\gamma)}{\dd \gamma} = \frac{\lambda_i}{(\lambda_i+\gamma)^2},
\qquad
\frac{\dd}{\dd \gamma}\kappa\bigl(x_i(\gamma)\bigr)
= \Bigl(1-\frac{1}{x_i(\gamma)}\Bigr)\frac{\dd x_i(\gamma)}{\dd \gamma}
= -\frac{\lambda_i^2}{\gamma(\lambda_i+\gamma)^2},
\]
so
\[
\varphi'(\gamma)
= -\sum_{i=1}^d \frac{\lambda_i^2}{\gamma(\lambda_i+\gamma)^2}
<0
\quad \text{for all }\gamma>\gamma_0,
\]
where the strict inequality uses $G^{(k)}\neq 0$. Therefore $\varphi$ is continuous and strictly decreasing on $(\gamma_0,+\infty)$ with range $(-\rho_{\mathcal T},+\infty)$, so there exists a unique $\gamma\opt \in (\gamma_0,+\infty)$ such that $\varphi(\gamma\opt)=0$.

We now define the candidate primal point $X\opt \Let \gamma\opt\bigl(G^{(k)}+\gamma\opt I\bigr)^{-1}$ 
and the candidate multiplier $\gamma\opt$. We verify the KKT conditions for $(X\opt,\gamma\opt)$.
The inequality $\gamma\opt>\gamma_0 \ge -\lambda_{\min}$ implies $G^{(k)}+\gamma\opt I \succ 0$, so $X\opt \succ 0$. The eigenvalues of $X\opt$ equal $x_i(\gamma\opt)=\gamma\opt/(\lambda_i+\gamma\opt)$, so
$\Tr(X\opt) - \log\det(X\opt) - d - \rho_{\mathcal T}
= \sum_{i=1}^d \kappa\bigl(x_i(\gamma\opt)\bigr) - \rho_{\mathcal T}
= \varphi(\gamma\opt)
= 0$.
This identity demonstrates primal feasibility and shows that the constraint is active.

Dual feasibility holds because $\gamma\opt>\gamma_0\ge 0$. Complementary slackness holds because the constraint is active. 
It remains to verify stationarity. Define the Lagrangian with multiplier $\gamma \ge 0$ as
\[
L(X,\gamma)
\Let \inner{G^{(k)}}{X} + \gamma\bigl(\Tr(X) - \log\det(X) - d - \rho_{\mathcal T}\bigr).
\]
The matrix derivative of $X \mapsto \Tr(X) - \log\det(X)$ is $I - X^{-1}$, so
$\nabla_X L(X,\gamma)
= G^{(k)} + \gamma\bigl(I - X^{-1}\bigr)$.
The definition of $X\opt$ gives
$(X\opt)^{-1}
= \frac{1}{\gamma\opt}\bigl(G^{(k)}+\gamma\opt I\bigr)
= \frac{1}{\gamma\opt}G^{(k)} + I$.
Substituting into the stationarity expression yields
\[
\nabla_X L(X\opt,\gamma\opt)
= G^{(k)} + \gamma\opt\Bigl(I - \bigl(\frac{1}{\gamma\opt}G^{(k)} + I\bigr)\Bigr)
= G^{(k)} - G^{(k)} = 0,
\]
so stationarity holds.

The pair $(X\opt,\gamma\opt)$ satisfies all KKT conditions. Convexity and Slater's condition imply that $X\opt$ is a global minimizer.
We now show uniqueness. Let $\tilde X$ be any minimizer. Slater's condition implies that there exists $\tilde\gamma \ge 0$ such that $(\tilde X,\tilde\gamma)$ satisfies the KKT conditions. Stationarity gives
$G^{(k)} + \tilde\gamma\bigl(I-\tilde X^{-1}\bigr)=0$.
If $\tilde\gamma=0$, this identity forces $G^{(k)}=0$, which contradicts the assumption $G^{(k)}\neq 0$. Hence $\tilde\gamma>0$, complementary slackness forces the constraint to be active at $\tilde X$, and the same eigendecomposition argument yields $\varphi(\tilde\gamma)=0$. The strict monotonicity of $\varphi$ implies $\tilde\gamma=\gamma\opt$, and stationarity then forces $\tilde X = \gamma\opt(G^{(k)}+\gamma\opt I)^{-1} = X\opt$. This proves the closed-form solution and uniqueness. It remains to prove the search interval.

We first derive an upper bound for $\kappa(x)$ when $x>0$. The inequality $\log(x) \ge 1 - 1/x$ for all $x>0$, which follows from the concavity of the logarithm, implies
\[
\kappa(x)
= x - \log x - 1
\le x - \Bigl(1 - \frac{1}{x}\Bigr) - 1
= x - 2 + \frac{1}{x}
= \frac{(x-1)^2}{x}.
\]
For each $i$ and $\gamma > \gamma_0$ we therefore have
$\kappa\bigl(x_i(\gamma)\bigr)
\le \bigl(x_i(\gamma)-1\bigr)^2 / x_i(\gamma)$.
We now simplify the right-hand side. Since
\[
\frac{\bigl(x_i(\gamma)-1\bigr)^2}{x_i(\gamma)}
= \bigl(\frac{\gamma}{\lambda_i + \gamma} - 1\bigr)^2 / \bigl( \frac{\gamma}{\lambda_i + \gamma} \bigr)
= \frac{\lambda_i^2}{(\lambda_i + \gamma)^2} \frac{\lambda_i + \gamma}{\gamma}
= \frac{\lambda_i^2}{\gamma(\lambda_i + \gamma)},
\]
summing over $i$ yields
\[
\varphi(\gamma)
= \sum_{i=1}^d \kappa\bigl(x_i(\gamma)\bigr) - \rho_{\mathcal T}
\le \sum_{i=1}^d \frac{\lambda_i^2}{\gamma(\lambda_i + \gamma)} - \rho_{\mathcal T}
\le \sum_{i=1}^d \frac{\lambda_i^2}{\gamma(\lambda_{\min} + \gamma)} - \rho_{\mathcal T}
= \frac{\sum_{i=1}^d \lambda_i^2}{\gamma(\lambda_{\min} + \gamma)} - \rho_{\mathcal T}.
\]
Define the upper bound
\[
\varphi_{\mathrm{upper}}(\gamma)
\Let \frac{\sum_{i=1}^d \lambda_i^2}{\gamma(\lambda_{\min} + \gamma)} - \rho_{\mathcal T}.
\]
We have $\varphi(\gamma) \le \varphi_{\mathrm{upper}}(\gamma)$ for all $\gamma > \gamma_0$. Any $\gamma$ that satisfies $\varphi_{\mathrm{upper}}(\gamma) \le 0$ also satisfies $\varphi(\gamma) \le 0$.
The equation $\varphi_{\mathrm{upper}}(\gamma) = 0$ expands to the quadratic equation
\[
\gamma^2 + \lambda_{\min}\gamma - \frac{\sum_{i=1}^d \lambda_i^2}{\rho_{\mathcal T}} = 0.
\]
Solving this equation gives the positive root $\gamma_{\max}$ in~\eqref{eq:upper-gamma}.
Hence
$\varphi_{\mathrm{upper}}(\gamma_{\max})=0$ and $\varphi(\gamma_{\max})\le0$. Moreover
$\gamma_{\max}>(-\lambda_{\min}+|\lambda_{\min}|)/2=\gamma_0$. Since
$\varphi(\gamma_{\max})\le0=\varphi(\gamma\opt)$ and $\varphi$ is strictly decreasing, we obtain
$\gamma\opt\le\gamma_{\max}$. Together with $\gamma\opt>\gamma_0$, this proves
$\gamma\opt\in(\gamma_0,\gamma_{\max}]$.
\end{proof}

\begin{proof}[Proof of Lemma~\ref{lem:hybrid-gamma-search-app}]
Define the root-finding residual for the multiplier equation by
\[
\varphi(\gamma) \Let \sum_{i=1}^d \Bigl( \frac{\gamma}{\lambda_i+\gamma} -\log\bigl(\frac{\gamma}{\lambda_i+\gamma}\bigr) -1 \Bigr) -\rho_{\mathcal T},
\qquad \gamma\in(\gamma_0,\gamma_{\max}],
\]
so that $\varphi(\gamma\opt)=0$. 
Set
$t\Let\gamma-\gamma_0\in(0,\Delta_\gamma]$ and
$\widehat\varphi(t)\Let\varphi(\gamma_0+t)$. Since $\varphi$ is analytic on
$(\gamma_0,+\infty)$, $\widehat\varphi$ is analytic on $(0,\Delta_\gamma]$.

Let $\mathcal I \Let \{i:\lambda_i\neq 0\}$ denote the nonzero-eigenvalue index set. For each $i\in\mathcal I$, define the $i$th summand of $\widehat\varphi$ by
\[
g_i(t) \Let \frac{\gamma_0+t}{\lambda_i+\gamma_0+t} -\log\Bigl(\frac{\gamma_0+t}{\lambda_i+\gamma_0+t}\Bigr) -1.
\]
Thus, for $t\in(0,\Delta_\gamma]$, the residual and its derivatives satisfy
\[
\widehat\varphi(t)=\sum_{i\in\mathcal I} g_i(t)-\rho_{\mathcal T},
\qquad
g_i'(t) = -\frac{\lambda_i^2}{(\lambda_i+\gamma_0+t)^2(\gamma_0+t)} <0,
\qquad
g_i''(t) = \frac{\lambda_i^2(\lambda_i+3(\gamma_0+t))} {(\gamma_0+t)^2(\lambda_i+\gamma_0+t)^3}>0.
\]
Thus $\widehat\varphi$ is strictly decreasing and convex on $(0,\Delta_\gamma]$.

To verify the derivative growth condition, write
$u(t)\Let(\gamma_0+t)^{-1}$ and $v_i(t)\Let(\lambda_i+\gamma_0+t)^{-2}$,
so $g_i'(t)=-\lambda_i^2 u(t)v_i(t)$. For every integer $j\ge 0$,
\[
u^{(j)}(t)=(-1)^j j!(\gamma_0+t)^{-j-1},
\qquad
v_i^{(j)}(t)=(-1)^j (j+1)!(\lambda_i+\gamma_0+t)^{-j-2}.
\]
Applying Leibniz's rule and taking the absolute value, for each integer $m\ge 2$,
\[
|g_i^{(m)}(t)|
=
|g_i'(t)|
\sum_{j=0}^{m-1}
\binom{m-1}{j}
j!(m-j)!
(\gamma_0+t)^{-j}
(\lambda_i+\gamma_0+t)^{-m+1+j}.
\]
Because $\gamma_0+t\ge t$ and $\lambda_i+\gamma_0+t\ge t$, we obtain
\begin{align*}
|g_i^{(m)}(t)|
&\le
|g_i'(t)|
\sum_{j=0}^{m-1}
\binom{m-1}{j}j!(m-j)! \, t^{-m+1}
= |g_i'(t)|(m-1)!t^{-m+1}\sum_{j=0}^{m-1}(m-j)
= |g_i'(t)|\,\frac{(m+1)!}{2}\,t^{-m+1}.
\end{align*}
Summing over $i\in\mathcal I$ yields
$|\widehat\varphi^{(m)}(t)|
\le |\widehat\varphi'(t)|(m+1)!t^{-m+1}/2$,
and hence
\[
\bigl| \frac{\widehat\varphi^{(m)}(t)}{m!\widehat\varphi'(t)} \bigr|^{1/(m-1)}
\le \bigl(\frac{m+1}{2}\bigr)^{1/(m-1)} t^{-1}
\le \frac{3}{2}t^{-1}.
\]

Normalize the interval by $\gamma=\gamma_0+\Delta_\gamma y$ and define $\Theta(y) = \varphi(\gamma_0+\Delta_\gamma y)=\widehat\varphi(\Delta_\gamma y)$.
With $y\opt=(\gamma\opt-\gamma_0)/\Delta_\gamma$, the chain rule gives
\[
\Theta^{(m)}(y)=\Delta_\gamma^m \widehat\varphi^{(m)}(\Delta_\gamma y),
\qquad
\Theta'(y)=\Delta_\gamma \widehat\varphi'(\Delta_\gamma y),
\]
and hence
\[
\bigl| \frac{\Theta^{(m)}(y)}{m!\Theta'(y)} \bigr|^{1/(m-1)}
= \Delta_\gamma \bigl| \frac{\widehat\varphi^{(m)}(\Delta_\gamma y)}{m!\widehat\varphi'(\Delta_\gamma y)} \bigr|^{1/(m-1)}
\le \frac{3}{2y}.
\]
Thus $\Theta$ satisfies the hypothesis of the approximate-root criterion in~\cite[Theorem~2]{ref:ye1992new} with $\alpha=3/2$.
Fix a target accuracy $\epsilon_\gamma\in(0,\Delta_\gamma]$ and set
$\epsilon_y = \epsilon_\gamma/\Delta_\gamma \in (0,1]$.
If $y\opt \le \epsilon_y^3$, then setting $\hat y \Let \epsilon_y^3$ gives
$0 \le \hat y-y\opt \le \epsilon_y^3 \le \epsilon_y$.
Otherwise $y\opt\in[\epsilon_y^3,1]$. In that case, \cite[Theorem~2]{ref:ye1992new} implies that whenever $y\opt\in[\hat y,(13/12)\hat y]$, the point $\hat y$ is an approximate root of $\Theta$. Applying the geometric interval-selection and Newton refinement scheme in~\cite[Section~4]{ref:ye1992new} on $[\epsilon_y^3,1]$ therefore computes an $\epsilon_y$-accurate estimate of $y\opt$ in
$O\left(\log\log(1/\epsilon_y)\right)
= O\left(\log\log(\Delta_\gamma/\epsilon_\gamma)\right)$
scalar iterations. Mapping back by $\gamma=\gamma_0+\Delta_\gamma y$ gives an $\epsilon_\gamma$-accurate estimate of $\gamma\opt$ with the same complexity.
\end{proof}

\begin{proof}[Proof of Proposition~\ref{prop:complexity}]
We analyze the main operations in one iteration of Algorithm~\ref{alg:DCA-algorithm} and bound their running times in terms of $|\mathcal E|$, $|\mathcal V|$, and $d$. First, the algorithm updates the edge travel times. For each edge $e \in \mathcal E$, it computes $c_{\mathcal T}(X^{(k)})[e]$, which involves a matrix--vector product and a vector inner product in $\mathbb R^d$. This requires $O(d^2)$ operations per edge. Updating the costs for all edges, therefore, takes $O(|\mathcal E| d^2)$ time.

Second, the algorithm finds a shortest path $z^{(k)}$ from either depot to the demand node $D$ under the current costs $c_{\mathcal T}(X^{(k)})$. It can do so by running Dijkstra's algorithm once on the network, with the depots treated as sources. On a sparse graph with nonnegative edge weights, Dijkstra's algorithm using a Fibonacci heap has complexity $O(|\mathcal E| + |\mathcal V| \log |\mathcal V|)$.

Third, the algorithm forms the matrix $G^{(k)}$. 
For each edge $e$, the outer product $\phi_{\mathcal T}[e]\bigl(\phi_{\mathcal T}[e]\bigr)^\top$ costs $O(d^2)$ operations, and the sum extends over the edges used by $z^{(k)}$ and $z_1$. Hence forming $G^{(k)}$ costs $O(|\mathcal E| d^2)$ operations per iteration.

Fourth, the algorithm computes the eigendecomposition $G^{(k)} = Q \Lambda_\lambda Q^\top$ of the $d \times d$ symmetric matrix $G^{(k)}$. Standard numerical routines for the full eigendecomposition of a dense symmetric matrix have complexity $O(d^3)$.

Fifth, the algorithm solves the scalar equation~\eqref{eq:lagrangian-target} for $\gamma\opt$ with the hybrid root-finding routine described in Lemma~\ref{lem:hybrid-gamma-search-app}. The complexity is $O\bigl(d \log\log(\Delta_\gamma/\epsilon_\gamma)\bigr)$.

Sixth, the algorithm updates the matrix $X^{(k+1)}\gets\gamma^\star\bigl(G^{(k)}+\gamma^\star I\bigr)^{-1}$.
Given the eigendecomposition $G^{(k)}=Q\Lambda_\lambda Q^\top$ with $\Lambda_\lambda=\Diag(\lambda_1,\ldots,\lambda_d)$, we can write
$X^{(k+1)} = Q\Lambda_XQ^\top$ with 
$\Lambda_X = \Diag\bigl(\frac{\gamma^\star}{\lambda_1+\gamma^\star}, \ldots, \frac{\gamma^\star}{\lambda_d+\gamma^\star} \bigr)$.
Forming the diagonal entries of $\Lambda_X$ costs $O(d)$ time. 
Constructing $X^{(k+1)}=Q\Lambda_XQ^\top$ is dominated by dense matrix
multiplication and costs $O(d^3)$ time.

The remaining operations in the iteration are at most linear in $d$ or $|\mathcal E|$ and are dominated by the costs above. Summing all contributions, one iteration of Algorithm~\ref{alg:DCA-algorithm} costs
$O\bigl(|\mathcal E| d^2 + |\mathcal V| \log |\mathcal V| + d^3 + d \log\log(\Delta_\gamma/\epsilon_\gamma)\bigr)$.
This completes the proof.
\end{proof}

\begin{proof}[Proof of Theorem~\ref{thm:convergence-of-DCA}]
Following the standard DCA reformulation for convexly constrained DC programs in \cite[Section~1]{ref:le2014dc}, we absorb the convex constraint into the first DC component and work with
\[
\mathcal X \Let \mathcal X_{\rho_{\mathcal T}},
\qquad \tilde f_1 \Let f_1 + \iota_{\mathcal X},
\qquad q_k(X) \Let \tilde f_1(X) - \inner{\zeta^{(k)}}{X}.
\]
Then $\tilde f_1$ is proper, convex, and lower semicontinuous on $\mathbb S^d$, while $f_2$ is finite, convex, and lower semicontinuous on $\mathbb S^d$. Since $\mathbb S^d$ is finite-dimensional, $(\mathbb S^d,\|\cdot\|_F)$ is a Banach space, and $\partial \tilde f_1$ and $\partial f_2$ have closed graphs. Lemma~\ref{lem:burg-compact} guarantees that $\mathcal X$ is nonempty and compact.

Every completed iteration of the DCA subproblem satisfies
\begin{equation}
\label{eq:qk_iteration}
X^{(k+1)} \in \argmin_{X\in\mathcal X} q_k(X).
\end{equation}
In particular, since $X^{(k)}\in\mathcal X$,
$q_k(X^{(k+1)}) \le q_k(X^{(k)})$.
Because $X^{(k)}$ and $X^{(k+1)}$ are feasible, this inequality becomes
$f_1(X^{(k)}) - f_1(X^{(k+1)})
\ge \inner{\zeta^{(k)}}{X^{(k)} - X^{(k+1)}}$.
We obtain
\begin{equation}
\label{eq:routeB-W-lower}
W_k
\Let f_1(X^{(k)}) - f_1(X^{(k+1)}) - \inner{\zeta^{(k)}}{X^{(k)} - X^{(k+1)}} \ge 0.
\end{equation}

Convexity of $f_2$ and $\zeta^{(k)} \in \partial f_2(X^{(k)})$ give
$f_2(X^{(k+1)}) \ge f_2(X^{(k)}) + \inner{\zeta^{(k)}}{X^{(k+1)} - X^{(k)}}$.
So 
\[
\begin{aligned}
f(X^{(k)}) - f(X^{(k+1)})
&= \bigl(f_1(X^{(k)}) - f_1(X^{(k+1)}) - \inner{\zeta^{(k)}}{X^{(k)} - X^{(k+1)}}\bigr) \\ 
&\qquad + \bigl(\inner{\zeta^{(k)}}{X^{(k)} - X^{(k+1)}} - (f_2(X^{(k)}) - f_2(X^{(k+1)}))\bigr)
\ge W_k.
\end{aligned}
\]
Together with \eqref{eq:routeB-W-lower}, this shows that $f(X^{(k+1)})\le f(X^{(k)})$ for every completed iteration.
Because $\mathcal X$ is compact and $f$ is continuous, $f$ attains a finite minimum value $f\opt = \min_{X \in \mathcal X} f(X)$.
Summing the previous inequality from $k = 0$ to $m-1$, for any $m$ not exceeding the number of completed DCA iterations, gives the finite-run bound
\[
0\le \sum_{k=0}^{m-1} W_k
\le f(X^{(0)}) - f(X^{(m)})
\le f(X^{(0)}) - f\opt.
\]

For the asymptotic statement, suppose that the same DCA recursion is continued indefinitely. Letting $m\to+\infty$ in the finite-run bound gives $\sum_{k=0}^{+\infty} W_k < +\infty$, hence $W_k \to 0$.
The monotonicity $f(X^{(k+1)})\le f(X^{(k)})$ and the lower bound $f(X^{(k)})\ge f\opt$ imply that the sequence $\{f(X^{(k)})\}$ converges.
We now show that every accumulation point of $\{X^{(k)}\}$ is DC-critical. Compactness of $\mathcal X$ implies that $\{X^{(k)}\}$ has at least one accumulation point. Fix any such point $\bar X$ and extract a subsequence $\{X^{(k_j)}\}$ such that $X^{(k_j)} \to \bar X$.
Because $\{X^{(k_j+1)}\} \subseteq \mathcal X$, after passing to a further subsequence we may assume
$X^{(k_j+1)} \to \hat X$ for some $\hat X \in \mathcal X$.
Proposition~\ref{prop:subgradient-sA} shows that each $\zeta^{(k)}$ belongs to the convex hull of finitely many matrices $\{-G(z) : z \in \mathcal Z\}$, so $\{\zeta^{(k)}\}$ is bounded. Passing to a further subsequence if necessary, we may also assume
$\zeta^{(k_j)} \to \bar \zeta$ for some $\bar \zeta \in \mathbb S^d$.

The optimality condition in \eqref{eq:qk_iteration} gives $\zeta^{(k)} \in \partial \tilde f_1(X^{(k+1)})$.
By closedness of $\partial \tilde f_1$, we obtain $\bar \zeta \in \partial \tilde f_1(\hat X)$, which implies that $\hat X$ minimizes the affine surrogate $X \mapsto \tilde f_1(X) - \inner{\bar \zeta}{X}$.
The subgradient inclusion $\zeta^{(k)} \in \partial f_2(X^{(k)})$ and the closedness of $\partial f_2$ yield $\bar \zeta \in \partial f_2(\bar X)$.
Because $W_{k_j} \to 0$, we have
\[
\bigl(f_1(X^{(k_j)}) - f_1(X^{(k_j+1)})\bigr) - \inner{\zeta^{(k_j)}}{X^{(k_j)} - X^{(k_j+1)}}
\to 0.
\]
Passing to the limit along this subsequence and using continuity of $f_1$ and the inner product gives
$f_1(\bar X) - f_1(\hat X) - \inner{\bar \zeta}{\bar X - \hat X} = 0$.
Since both $\bar X$ and $\hat X$ are feasible, this identity is equivalent to
$\tilde f_1(\bar X) - \inner{\bar \zeta}{\bar X}
= \tilde f_1(\hat X) - \inner{\bar \zeta}{\hat X}$.
Thus $\bar X$ and $\hat X$ minimize the same affine surrogate.
Therefore, the convex first-order optimality condition at $\bar X$ yields
$0 \in \partial\bigl(\tilde f_1 - \inner{\bar \zeta}{\cdot}\bigr)(\bar X)
= \partial \tilde f_1(\bar X) - \bar \zeta$.
Together with $\bar \zeta \in \partial f_2(\bar X)$, we conclude
$0 \in \partial \tilde f_1(\bar X) - \partial f_2(\bar X)
= \partial\bigl(f_1 + \iota_{\mathcal X_{\rho_{\mathcal T}}}\bigr)(\bar X) - \partial f_2(\bar X)$.
This is exactly the standard DC-criticality condition $\partial h(x\opt)\cap \partial g(x\opt)\neq\emptyset$ in \cite[Section~1]{ref:le2014dc}, applied with $g = f_1 + \iota_{\mathcal X_{\rho_{\mathcal T}}}$ and $h = f_2$. Thus $\bar X$ is DC-critical.
\end{proof}

\section{Transfer Test in Sha Tin}
\label{ec:sha-tin-transfer}

We repeat the evaluation design from Section~\ref{sec:experiments:setup} in Sha Tin, replacing the depot pair by Sha Tin Ambulance Depot (SHA) and Tin Sum Ambulance Depot (TIN). We build the local directed OSMnx road network, map historical OHCA demands to the nearest nodes, and replay each demand at the evaluation time. For each incident, the Google primary path is the Google-recommended path from the depot, among SHA and TIN, with the shorter Google-predicted travel time at dispatch. Algorithm~\ref{alg:DCA-algorithm} is then run conditional on this primary path to obtain the secondary path $z_2$. The candidate set includes the Google-recommended paths from SHA and TIN with adaptive rerouting and the implemented IDEAL secondary path $z_2$ using the hybrid execution rule in Appendix~\ref{ec:hybrid-simulation}.

Figure~\ref{fig:opt-rate-sha-tin} summarizes candidate-optimal-rate performance in the Sha Tin transfer test. Region-based, Google primary, and Google dual dispatch achieve candidate-optimal rates of 63.7\%, 72.1\%, and 73.9\%, respectively. The always-dual IDEAL endpoint, denoted by IDEAL dual, reaches a 98.2\% candidate-optimal rate. Another IDEAL operating point exceeds the strongest baseline, reaching a 74.8\% candidate-optimal rate with only a 12.6\% increase in average dispatch volume over single-dispatch baselines. These results show that the same IDEAL overlay-and-replay protocol can be applied to a different pair of road networks and depots.

\begin{figure}[ht]
    \centering
    \includegraphics[width=0.5\linewidth]{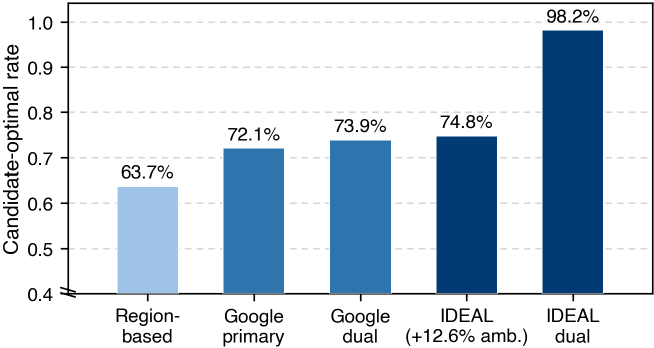}
    \caption{Candidate-optimal rates in the Sha Tin transfer test. The figure marks the always-dual IDEAL operating point and another IDEAL operating point that exceeds the strongest baseline with a lower average dispatch volume.}
    \label{fig:opt-rate-sha-tin}
\end{figure}

Figure~\ref{fig:pareto-sha-tin} shows the regret--dispatch-volume trade-off. At $\bar a=2$, IDEAL lowers mean regret from 8.28~seconds under Google dual dispatch to 0.33~seconds, and lowers $\mathrm{CVaR}_{0.95}$ regret from 126.5~seconds to 6.5~seconds.
Figure~\ref{fig:regret-cdf-sha-tin} and Table~\ref{tab:regret-metrics-sha-tin} report the regret distribution. IDEAL dual has zero regret through the 95th percentile, whereas Google dual dispatch has 95th-percentile regret of 46.8~seconds. The table shows the same pattern across mean, P95, P99, and CVaR metrics.

\begin{figure}[h]
    \centering
    \includegraphics[width=\linewidth]{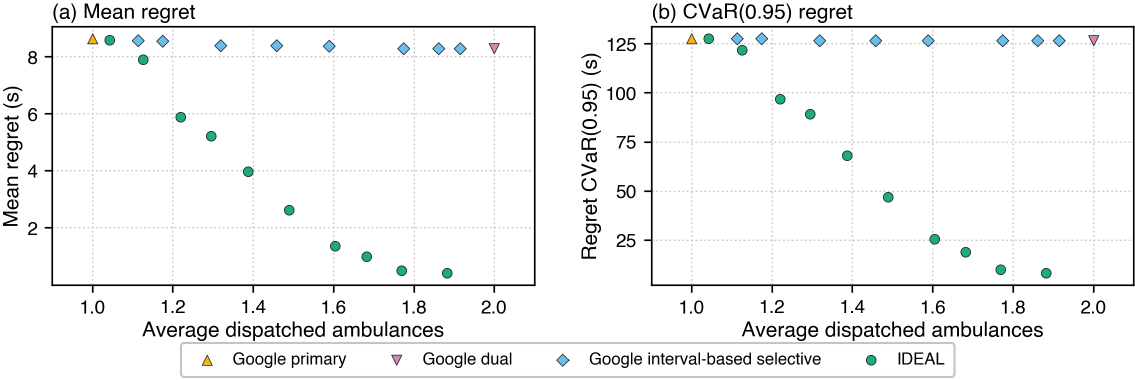}
    \caption{Mean-regret and tail-risk trade-offs for the same threshold sweep in the Sha Tin transfer test. Left: mean regret versus average dispatch volume. Right: $\mathrm{CVaR}_{0.95}$ regret versus average dispatch volume.}
    \label{fig:pareto-sha-tin}
\end{figure}

\begin{figure}[h]
    \centering
    \includegraphics[width=0.45\linewidth]{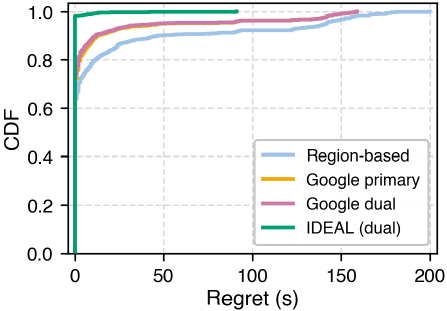}
    \caption{Empirical CDF of regret in the Sha Tin transfer test. For each regret $x$, the curve reports the fraction of incidents with regret at most $x$.}
    \label{fig:regret-cdf-sha-tin}
\end{figure}

\begin{table}[h]
    \centering
    \small
    \setlength{\tabcolsep}{3pt}
    \caption{Summary statistics of regret in seconds in the Sha Tin transfer test. Rows containing a percentage correspond to threshold-based selective policies at selected dispatch-volume levels, where the percentage denotes the increase in average dispatch volume relative to a single-dispatch policy, and \textbf{IDEAL, dual} is the always-dual operating point.}
    \label{tab:regret-metrics-sha-tin}
    \begin{tabular}{lrrrrr}
    \toprule
    Strategy & Mean $\downarrow$ & P95 $\downarrow$ & P99 $\downarrow$ & $\mathrm{CVaR}_{0.95}\downarrow$ & $\mathrm{CVaR}_{0.99}\downarrow$ \\
    \midrule
Region-based & 16.5 & 141.8 & 169.9 & 158.2 & 179.2 \\
Google primary & 8.6 & 50.2 & 149.0 & 127.7 & 154.2 \\
Google interval, +25.0\% amb. & 8.5 & 50.2 & 149.0 & 127.7 & 154.2 \\
Google interval, +49.9\% amb. & 8.4 & 46.8 & 149.0 & 126.5 & 154.2 \\
Google interval, +75.2\% amb. & 8.3 & 46.8 & 149.0 & 126.5 & 154.2 \\
Google dual & 8.3 & 46.8 & 149.0 & 126.5 & 154.2 \\
\textbf{IDEAL, +25.0\% amb.} & 5.7 & 17.6 & 144.9 & 96.5 & 151.2 \\
\textbf{IDEAL, +49.9\% amb.} & 2.6 & 8.0 & 89.6 & 46.8 & 140.5 \\
\textbf{IDEAL, +75.0\% amb.} & 1.0 & 1.0 & 13.0 & 18.5 & 65.1 \\
\textbf{IDEAL, dual} & 0.3 & 0.0 & 8.9 & 6.5 & 26.1 \\
    \bottomrule
    \end{tabular}
\end{table}

Table~\ref{tab:wilcoxon-sha-tin} reports one-sided paired Wilcoxon signed-rank tests against baselines on
$\operatorname{Reg}_i(b)-\operatorname{Reg}_i(\text{IDEAL dual})$, with zero differences discarded and tied absolute differences assigned average ranks. The alternative hypothesis is that baseline regret is larger than the regret of IDEAL dual. At the 5\% significance level, we reject the null against all three baselines. The largest one-sided $p$-value occurs in the comparison with Google dual dispatch and equals $1.3\times10^{-25}$, indicating that the improvement of IDEAL dual remains statistically significant against the strongest baseline.

\begin{table}[h]
\centering
\small
\caption{One-sided paired Wilcoxon signed-rank tests on per-incident regret against non-selective baselines in the Sha Tin transfer test. The alternative hypothesis is that baseline regret exceeds the regret of \textbf{IDEAL dual}.}
\label{tab:wilcoxon-sha-tin}
\begin{tabular}{lrrrrr}
\toprule
Baseline & Evaluation $n$ & Wilcoxon $n_{\ne0}$ & Mean diff.\ (s; normal 95\% CI) & Statistic & $p$-value \\
\midrule
(1) Region-based & 705 & 247 & 16.2 [13.2, 19.2] & 30455.5 & $1.1\times 10^{-41}$ \\
(2) Google primary & 705 & 184 & 8.3 [6.2, 10.4] & 17020 & $2.7\times 10^{-32}$ \\
(3) Google dual & 705 & 197 & 7.9 [5.8, 10.1] & 18077 & $1.3\times 10^{-25}$ \\
\bottomrule
\end{tabular}
\end{table}

\end{document}